\documentclass[11pt,oneside,reqno]{amsart}

\usepackage{amsmath,amssymb,amsthm,textcomp}
\usepackage{amsfonts,graphicx}
\usepackage[mathscr]{eucal}
\usepackage{color}
\usepackage{csquotes}
\usepackage[backend=bibtex,%
firstinits=true,%
doi=false,%
isbn=true,%
url=false,%
maxnames=99]{biblatex}%

\AtEveryBibitem{\clearfield{issn}}
\AtEveryCitekey{\clearfield{issn}}
\numberwithin{equation}{section}
\DeclareNameAlias{sortname}{last-first}
\theoremstyle{definition}
\usepackage{mathtools}
\numberwithin{equation}{section}

\newcommand{\ncom}{\newcommand}

\ncom{\beq}{\begin{equation}}
	\ncom{\eeq}{\end{equation}}
\ncom{\bea}{\begin{eqnarray*}}
	\ncom{\eea}{\end{eqnarray*}}
\ncom{\beqa}{\begin{eqnarray}}
	\ncom{\eeqa}{\end{eqnarray}}
\ncom{\nno}{\nonumber}
\ncom{\non}{\nonumber}
\ncom{\ds}{\displaystyle}
\ncom{\half}{\frac{1}{2}}
\ncom{\mbx}{\makebox{.25cm}}
\ncom{\hs}{\mbox{\hspace{.25cm}}}
\ncom{\rar}{\rightarrow}
\ncom{\Rar}{\Rightarrow}
\ncom{\noin}{\noindent}
\ncom{\bc}{\begin{center}}
	\ncom{\ec}{\end{center}}
\ncom{\sz}{\scriptsize}
\ncom{\rf}{\ref}
\ncom{\s}{\sqrt{2}}
\ncom{\sgm}{\sigma}
\ncom{\Sgm}{\Sigma}
\ncom{\psgm}{\sigma^{\prime}}
\ncom{\dt}{\delta}
\ncom{\Dt}{\Delta}
\ncom{\lmd}{\lambda}
\ncom{\Lmd}{\Lambda}
\ncom{\Th}{\Theta}
\ncom{\e}{\eta}
\ncom{\eps}{\epsilon}
\ncom{\pcc}{\stackrel{P}{>}}
\ncom{\lp}{\stackrel{L_{p}}{>}}
\ncom{\dist}{{\rm\,dist}}
\ncom{\sspan}{{\rm\,span}}
\ncom{\re}{{\rm Re\,}}
\ncom{\im}{{\rm Im\,}}
\ncom{\sgn}{{\rm sgn\,}}
\ncom{\ba}{\begin{array}}
	\ncom{\ea}{\end{array}}
\ncom{\hone}{\mbox{\hspace{1em}}}
\ncom{\htwo}{\mbox{\hspace{2em}}}
\ncom{\hthree}{\mbox{\hspace{3em}}}
\ncom{\hfour}{\mbox{\hspace{4em}}}
\ncom{\vone}{\vskip 2ex}
\ncom{\vtwo}{\vskip 4ex}
\ncom{\vonee}{\vskip 1.5ex}
\ncom{\vthree}{\vskip 6ex}
\ncom{\vfour}{\vspace*{8ex}}
\ncom{\norm}{\|\;\;\|}
\ncom{\integ}[4]{\int_{#1}^{#2}\,{#3}\,d{#4}}
\ncom{\vspan}[1]{{{\rm\,span}\{ #1 \}}}
\ncom{\dm}[1]{ {\displaystyle{#1} } }
\ncom{\ri}[1]{{#1} \index{#1}}

\newtheorem{theorem}{\bf Theorem}[section]
\newtheorem{remark}{\bf Remark}[section]
\newtheorem{proposition}{Proposition}[section]

\newtheorem{corollary}{Corollary}[section]

\newtheoremstyle
{remarkstyle}
{}
{11pt}
{}
{}
{\bfseries}
{:}
{     }
{\thmname{#1} \thmnumber{#2} }

\theoremstyle{remarkstyle}

\def\l{{\langle}}
\def\r{\rangle}

\def\eps{\varepsilon}

\begin{document}
	\title{On the multivariate generalized  counting process and its time-changed variants}
	\author[Kuldeep Kumar Kataria]{Kuldeep Kumar Kataria}
	\address{Kuldeep Kumar Kataria, Department of Mathematics, Indian Institute of Technology Bhilai, Durg 491002, India.}
	\email{kuldeepk@iitbhilai.ac.in}
	\author[Manisha Dhillon]{Manisha Dhillon}
	\address{Manisha Dhillon, Department of Mathematics, Indian Institute of Technology Bhilai, Durg 491002, India.}
	\email{manishadh@iitbhilai.ac.in}
	\subjclass[2010]{Primary: 60G22; 60G52; Secondary: 26A33; 33E12}
	\keywords{Generalized counting process; multivariate counting process; random time-change; Bern\v{s}tein function; L\'evy measure.}
	\date{\today}
	\begin{abstract}
	In this paper, we study a multivariate version of the generalized counting process (GCP) and discuss its various time-changed variants. The time is changed using random processes such as the stable subordinator, inverse stable subordinator, and their composition, tempered stable subordinator, gamma subordinator \textit{etc.}
Several distributional properties that include the probability generating function, probability mass function and their governing differential equations are obtained for these variants. It is shown that some of these time-changed processes are L\'evy and  for such processes we have derived the associated L\'evy measure. The explicit expressions for the covariance and codifference of the component processes for some of these time-changed variants are obtained. 
An application of the multivariate generalized space fractional counting process to shock models is discussed.
\end{abstract}
	\maketitle
	\section{Introduction}
	The Poisson process is an important counting process which has many applications in various fields. It's a pure birth process with a possibility of atmost one arrival in an infinitesimal interval. Recently, a generalization of this process, namely, the generalized counting process (GCP) is introduced and studied by Di Crescenzo \textit{et al.} (2016). In GCP, which we denote by $\{M(t)\}_{t\geq0}$, there is a possibility of $k$ kinds of jumps of size $1,2,\dots,k$ with positive rates $\lambda_1,\lambda_2,\dots,\lambda_k$, respectively. Its state probabilities $p(n,t)=\mathrm{Pr}\{M(t)=n\}$ are given by (see Di Crescenzo \textit{et al.} (2016))
	\begin{equation*}
		p(n,t)=\sum_{\Omega(k,n)}\prod_{j=1}^{k}\frac{(\lambda_jt)^{x_j}}{x_j!}e^{-\lambda_j t}, \, \, n\geq0,
	\end{equation*}
	where $\Omega(k,n)=\{(x_1,x_2,\ldots,x_k):\sum_{j=1}^{k}jx_j=n,\ x_j\in \mathbb{N}_0\}$. Its L\'evy measure and probability generating function (pgf) are given by (see Kataria and Khandakar (2022a)) $
	\Pi(\mathrm{d}s) = \sum_{j=1}^{k}\lambda_j\delta_j\mathrm{d}s
	$ and 
	\begin{equation}\label{pgfgcp}
		G(u,t)=\mathbb{E}\left(u^{M(t)}\right)=\exp\bigg(-t\sum_{j=1}^{k}\lambda_{j}(1-u^{j})\bigg),\  \ |u|\le 1,
	\end{equation}
respectively. Here, $\delta_j$'s are Dirac measure.
	
	Let $\{X_i\}_{i\ge1}$ be a sequence of independent and identically distributed (iid) random variables with the following probability mass function (pmf):
	\begin{equation*}
		\mathrm{Pr}\{X_1=j\}=\frac{\lambda_j}{\Lambda},\, j=1,2,\dots,k,
	\end{equation*}
	where $\Lambda=\lambda_1+\lambda_2+\dots+\lambda_k$.
	Di Crescenzo \textit{et al.} (2016) showed that the GCP is equal in distribution to a
	compound Poisson process, that is,
	\begin{equation}\label{ccc}
		M(t)\overset{d}{=}\sum_{i=1}^{N(t)}X_i,
	\end{equation}
	where the sequence $\{X_i\}_{i\ge1}$ is independent of the Poisson process $\{N(t)\}_{t\ge0}$ with parameter $\Lambda$. Here, $\overset{d}{=}$ denotes equal in distribution.
	
	 It is well known that the Poisson process has light-tailed distributed waiting times. This limits its applications to the phenomenon involving long memory. To overcome this limitation,
	 several authors have introduced and studied various time-changed generalizations of the Poisson process. These generalized processes are characterized by heavy-tailed distributions of their waiting times. Generally, either the time is changed by a stable subordinator in which case the time-changed Poisson process is termed as the space fractional Poisson
	 process or the first passage time of stable subordinator is used as a time-changed process in which case it is termed as the time fractional Poisson process, for more details, we refer the reader to Mainardi \textit{et al.} (2004), Beghin and Orsingher (2009), Orsingher and Polito (2012) \textit{etc.}
	
Di Crescenzo {\it et al.} (2016) studied a time-changed variant of the GCP known as the generalized fractional counting process (GFCP). Here, we denote it by $\{M^{\beta}(t)\}_{t\ge0}$, $0<\beta\le 1$. Its pmf $p^{\beta}(n,t)=\mathrm{Pr}\{M^{\beta}(t)=n\}$ solves the following system of differential equations:
	\begin{equation*}
		\frac{\mathrm{d}^{\beta}}{\mathrm{d}t^{\beta}}p^{\beta}(n,t)=-\sum_{j=1}^{k}\lambda_jp^{\beta}(n,t)+	\sum_{j=1}^{\min\{n,k\}}\lambda_{j}p^{\beta}(n-j,t),\ \ n\ge0,
	\end{equation*}
	with the initial condition $p^{\beta}(n,0)=\mathbb{I}_{\{n=0\}}$,
	where $\mathbb{I}_{\{\cdot\}}$ is the indicator function and $\dfrac{\mathrm{d}^{\beta}}{\mathrm{d}t^{\beta}}$ is the Caputo fractional derivative defined as (see Kilbas {\it et al.} (2006))
	\begin{equation}\label{caputo}
		\frac{\mathrm{d}^{\beta}}{\mathrm{d}t^{\beta}}f(t)=\left\{
		\begin{array}{ll}
			\dfrac{1}{\Gamma{(1-\beta)}}\displaystyle\int^t_{0} (t-s)^{-\beta}f'(s)\,\mathrm{d}s,\ \ 0<\beta<1,\\\\
			f'(t),\ \ \beta=1.
		\end{array}
		\right.
	\end{equation}
 Its  mean and variance are given by (see 	Di Crescenzo {\it et al.} (2016))
\begin{equation}\label{meanvargfcp}
	\mathbb{E}(M^\beta (t))=\frac{\sum_{j=1}^{k}j\lambda_jt^\beta}{\Gamma(\beta+1)}
	\end{equation}
	and
\begin{equation}\label{meanvargfcpvv}	 \operatorname{Var}(M^\beta(t))=\bigg(\sum_{j=1}^{k}j\lambda_jt^\beta\bigg)^2\left(\frac{2}{\Gamma(2\beta+1)}-\frac{1}{\Gamma^2(\beta+1)}\right)+\frac{\sum_{j=1}^{k}j^2\lambda_jt^\beta}{\Gamma(\beta+1)}.
\end{equation}

 For $\beta=1$, the GFCP reduces to the GCP. Further, for $k=1$, the GFCP and the GCP reduces to the time fractional Poisson process (see Beghin and Orsingher (2009), Meerschaert {\it et al.} (2011)) and the homogeneous Poisson process, respectively. Kataria and Khandakar ((2022a), (2022b)) have shown that various counting processes, namely, the Poisson process of order $k$ (see Kostadinova and Minkova (2013)), P\'olya-Aeppli process, P\'olya-Aeppli process of order $k$ (see  Chukova and Minkova (2015)), Poisson-logarthmic process (see Sendova and Minkova (2018)), Bell-Touchard process (see Freud and Rodriguez (2022)), {\it etc}. and their time-changed variants are particular cases of the GFCP.
 For martingale characterizations of GCP and its time-changed variants, we refer the readers to Dhillon and Kataria (2025). 
 
  Cha and Giorgio (2018) introduced and studied a new class of bivariate counting processes having marginal regularity property. Moreover, they discussed an application of this process to a shock model. Recently, an application of a bivariate tempered space-fractional Poisson process to shock models is given by Soni \textit{et al.} (2024).
  Multivariate data are frequently used in ecology, biology, climate study, reliability theory and related areas. In the last decade, the multivariate Poisson process has been an important topic of study for many researchers, for example, Beghin and Macci ((2016), (2017)), Di Crescenzo and Meoli (2022) \textit{etc.} Beghin and Macci (2016) studied a time-changed multivariate Poisson process, namely, the multivariate space-time fractional Poisson process. 
  For more details on time-changed multivariate random processes, we refer the reader to Barndorff \textit{et al.} (2001),  Beghin \textit{et al.} (2020). 
  
Here, we introduce a multivariate GCP of dimension $q\ge1$ and study various time-changed variants of it. For $q=1$, we get the GCP and its corresponding time-changed variants. The paper is organized as follows:

In Section 2, we set some notations that are used throughout the paper. Here, we give few definitions and known results about some special functions and subordinators. In Section 3, we study the multivariate generalized counting process (MGCP). Its state probabilities, pgf and L\'evy measure are obtained. Subsequently, we study some time-changed variants of it where the time is changed via various random processes, for example, stable subordinator, inverse stable subordinator, composition of stable and inverse stable subordinator, tempered stable subordinator, gamma subordinator \textit{etc.} For all these time-changed processes, the corresponding pgf, pmf and their governing system of differential equations are obtained. It is shown that some of these time-changed processes are L\'evy, and for such processes the associated L\'evy measures are derived. For some cases, we compute the covariance and codifference of the component processes. In Section 4, we discuss a more general time-changed variant of the MGCP where the random time-change component is a subordinator associated to some Bern\v{s}tein function. We derive an explicit expression for its transition probabilities and obtain its pgf, pmf along with their governing system of differential equations. It is shown that for particular choices of Bern\v{s}tein function, this time-changed variant of MGCP reduces to the time-changed processes that are discussed in Section 3. Later, we discuss an application of the multivariate generalized space fractional counting process to a shock model.  Several results such as the reliability function, hazard rates, failure density, and the probability that failure occurs due to a particular type of shock are obtained.

\section{Preliminaries}
First, we set some notations which will be used throughout the paper. Then, we give some definitions and known results for some special functions, stable subordinator, inverse stable subordinator, tempered stable subordinator, gamma subordinator and inverse Gaussian subordinator.

Throughout the paper, we will be using the following notations:
Let $\mathbb{R}$ denote the set of real numbers, $\mathbb{N}_{0}=\{0,1,2,\dots\}$ denote the set of non-negative integers and $\bar{n}=(n_1,n_2,\dots,n_q)$, $\bar{0}=(0,0,\dots,0)$, $\bar{1}=(1,1,\dots,1)$ be $q$-tuple vectors. By $\bar{n}\ge\bar{m}$, we mean $n_i\ge m_i$ for each $i=1,2,\dots,q$. Moreover, $\bar{n}\succ \bar{m}$ ( $\bar{n}\prec\bar{m}$) denotes that $n_i\ge m_i$ ($n_i\leq m_i$) for all $i=1,2,\dots,q$ and $\bar{n}\neq \bar{m}$. 

\subsection{Some special functions}
Here, we give definitions of some special functions.
\subsubsection{Mittag-Leffler function}
The three-parameter Mittag-Leffler function is defined as
(see Kilbas \textit{et al.} (2006), p. 45)
\begin{equation}\label{Mitagdef}
	E_{\alpha,\,\beta}^{\gamma}(x)\coloneqq\sum_{r=0}^{\infty}\frac{(\gamma)_r\,x^r}{r!\Gamma(r\alpha+\beta)},\,\, x\in\mathbb{R},
\end{equation}
where $\alpha>0$, $\beta>0$, $\gamma>0$, $(\gamma)_0=1$ and  $(\gamma)_r=(\gamma)(\gamma+1)\dots(\gamma+r-1)$ is the Pochhammer symbol.

For $\gamma=1$, it reduces to the two parameter Mittag-Leffler function. For $\beta=\gamma=1$, it reduces to Mittag-Leffler function. Further, for $\alpha=\beta=\gamma=1$, we get the exponential function that is $E_{1,1}^{1}(x)=e^x$.
\subsubsection{Generalized Wright function}
The generalized Wright function is defined as (see Kilbas \textit{ et al.} (2006), Eq. (1.11.14))
\begin{equation}\label{wrightfn}
	{}_p\Psi_q\left[\begin{matrix}
		(a_i,\alpha_i)_{1,p}\\
		(b_j,\beta_j)_{1,q}
	\end{matrix}\bigg| x\right]=\sum_{r=0}^{\infty}\frac{\prod_{i=1}^{p}\Gamma(a_i+\alpha_ir)}{\prod_{j=1}^{q}\Gamma(b_j+\beta_jr)}\frac{x^r}{r!},\, \, x\in \mathbb{R},
\end{equation}
where $a_i$'s, $b_j$'s, $\alpha_i$'s and $\beta_j$'s are real numbers.
\subsubsection{Incomplete beta function}
	The incomplete beta function is defined as (see Gradshteyn and Ryzhik (2007), p. 910)
	\begin{equation}\label{defbeta}
		B(x,p,q)=\int_{0}^{x}t^{p-1}(1-t)^{q-1}\, \mathrm{d}t,\, 0< x<1,
	\end{equation}
	where $p>0,\, q>0.$
	\subsubsection{Generalized incomplete gamma function}
	The generalized incomplete gamma function is defined as 
	\begin{equation}\label{defgama}
		\gamma(x;p,q)=\int_{p}^{q}e^{-t}t^{x-1}\mathrm{d}t,\,x>0, 
	\end{equation}
	where $0\le p < q \le 1.$
\subsection{Subordinator}\label{seci}
A subordinator $\{D(t)\}_{t\ge0}$ is a one dimensional non-decreasing L\'evy process which is characterized by the following Laplace transform (see Applebaum (2009)):
\begin{equation*}
\mathbb{E}(e^{-sD(t)})=e^{-t\phi(s)},\, s>0,
\end{equation*}
where $\phi(s)=bs+f(s)$ is the Laplace exponent. Here, $b\ge0$ is the drift and 
\begin{equation*}
	f(s)=\int_{0}^{\infty}(1-e^{-sx})\,\mu_{D}(\mathrm{d}x)
\end{equation*} 
is the Bern\v{s}tein function with L\'evy measure $\mu_D$. So, it satisfies $\mu_{D}([0,\infty))=\infty$ and $\int_{0}^{\infty}(1\wedge x)\,\mu_{D}(\mathrm{d}x)<\infty$, where $x\wedge y=\min\{x,y\}$.

\subsubsection{Stable subordinator and its inverse} An $\alpha$-stable subordinator $\{D_\alpha(t)\}_{t\ge0}$, $0<\alpha<1$ is a driftless subordinator whose Laplace exponent is given by $\phi(s)=s^\alpha$, $s>0$. So, $\mathbb{E}(e^{-sD_\alpha(t)})=e^{-ts^\alpha}$. Its L\'evy measure is given by
  \begin{equation}\label{lmeasalpha}
  \mu_{D_\alpha}(\mathrm{d}s)=\frac{\alpha}{\Gamma(1-\alpha)}s^{-\alpha-1}\mathrm{d}s.
  \end{equation} 
 Its density function $h_{D_\alpha(t)}(x,t)\mathrm{d}x=\mathrm{Pr}\{D_\alpha(t)\in\mathrm{d}x\}$ is given by
 \begin{equation*}
 h_{D_{\alpha}(t)}(x,t)\mathrm{d}x=\frac{1}{\pi}\sum_{n=1}^{\infty}(-1)^{n+1}\frac{\Gamma(\alpha n+1)}{n!}\frac{t^n}{y^{\alpha n+1}}\sin(\pi\alpha n)\mathrm{d}x,\ \ x>0.
 \end{equation*} 
Moreover, its joint density function is defined as
\begin{equation*}
h_{D_{\alpha}(t_1),D_{\alpha}(t_2)}(x_1,t_1;x_2,t_2)\mathrm{d}x_1\,\mathrm{d}x_2=\mathrm{Pr}\{D_{\alpha}(t_1)\in\mathrm{d}x_1, D_{\alpha}(t_2)\in\mathrm{d}x_2\}
\end{equation*}
and it satisfies (see Di Crescenzo and Meoli (2022), Eq. (3.2))
\begin{equation}\label{eqjoalpha}
h_{D_{\alpha}(t_1),D_{\alpha}(t_2)}(x_1,t_1;x_2,t_2)\mathrm{d}x_1\,\mathrm{d}x_2=h_{D_\alpha(t_2-t_1)}(x_2-x_2,t_2-t_1)h_{D_{\alpha}(t_1)}(x_1,t_1)\mathrm{d}x_1\, \mathrm{d}x_2.
\end{equation}
The first passage time process of an $\alpha$-stable subordinator $\{D_\alpha(t)\}_{t\ge0}$ is known as the inverse $\alpha$-stable subordinator $\{Y_\alpha(t)\}_{t\ge0}$. It is defined as follows:
\begin{equation*}
Y_\alpha(t)=\inf\{s>0:D_\alpha(s)>t\}
\end{equation*}
and its Laplace transform is given by $\mathbb{E}(e^{-sY_\alpha(t)})=E_{\alpha,1}(-st^\alpha)$. 

From Eq. (2.4) and Eq. (2.7) of Piryatinska  \textit{et al.} (2005), the $n$th moment of $\{Y_\alpha(t)\}_{t\ge0}$ is given by
\begin{equation}\label{beghin}
	\mathbb{E}(Y_\alpha^n(t))=\frac{n!t^{n\alpha}}{\Gamma(n\alpha+1)},\, n\ge0.
\end{equation} 
For an alternate proof of \eqref{beghin} using Mellin transform, we refer the reader to Remark 4.3 of Kataria and Vellaisamy (2018). 
\subsubsection{Tempered stable subordinator}\label{sectempre}
Let $0<\alpha<1$, $\theta>0$. The Laplace exponent of tempered stable subordinator $\{D_{\alpha,\theta}(t)\}_{t\ge0}$  is  $\Psi(s)=(s+\theta)^\alpha-\theta^\alpha$, $s>0$ (see Meerschaert \textit{et al.} (2013)). Thus,
\begin{equation}\label{temppp}
\mathbb{E}(e^{-sD_{\alpha,\theta}(t)})=e^{-t\Psi(s)}.
\end{equation}

\subsubsection{Gamma subordinator}
The density function of gamma subordinator $\{G_{a,b}(t)\}_{t\ge0}$ with parameters $a>0$, $b>0$ is given by (see Applebaum (2009), p. 55)
\begin{equation}\label{gammadensity}
h(x,t)=\frac{a^{bt}}{\Gamma(bt)}x^{bt-1}e^{-ax},\,\, x>0
\end{equation}
and its Laplace transform is given by
\begin{equation}\label{gammaberfn}
\mathbb{E}(e^{-sG_{a,b}(t)})=\exp(-bt\log(1+s/a)),\,\, s\ge0.
\end{equation}
Let $e^{c\partial_t}$ be a shift operator such that 
\begin{equation}\label{gammaoper}
	e^{c\partial_t}f(t)\coloneqq\sum_{n=0}^{\infty}\frac{(c\partial_t)^n}{n!}f(t)=f(t+c),\,\, c\in\mathbb{R},
\end{equation}
where $\partial_t=\frac{\partial}{\partial t}$ and $f:\mathbb{R}\to\mathbb{R}$ is an analytic function. Then, for $x\ge0,\, t\ge0$, the density function of gamma subordinator solves (see Beghin (2014), Lemma 2.1)
\begin{equation}\label{erkq}
\frac{\partial}{\partial x}h(x,t)=-b(1-e^{-\partial_t/a})h(x,t),
\end{equation}
such that $h(x,0)=\delta(x)$
and $\lim_{|x|\to+\infty}h(x,t)=0$.
Here, $\delta(x)$ is the Dirac delta function.

\subsubsection{Inverse Gaussian subordinator}
The density function $g(x,t)$ of an inverse Gaussian subordinator $\{I_{\delta,\gamma}(t)\}_{t\ge0}$, $\delta>0$, $\gamma>0$ is given by (see Applebaum (2009), Eq. (1.27))
\begin{equation}\label{gxtinv}
g(x,t)=\frac{\delta t}{\sqrt{2\pi x^3}}\exp\Big(\delta \gamma t-\frac{1}{2}(\delta^2t^2x^{-1}+\gamma^2x)\Big), \,\, t\ge0,\,x\ge0
\end{equation}
 and its Laplace transform is given by
\begin{equation}\label{lapinv}
\mathbb{E}(e^{-sI_{\delta,\gamma}(t)})=\exp(-\delta t(\sqrt{2s+\gamma^2}-\gamma)),\,\, t>0,\, s>0.
\end{equation}

\section{Multivariate GCP and its time-changed variants}\label{secc1}
In this section, we introduce a multivariate GCP and study its time-changed variants.

 For $i=1,2,\dots,q$, let $\{M_i(t)\}_{t\ge0}$ be a GCP which performs $k_i$ kinds of jumps of size $j_i$ with positive rate $\lambda_{ij_i}$, $j_i=1, 2, \dots, k_i$. We define a $q$-variate GCP $\{\bar{M}(t)\}_{t\geq0}$ as follows:
	\begin{equation*}\label{mgcp}
	\bar{M}(t)\coloneqq(M_1(t),M_2(t),\dots ,M_q(t))
	\end{equation*}
  whose transition probabilities are given by 
  \begin{equation}\label{transpb}
 	\mathrm{Pr}\{\bar{M}(t+h)=\bar{n}+\bar{j}|\bar{M}(t)=\bar{n}\}=\begin{cases}
 		\lambda_{ij_i}h+o(h),\ \bar{j}=\bar{\epsilon}_i^{j_i},\,1\leq j_i\leq k_i,\,1\leq i\leq q,\vspace{0.1cm}\\
 		1-\sum_{i=1}^{q}\sum_{j_i=1}^{k_i}\lambda_{ij_i} h+o(h),\ \bar{j}=\bar{0},\vspace{0.1cm}\\
 		o(h),\ \text{otherwise},
 	\end{cases}
 \end{equation}
where $o(h)/h\to0$ as $h\to0$. Here, $\bar{\epsilon}_i^{j_i}\in \mathbb{N}_{0}^q$ is a $q$-tuple vector whose $i^{th}$ entry is $j_i$ and other entries are zero. 
We call this $q$-variate GCP $\{\bar{M}(t)\}_{t\geq0}$ as the multivariate generalized counting process (MGCP). 

In an infinitesimal time interval of length $h$, there is a transition in MGCP iff there is a transition in exactly one of its component GCPs. That is, in such a time interval, the probability of simultaneous transitions in more than one component GCPs is negligible.
Note that the MGCP has independent and stationary increments because its component GCPs have independent and stationary increments.
  
Next, we obtain the system of differential equations that governs the state probabilities of MGCP.
\begin{proposition}
The state probabilities $p(\bar{n},t)=\mathrm{Pr}\{\bar{M}(t)=\bar{n}\}$, $\bar{n}\geq \bar{0}$ solve the following system of differential equations:
\begin{equation}\label{propmgcp}
\frac{\mathrm{d}}{\mathrm{d}t}p(\bar{n},t)=-\sum_{i=1}^{q}\sum_{j_i=1}^{k_i}\lambda_{ij_i}\big(p(\bar{n},t)-p(\bar{n}-\bar{\epsilon}_i^{j_i},t)\big)
\end{equation}
with initial conditions
\begin{equation*}
	p(\bar{n},0)=\begin{cases}
		1,\ \bar{n}=\bar{0},\\
		0,\ \bar{n}\ne\bar{0},
	\end{cases}
\end{equation*}
where $p(\bar{n}-\bar{m},t)=0$ for all $\bar{n}\prec\bar{m}$. 
\end{proposition}
\begin{proof}
Note that
\begin{align*}
p(\bar{n},t+h)&=p(\bar{n},t)\mathrm{Pr}\{\bar{M}(t+h)=\bar{n}|\bar{M}(t)=\bar{n}\}\\
&\, \, +\sum_{i=1}^{q}\sum_{j_i=1}^{k_i}p(\bar{n}-\bar{\epsilon}_i^{j_i},t)\mathrm{Pr}\{\bar{M}(t+h)=\bar{n}|\bar{M}(t)=\bar{n}-\bar{\epsilon}_i^{j_i}\}+o(h).
\end{align*}
Using \eqref{transpb}, we get
\begin{equation*}
p(\bar{n},t+h)=p(\bar{n},t)\bigg(1-\sum_{i=1}^{q}\sum_{j_i=1}^{k_i}\lambda_{ij_i} h\bigg)+\sum_{i=1}^{q}\sum_{j_i=1}^{k_i}p(\bar{n}-\bar{\epsilon}_i^{j_i},t)\lambda_{ij_i}h+o(h).
\end{equation*}
Equivalently,
\begin{equation*}
\frac{p(\bar{n},t+h)-p(\bar{n},t)}{h}=-\sum_{i=1}^{q}\sum_{j_i=1}^{k_i}\lambda_{ij_i} p(\bar{n},t)+\sum_{i=1}^{q}\sum_{j_i=1}^{k_i}\lambda_{ij_i}p(\bar{n}-\bar{\epsilon}_i^{j_i},t)+\frac{o(h)}{h}.
\end{equation*}
On letting $h\to 0$, we obtain the required result.
\end{proof}

Let $\bar{u}=(u_1,u_2,\dots,u_q)$ such that $|u_i|\leq 1$ for all $i$. In the following result, we obtain the differential equation that governs the pgf 
	\begin{equation}\label{pgfdef}
	G(\bar{u},t)=\mathbb{E}\bigg(u_1^{M_1(t)}u_2^{M_2(t)}\dots u_q^{M_q(t)}\bigg)
	=\sum_{\substack{n_i\geq 0\\  i=1,2,\dots,q}}u_1^{n_1} u_2^{n_2}\dots u_q^{n_q}p(\bar{n},t)
\end{equation}
 of MGCP.
\begin{proposition}
		The pgf $G(\bar{u},t)$ of MGCP solves the following differential equation:
\begin{equation}\label{pgfde}
\frac{\partial}{\partial t}G(\bar{u},t)=-\sum_{i=1}^{q}\sum_{j_i=1}^{k_i}\lambda_{ij_i}(1-u_i^{j_i})G(\bar{u},t),\, \,G(\bar{u},0)=1.
\end{equation}
\end{proposition}
\begin{proof}
On multiplying  $u_1^{n_1}u_2^{n_2}\dots u_q^{n_q}$ on both sides of \eqref{propmgcp}, we get
	{\small\begin{align*}
\bigg(\prod_{l=1}^{q}u_l^{n_l}\bigg)\frac{\mathrm{d}}{\mathrm{d}t}p(\bar{n},t)=-\bigg(\prod_{l=1}^{q}u_l^{n_l}\bigg)\sum_{i=1}^{q}\sum_{j_i=1}^{k_i}\lambda_{ij_i}p(\bar{n},t)	+\sum_{i=1}^{q}\bigg(\prod_{l=1,\, l\ne i}^{q}u_l^{n_l}\bigg)\sum_{j_i=1}^{k_i}\lambda_{ij_i}p(\bar{n}-\bar{\epsilon}_i^{j_i},t)u_i^{n_i-j_i}u_i^{j_i}
	\end{align*}}
whose sum over the range $\bar{n}\geq \bar{0}$ on both sides gives the required result.
\end{proof}
On solving \eqref{pgfde}, the pgf of MGCP is given by
\begin{equation}\label{pgffmgcp}
G(\bar{u},t)=\exp\bigg(-t\sum_{i=1}^{q}\sum_{j_i=1}^{k_i}\lambda_{ij_i}(1-u_i^{j_i})\bigg), \, |u_i|\le1.
\end{equation}
For $q=1$, the pgf \eqref{pgffmgcp} reduces to 
\begin{equation*}
	G(u,t)=\exp\bigg(-t\sum_{j=1}^{k}\lambda_j(1-u^j)\bigg),\, |u|\le1
\end{equation*}
which agrees with \eqref{pgfgcp}, the pgf of GCP $\{M(t)\}_{t\ge0}$.

Equivalently, we have the  following representation:
\begin{align}
G(\bar{u},t)&=e^{-\lambda t}\sum_{r=0}^{\infty}\frac{1}{r!}\bigg(t\sum_{i=1}^{q}\sum_{j_i=1}^{k_i}\lambda_{ij_i}u_i^{j_i}\bigg)^r\nonumber\\
&=e^{-\lambda t}\sum_{r=0}^{\infty}\sum_{\substack{\sum_{i=1}^{q}x_i=r\\x_i\in\mathbb{N}_0}}\prod_{i=1}^{q}\frac{1}{x_i!}\bigg(t\sum_{j_i=1}^{k_i}\lambda_{ij_i}u_i^{j_i}\bigg)^{x_i}\nonumber\\
&=e^{-\lambda t}\sum_{r=0}^{\infty}\sum_{\substack{\sum_{i=1}^{q}x_i=r\\x_i\in\mathbb{N}_0}}\prod_{i=1}^{q}\sum_{\substack{\sum_{j_i=1}^{k_i}x_{ij_i}=x_i\\x_{ij_i}\in\mathbb{N}_0}}\prod_{j_i=1}^{k_i}\frac{(\lambda_{ij_i} t)^{x_{ij_i}}}{x_{ij_i}!}u_i^{j_ix_{ij_i}}\nonumber\\
&=e^{-\lambda t}\sum_{\substack{x_i\geq 0\\ i=1,2,\dots, q}}\prod_{i=1}^{q}\sum_{\substack{\sum_{j_i=1}^{k_i}x_{ij_i}=x_i\\x_{ij_i}\in\mathbb{N}_0}}\prod_{j_i=1}^{k_i}\frac{(\lambda_{ij_i} t)^{x_{ij_i}}}{x_{ij_i}!}u_i^{j_ix_{ij_i}}\nonumber\\
&=e^{-\lambda t}\sum_{\substack{n_i\ge0\\i=1,2,\dots,q}}\Big(\prod_{i=1}^{q}u_i^{n_i}\Big)\prod_{i=1}^{q}\sum_{\Omega(k_i,n_i)}\prod_{j_i=1}^{k_i}\frac{(\lambda_{ij_i}t)^{x_{ij_i}}}{x_{ij_i}!},\label{pgfmcp}
\end{align}
where 
$\lambda=\sum_{i=1}^{q}\sum_{j_i=1}^{k_i}\lambda_{ij_i}$ and $\Omega(k_i,n_i)=\big\{(x_{i1}, x_{i2}, \dots, x_{ik_i}):\sum_{j_i=1}^{k_i}j_ix_{ij_i}=n_i,\, x_{ij_i}\in \mathbb{N}_{0}\big\}$.

Finally, on comparing the coefficient of $u_1^{n_1}u_2^{n_2}\dots u_q^{n_q}$ over the range $\bar{n}\ge\bar{0}$ on both sides of \eqref{pgfdef} and \eqref{pgfmcp}, we get
\begin{equation}\label{jopmf}
p(\bar{n},t)=\prod_{i=1}^{q}\sum_{\Omega(k_i,n_i)}\prod_{j_i=1}^{k_i}\frac{(\lambda_{ij_i}t)^{n_{ij_i}}}{n_{ij_i}!}e^{-\lambda_{ij_i} t},\,\,\bar{n}\ge\bar{0}.
\end{equation}

For $q=1$, the pmf \eqref{jopmf} reduces to that of $\{M(t)\}_{t\ge0}$.
\begin{remark}\label{remk}
From \eqref{jopmf}, it follows that the component GCPs $\{M_i(t)\}_{t\ge0}$ of MGCP are independent. Hence, the state probabilities $	p(\bar{n},t)$ satisfy the following recurrence relation:
\begin{equation*}
	p(\bar{n},t)=t^q\prod_{i=1}^{q}\frac{1}{n_i}\sum_{j_i=1}^{\min\{n_i,k_i\}}j_i\lambda_{ij_i}p(\bar{n}-\bar{j},t),\,\, \bar{n}\ge\bar{1},
\end{equation*}
which can be established using a recurrence relationship for the state probabilities of GCP (see Kataria and Khandakar (2022a), Proposition 1). Here, $p(\bar{n}-\bar{j},t)=p(n_1-j_1,n_2-j_2,\dots,n_q-j_q,t)$.
\end{remark}

\begin{remark}
From \eqref{ccc}, the GCP is equal in distribution to a compound Poisson process. So, it's a L\'evy process. Thus, on substituting $\tilde{q}^i(A_i)=\sum_{j_i=1}^{k_i}\mathbb{I}_{\{j_i\in A_i\}}\lambda_{ij_i}/\lambda_i$ in Eq. (21) of Beghin and Macci (2016), it follows that the MGCP is also a L\'evy process with the following L\'evy measure:
\begin{equation}\label{measgcp}
\Pi(A_1\times A_2\times\dots \times A_q)=\sum_{	i=1}^{q}\sum_{j_i=1}^{k_i}\lambda_{ij_i}\mathbb{I}_{\{j_i\in A_i\}}.
\end{equation}
On substituting $k_1=k_2=\dots=k_q=1$ in \eqref{measgcp}, $\Pi$ reduces to the L\'evy measure of multivariate Poisson process (see Beghin and Macci (2016), Eq. (22)).
 \end{remark}
\begin{remark}
 Dhillon and Kataria (2023) showed that the process $\{\mathscr{M}^{q}(t)\}_{t\ge0}$ is a GCP, where  $\mathscr{M}^{q}(t)=\sum_{i=1}^{q}M_i(t)$ is the convolution of component GCPs of MGCP. 
\end{remark}
Next, we consider some time-changed variants of the MGCP.
\subsection{Multivariate  generalized space fractional counting process}
 For $i=1,2,\dots,q$, let the GCPs $\{M_i(t)\}_{t\ge0}$ be independent of an $\alpha$-stable subordinator $\{D_\alpha(t)\}_{t\ge0}$, $0<\alpha<1$. We consider the following  time-changed process:
\begin{equation}\label{mulalphadef}
\bar{M}^\alpha(t)\coloneqq \bar{M}(D_\alpha(t)), \, t\ge0,
\end{equation} 
where $\{\bar{M}(t)\}_{t\ge0}$ is the MGCP.
 We call it the multivariate
 generalized space fractional counting process (MGSFCP).
 
  In an infinitesimal time interval of length $h$ such that $o(h)/h\to 0$ as $h\to 0$, the transition probabilities of $\{\bar{M}^\alpha(t)\}_{t\ge0}$ are given by
 {\scriptsize\begin{align}\label{transprbstable}
 \mathrm{Pr}\{\bar{M}^\alpha(t+h)=\bar{n}+\bar{m}|\bar{M}^\alpha(t)=\bar{n}\}&=\begin{cases}
 -h\displaystyle\sum_{\substack{\Omega(k_i,m_i)\\i=1,2,\dots,q}}\frac{\lambda^\alpha\Gamma(1+\alpha)}{\Gamma(\alpha-\sum_{i=1}^{q}\sum_{j_i=1}^{k_i}x_{ij_i}+1)}\prod_{i=1}^{q}\prod_{j_i=1}^{k_i}\frac{(-\lambda_{ij_i}/\lambda)^{x_{ij_i}}}{x_{ij_i}!}+o(h),\ \bar{m}\succ\bar{0},\vspace{.2cm}\\
 1-\lambda^\alpha h+o(h),\ \bar{m}=\bar{0},
 \end{cases}
\end{align}}
where $\Omega(k_i,m_i)=\{(x_{i1},x_{i2},\dots,x_{ik_i}):\sum_{j_i=1}^{k_i}j_ix_{ij_i}=m_i,\,\, x_{ij_i}\in\mathbb{N}_0\}$.
\begin{remark}
 For $q=1$, the transition probabilities in \eqref{transprbstable} reduces to that of generalized space fractional counting process (see Kataria \textit{et al.} (2022), Section 4).
\end{remark}

For $q=1$, the MGSFCP reduces to the generalized space fractional counting process (see Kataria \textit{et al.} (2022)). Further, for $q=1$ and $k_1=1$, the MGSFCP reduces to  the space fractional Poisson process (see Orsingher and Polito (2012)).

Let $h_{D_\alpha(t)}(\cdot,t)$ be the density of $\{D_\alpha(t)\}_{t\ge0}$. Then, the pgf $G^\alpha(\bar{u},t)$ of MGSFCP can be obtained as follows:
\begin{align}
G^\alpha(\bar{u},t)&=\int_{	0}^{\infty}G(\bar{u},x)h_{D_\alpha(t)}(x,t)\,\mathrm{d}x\nonumber\\
&=\int_{0}^{\infty}\exp\bigg(-x\sum_{i=1}^{q}\sum_{j_i=1}^{k_i}\lambda_{ij_i}(1-u_i^{j_i})\bigg)h_{D_\alpha(t)}(x,t)\,\mathrm{d}x,\,\, \text{(using \eqref{pgffmgcp})}\nonumber\\
&=\exp\bigg(-t\bigg(\sum_{i=1}^{q}\sum_{j_i=1}^{k_i}\lambda_{ij_i}(1-u_i^{j_i})\bigg)^\alpha\bigg),\label{zz}
\end{align}
where the last step follows on using the Laplace transform of $\alpha$-stable subordinator (see Section \ref{seci}).
It follows that the pgf of MGSFCP satisfies the following differential equation:
\begin{equation}\label{zzeq}
	\frac{\partial}{\partial t}G^\alpha(\bar{u},t)=-\bigg(\sum_{i=1}^{q}\sum_{j_i=1}^{k_i}\lambda_{ij_i}(1-u_i^{j_i})\bigg)^\alpha G^\alpha(\bar{u},t),\,\, G^\alpha(\bar{u},0)=1.
\end{equation}
\begin{remark}
On substituting $q=1$ in \eqref{zz}, we get the pgf of generalized space fractional counting process (see Kataria \textit{et al.} (2022)). 
\end{remark}

\begin{proposition}
Let $X_l$, $l\ge1$ be iid uniform random variables in $[0,1]$. Then,
\begin{equation*}
	G^\alpha(\bar{u},t)=\mathrm{Pr}\bigg\{\min_{0\le l\le N(t)}X_l^{1/\alpha}\ge1-\frac{1}{\lambda}\sum_{i=1}^{q}\sum_{j_i=1}^{k_i}\lambda_{ij_i}u_i^{j_i}\bigg\},\, 0<u_i<1,
\end{equation*}
where $\lambda=\sum_{i=1}^{q}\sum_{j_i=1}^{k_i}\lambda_{ij_i}$ and $\{N(t)\}_{t\ge0}$ is a Poisson process with parameter $\lambda^\alpha$  
such that $\min_{0\le l\le N(t)}X_l^{1/\alpha}=1$ whenever $N(t)=0$. 
\end{proposition}
\begin{proof}
From \eqref{zz}, we have
\begin{align*}
G^\alpha(\bar{u},t)&=\exp\bigg(-t\lambda^\alpha\bigg(1-\frac{1}{\lambda}\sum_{i=1}^{q}\sum_{j_i=1}^{k_i}\lambda_{ij_i}u_i^{j_i}\bigg)^\alpha\bigg)\\
&=e^{-t\lambda^\alpha}\sum_{r=0}^{\infty}\frac{(t\lambda^\alpha)^r}{r!}\bigg(1-\bigg(1-\frac{1}{\lambda}\sum_{i=1}^{q}\sum_{j_i=1}^{k_i}\lambda_{ij_i}u_i^{j_i}\bigg)^\alpha\bigg)^r\\
&=e^{-t\lambda^\alpha}\sum_{r=0}^{\infty}\bigg(\mathrm{Pr}\bigg\{X_l^{1/\alpha}\ge1-\frac{1}{\lambda}\sum_{i=1}^{q}\sum_{j_i=1}^{k_i}\lambda_{ij_i}u_i^{j_i}\bigg\}\bigg)^r\frac{(t\lambda^\alpha)^r}{r!}\\
&=\sum_{r=0}^{\infty}\mathrm{Pr}\bigg\{\min_{0\le l\le r}X_l^{1/\alpha}\ge1-\frac{1}{\lambda}\sum_{i=1}^{q}\sum_{j_i=1}^{k_i}\lambda_{ij_i}u_i^{j_i}\bigg\}\frac{(t\lambda^\alpha)^r}{r!}e^{-t\lambda^\alpha}\\
&=\mathrm{Pr}\bigg\{\min_{0\le l\le N(t)}X_l^{1/\alpha}\ge1-\frac{1}{\lambda}\sum_{i=1}^{q}\sum_{j_i=1}^{k_i}\lambda_{ij_i}u_i^{j_i}\bigg\}.
\end{align*}
 This completes the proof.
\end{proof}
\begin{proposition}\label{prppp}
	The state probabilities $p^\alpha(\bar{n},t)=\mathrm{Pr}\{\bar{M}^\alpha(t)=\bar{n}\}$, $\bar{n}\ge\bar{0}$  solve the following differential equation:
	\begin{equation*}
		\frac{\mathrm{d}}{\mathrm{d}t}p^\alpha(\bar{n},t)=-\sum_{\substack{m_i\geq 0\\ i=1,2,\dots,q}}\sum_{\substack{\Omega(k_i,m_i)\\i=1,2,\dots,q}}\frac{\lambda^\alpha\Gamma(\alpha +1)p^\alpha(\bar{n}-\bar{m},t)}{\Gamma(\alpha -\sum_{i=1}^{q}\sum_{j_i=1}^{k_i}x_{ij_i}+1)}\prod_{i=1}^{q}\prod_{j_i=1}^{k_i}\frac{\big(-\lambda_{ij_i}/\lambda\big)^{x_{ij_i}}}{x_{ij_i}!}
	\end{equation*}
	with initial condition $p^\alpha(\bar{n},0)=\mathbb{I}_{\{\bar{n}=\bar{0}\}}$, where $\lambda=\sum_{i=1}^{q}\sum_{j_i=1}^{k_i}\lambda_{ij_i}$ and  $\Omega(k_i,m_i)=\{(x_{i1},  x_{i2}$, $\dots,  x_{ik_i} ):\sum_{j_i=1}^{k_i}j_ix_{ij_i}=m_i,\, x_{ij_i}\in\mathbb{N}_{0}\}$.
\end{proposition}
\begin{proof}
	We have
		\begin{align*}
			p^\alpha(\bar{n},t+h)&=p^\alpha(\bar{n},t)\mathrm{Pr}\{\bar{M}^\alpha(t+h)=\bar{n}|\bar{M}^\alpha(t)=\bar{n}\}\\
			&\ \  +\sum_{\bar{m}\succ\bar{0}}p^\alpha(\bar{n}-\bar{m},t)\mathrm{Pr}\{\bar{M}^\alpha(t+h)=\bar{n}|\bar{M}^\alpha(t)=\bar{n}-\bar{m}\}+o(h).
		\end{align*}
		Using \eqref{transprbstable}, we get
		\begin{align*}
			p^\alpha(\bar{n},t+h)&=o(h)+p^\alpha(\bar{n},t)(1-\lambda^\alpha h)\\
			&\ \ -h\sum_{\bar{m}\succ\bar{0}}p^\alpha(\bar{n}-\bar{m},t)\displaystyle\sum_{\substack{\Omega(k_i,m_i)\\i=1,2,\dots,q}}\frac{\lambda^\alpha\Gamma(1+\alpha)}{\Gamma(\alpha-\sum_{i=1}^{q}\sum_{j_i=1}^{k_i}x_{ij_i}+1)}\prod_{i=1}^{q}\prod_{j_i=1}^{k_i}\frac{(-\lambda_{ij_i}/\lambda)^{x_{ij_i}}}{x_{ij_i}!}.
		\end{align*}
		Equivalently,
		\begin{align*}
			\frac{p^\alpha(\bar{n},t+h)-p^\alpha(\bar{n},t)}{h}&=\frac{o(h)}{h}-\lambda^\alpha p^\alpha(\bar{n},t)\\
			&\ \ -\sum_{\bar{m}\succ\bar{0}}\displaystyle\sum_{\substack{\Omega(k_i,m_i)\\i=1,2,\dots,q}}\frac{\lambda^\alpha\Gamma(1+\alpha)p^\alpha(\bar{n}-\bar{m},t)}{\Gamma(\alpha-\sum_{i=1}^{q}\sum_{j_i=1}^{k_i}x_{ij_i}+1)}\prod_{i=1}^{q}\prod_{j_i=1}^{k_i}\frac{(-\lambda_{ij_i}/\lambda)^{x_{ij_i}}}{x_{ij_i}!}.
		\end{align*}
		On letting $h\to 0$, we get the required result.
\end{proof}
An alternate proof of Proposition \ref{prppp} is given in Appendix A1.
\begin{theorem}\label{thmyy}
The state probabilities $p^\alpha(\bar{n},t)$, $\bar{n}\ge\bar{0}$ of MGSFCP are given by
\begin{equation*}
p^\alpha(\bar{n},t)=\sum_{\substack{\Omega(k_i,n_i)\\i=1,2,\dots,q}}\sum_{r=0}^{\infty}\frac{(-t\lambda^\alpha )^r}{r!}\frac{\Gamma(\alpha r+1)}{\Gamma(\alpha r-\sum_{i=1}^{q}\sum_{j_i=1}^{k_i}x_{ij_i}+1)}\prod_{i=1}^{q}\prod_{j_i=1}^{k_i}\frac{\big(-\lambda_{ij_i}/\lambda\big)^{x_{ij_i}}}{x_{ij_i}!},
\end{equation*}
where $\lambda=\sum_{i=1}^{q}\sum_{j_i=1}^{k_i}\lambda_{ij_i}$ and $\Omega(k_i,n_i)=\{(x_{i1}$, $ x_{i2}, \dots, x_{ik_i} ):\sum_{j_i=1}^{k_i}j_ix_{ij_i}=n_i,\, x_{ij_i}\in\mathbb{N}_{0}\}$.
\end{theorem}
\begin{proof}
From \eqref{zz}, we have
{\small\begin{align}
G^\alpha(\bar{u},t)&=\exp\bigg(-t\lambda^\alpha\bigg(1-\frac{1}{\lambda}\sum_{i=1}^{q}\sum_{j_i=1}^{k_i}\lambda_{ij_i}u_i^{j_i}\bigg)^\alpha\bigg)\nonumber\\
&=\sum_{r=0}^{\infty}\frac{(-t\lambda^\alpha )^r}{r!}\bigg(1-\frac{1}{\lambda}\sum_{i=1}^{q}\sum_{j_i=1}^{k_i}\lambda_{ij_i}u_i^{j_i}\bigg)^{\alpha r}\nonumber\\
&=\sum_{r=0}^{\infty}\frac{(-t\lambda^\alpha )^r}{r!}\sum_{x\geq0}\binom{\alpha r}{x}\bigg(-\frac{1}{\lambda}\bigg)^x\bigg(\sum_{i=1}^{q}\sum_{j_i=1}^{k_i}\lambda_{ij_i}u_i^{j_i}\bigg)^x\nonumber\\
&=\sum_{r=0}^{\infty}\frac{(-t\lambda^\alpha )^r}{r!}\sum_{x\geq0}\binom{\alpha r}{x}\bigg(-\frac{1}{\lambda}\bigg)^x\sum_{\substack{\sum_{i=1}^{q}x_i=x\\x_i\in\mathbb{N}_0}}x!\prod_{i=1}^{q}\frac{(\sum_{j_i=1}^{k_i}\lambda_{ij_i}u_i^{j_i})^{x_i}}{x_i!}\nonumber\\
&=\sum_{r=0}^{\infty}\frac{(-t\lambda^\alpha )^r}{r!}\sum_{x\geq0}\frac{\Gamma(\alpha r+1)}{\Gamma(\alpha r-x+1)}\bigg(-\frac{1}{\lambda}\bigg)^x\sum_{\substack{\sum_{i=1}^{q}x_i=x\\x_i\in\mathbb{N}_0}}\prod_{i=1}^{q}\sum_{\substack{\sum_{j_i=1}^{k_i}x_{ij_i}=x_i\\
x_{ij_i\in\mathbb{N}_0}}}\prod_{j_i=1}^{k_i}\frac{(\lambda_{ij_i}u_i^{j_i})^{x_{ij_i}}}{x_{ij_i}!}\nonumber\\
&=\sum_{r=0}^{\infty}\frac{(-t\lambda^\alpha )^r}{r!}\sum_{\substack{x_i\geq0\\i=1,2,\dots,q}}\frac{\Gamma(\alpha r+1)}{\Gamma(\alpha r-\sum_{i=1}^{q}x_i+1)}\prod_{i=1}^{q}\sum_{\substack{\sum_{j_i=1}^{k_i}x_{ij_i}=x_i\\
x_{ij_i\in\mathbb{N}_0}}}\prod_{j_i=1}^{k_i}\bigg(\frac{-\lambda_{ij_i}u_i^{j_i}}{\lambda}\bigg)^{x_{ij_i}}\frac{1}{x_{ij_i}!}\nonumber\\
&=\sum_{r=0}^{\infty}\frac{(-t\lambda^\alpha )^r}{r!}\sum_{\substack{n_i\ge0\\i=1,2,\dots,q}}\sum_{\substack{\Omega(k_i,n_i)\\i=1,2,\dots,q}}\frac{\Gamma(\alpha r+1)}{\Gamma(\alpha r-\sum_{i=1}^{q}\sum_{j_i=1}^{k_i}x_{ij_i}+1)}\Big(\prod_{i=1}^{q}u_i^{n_i}\Big)\prod_{i=1}^{q}\prod_{j_i=1}^{k_i}\frac{\big(-\lambda_{ij_i}/\lambda\big)^{x_{ij_i}}}{x_{ij_i}!}\nonumber\\
&=\sum_{\substack{n_i\ge0\\i=1,2,\dots,q}}\Big(\prod_{i=1}^{q}u_i^{n_i}\Big)\sum_{\substack{\Omega(k_i,n_i)\\i=1,2,\dots,q}}\sum_{r=0}^{\infty}\frac{(-t\lambda^\alpha )^r}{r!}\frac{\Gamma(\alpha r+1)}{\Gamma(\alpha r-\sum_{i=1}^{q}\sum_{j_i=1}^{k_i}x_{ij_i}+1)}\prod_{i=1}^{q}\prod_{j_i=1}^{k_i}\frac{\big(-\lambda_{ij_i}/\lambda\big)^{x_{ij_i}}}{x_{ij_i}!}.\label{**}
\end{align}}
Also, we have
\begin{equation}\label{*}
	G^\alpha(\bar{u},t)=\sum_{\substack{n_i\geq 0\\ i=1,2,\dots,q}}\Big(\prod_{i=1}^{q}u_i^{n_i}\Big)p^\alpha(\bar{n},t).
\end{equation}
Finally, on comparing the coefficient of $u_1^{n_1}u_2^{n_2}\dots u_q^{n_q}$ over the range $\bar{n}\ge\bar{0}$ in \eqref{**} and \eqref{*}, the result follows.
\end{proof}
\begin{corollary}
The pmf of MGSFCP can be written in terms of the  generalized Wright function as follows:
\begin{equation}\label{alphawright}
p^\alpha(\bar{n},t)=
\sum_{\substack{\Omega(k_i,n_i)\\i=1,2,\dots,q}}
 {}_1\Psi_1\left[\begin{matrix}
(1,\alpha)_{1,1}\\
\Big(1-\sum_{i=1}^{q}\sum_{j_i=1}^{k_i}x_{ij_i}, \alpha\Big)_{1,1}
\end{matrix}\Bigg| -\lambda^\alpha t\right]\prod_{i=1}^{q}\prod_{j_i=1}^{k_i}\frac{\big(-\lambda_{ij_i}/\lambda\big)^{x_{ij_i}}}{x_{ij_i}!},\, \bar{n}\ge\bar{0},
\end{equation}
where ${}_1\Psi_1(\cdot)$ is the generalized Wright function  given in \eqref{wrightfn}.
\end{corollary}
\begin{remark}
For $q=1$,  $p^\alpha(\bar{n},t)$ reduces to the pmf of generalized space fractional counting process (see Kataria \textit{et al.} (2022), Eq. (4.6)). Further, on substituting  $q=1$ and $k_1=1$, the pmf of MGSFCP reduces to that of space fractional Poisson process (see Orsingher and Polito (2012)).  
\end{remark}

\begin{proposition}\label{levymeasalpha}
The L\'evy measure of  MGSFCP is given by
\begin{align}\label{lemeas}
\Pi^\alpha(A_1\times A_2\times\dots\times A_q)
&=\frac{\alpha\lambda^\alpha }{\Gamma(1-\alpha)}\sum_{\bar{n}\succ\bar{0}}\sum_{\substack{\Omega(k_i,n_i)\\i=1,2,\dots,q}}\Gamma\Big(\sum_{i=1}^{q}\sum_{j_i=1}^{k_i}n_{ij_i}-\alpha\Big)\prod_{i=1}^{q}\prod_{j_i=1}^{k_i}\frac{(\lambda_{ij_i}/\lambda)^{n_{ij_i}}}{n_{ij_i}!}\mathbb{I}_{\{n_i\in A_i\}},
\end{align}
where $\Omega(k_i,n_i)=\{(n_{i1},n_{i2},\dots,n_{ik_i}):\sum_{j_i=1}^{k_i}j_in_{ij_i}=n_i,\,n_{ij_i}\in\mathbb{N}_0\}$ and $\lambda=\sum_{i=1}^{q}\sum_{x_l=1}^{k_i}\lambda_{ij_i}$.
\end{proposition}
\begin{proof}
 On using
Eq. (30.8) of Sato (1999), the L\'evy measure of $\{\bar{M}^\alpha(t)\}_{t\ge0}$ can be obtained as follows:
{\small\begin{align*}
\Pi^\alpha(A_1\times A_2\times\dots\times A_q)&=\int_{0}^{\infty}\sum_{\bar{n}\succ\bar{0}}p(\bar{n},s)\mu_{D_\alpha}(\mathrm{d}s)\\
&=\frac{\alpha}{\Gamma(1-\alpha)}\sum_{\bar{n}\succ\bar{0}}\sum_{\substack{\Omega(k_i,n_i)\\i=1,2,\dots,q}}\bigg(\prod_{i=1}^{q}\prod_{j_i=1}^{k_i}\frac{\lambda_{ij_i}^{n_{ij_i}}}{n_{ij_i}!}\mathbb{I}_{\{n_i\in A_i\}}\bigg)\int_{0}^{\infty}e^{-\lambda s}s^{\sum_{i=1}^{q}\sum_{j_i=1}^{k_i}n_{ij_i}-\alpha-1}\mathrm{d}s\\
&=\frac{\alpha\lambda^\alpha }{\Gamma(1-\alpha)}\sum_{\bar{n}\succ\bar{0}}\sum_{\substack{\Omega(k_i,n_i)\\i=1,2,\dots,q}}\Gamma\Big(\sum_{i=1}^{q}\sum_{j_i=1}^{k_i}n_{ij_i}-\alpha\Big)\prod_{i=1}^{q}\prod_{j_i=1}^{k_i}\frac{(\lambda_{ij_i}/\lambda)^{n_{ij_i}}}{n_{ij_i}!}\mathbb{I}_{\{n_i\in A_i\}},
\end{align*}}
where the second step follows on using \eqref{jopmf} and the L\'evy measure of an $\alpha$-stable subordinator given in \eqref{lmeasalpha}. This completes the proof.
\end{proof}
\begin{remark}
On substituting $k_1=k_2=\dots=k_q=1$ in \eqref{lemeas}, we obtain the L\'evy measure  of multivariate space fractional Poisson process (see Beghin and Macci (2016), Eq. (24)).
\end{remark}

\begin{remark}
Let us consider a process $\{\bar{X}^\alpha(t)\}_{t\ge0}$, $\alpha\in(0,1]$ such that its state probabilities $p_{\bar{X}^\alpha}(\bar{n},t)$, $\bar{n}\ge\bar{0}$ satisfy the following system of differential equations:
\begin{equation}\label{aba}
	\frac{\mathrm{d}}{\mathrm{d}t}p_{\bar{X}^\alpha}(\bar{n},t)=-\sum_{\substack{m_i\ge0\\i=1,2,\dots,q}}\sum_{\substack{\Omega(k_i,m_i)\\i=1,2,\dots,q}}\frac{\lambda^\alpha \Gamma(\alpha +1)p_{\bar{X}^\alpha}(\bar{n}-\bar{m},t)}{\Gamma(\alpha -\sum_{i=1}^{q}\sum_{j_i=1}^{k_i}x_{ij_i}+1)}\prod_{i=1}^{q}\prod_{j_i=1}^{k_i}\frac{\big(-\lambda_{ij_i}/\lambda\big)^{x_{ij_i}}}{x_{ij_i}!}
\end{equation}
with $p_{\bar{X}^\alpha}(\bar{n},0)=\mathbb{I}_{\{\bar{n}=\bar{0}\}}$, where
$\lambda=\sum_{i=1}^{q}\sum_{j_i=1}^{k_i}\lambda_{ij_i}$ and $\Omega(k_i,m_i)=\{(x_{i1}$, $ x_{i2}, \dots, x_{ik_i} ):\sum_{j_i=1}^{k_i}j_ix_{ij_i}=m_i,\, x_{ij_i}\in\mathbb{N}_{0}\}$.

\paragraph{Case I}
For $\alpha\in(0,1)$, 
it can be shown that the pgf $G_{\bar{X}^\alpha}(\bar{u},t)$, $|u_i|\le1$ of $\{\bar{X}^\alpha(t)\}_{t\ge0}$ satisfies
\begin{equation*}
\frac{\partial}{\partial t}G_{\bar{X}^\alpha}(\bar{u},t)=-\bigg(\sum_{i=1}^{q}\sum_{j_i=1}^{k_i}\lambda_{ij_i}(1-u_i^{j_i})\bigg)^\alpha G_{\bar{X}^\alpha}(\bar{u},t),\,\, G_{\bar{X}^\alpha}(\bar{u},0)=1.
\end{equation*}
Hence,
\begin{equation*}
G_{\bar{X}^\alpha}(\bar{u},t)=\exp\bigg(-t\bigg(\sum_{i=1}^{q}\sum_{j_i=1}^{k_i}\lambda_{ij_i}(1-u_i^{j_i})\bigg)^\alpha\bigg), \, |u_i|\le1.
\end{equation*}

Thus, for $\alpha\in(0,1)$, we have
\begin{equation*}
	\bar{X}^\alpha(t)\overset{d}{=}\bar{M}^\alpha(t).
\end{equation*}

\paragraph{Case II}
For $\alpha=1$, \eqref{aba} reduces to the governing system of differential equations of the pmf of MGCP. In this case, it can be shown that the pgf $G_{\bar{X}^1}(\bar{u},t)$ reduces to the pgf of MGCP. Thus,  $\bar{X}^1(t)\overset{d}{=}\bar{M}(t)$.
\end{remark}

\subsection{Multivariate generalized fractional counting process}
Here, we consider a multivariate version of the GFCP introduced and studied by Di Crescenzo \textit{et al.} (2016). Let $\{Y_\beta(t)\}_{t\ge0}$, $0<\beta<1$ be an inverse $\beta$-stable subordinator independent of $\{M_i(t)\}_{t\ge0}$, $i=1,2,\dots,q$. Then, we define the time-changed process $\{\bar{\mathcal{M}}^\beta(t)\}$ as follows:
\begin{equation}\label{invrep}
\bar{\mathcal{M}}^\beta(t)\coloneqq\bar{M}(Y_\beta(t)),\, 0<\beta<1.
\end{equation}
We call it the multivariate generalized fractional counting process (MGFCP).

For $q=1$, it reduces to the GFCP (see Di Crescenzo \textit{et al.} (2016), Kataria and Khandakar (2022a)). Further, for $q=1$ and $k_1=1$, it reduces to the time fractional Poisson process (see Meershaert \textit{et al.} (2011)).

The pgf $H^\beta(\bar{u},t)=\mathbb{E}\bigg(\prod_{i=1}^{q}u_i^{M_i(Y_\beta(t))}\bigg),\, |u_i|\le1$ of $\{\bar{\mathcal{M}}^\beta(t)\}_{t\ge0}$ can be obtained as follows:
\begin{align}
H^\beta(\bar{u},t)&=\mathbb{E}\Big(\mathbb{E}\Big(\prod_{i=1}^{q}u_i^{M_i(Y_\beta(t))}\Big|Y_\beta(t)\Big)\Big)\nonumber\\
&=\mathbb{E}\Bigg(\exp\Big(-Y_\beta(t)\sum_{i=1}^{q}\sum_{j_i=1}^{k_i}\lambda_{ij_i}(1-u_i^{j_i})\Big)\bigg),\,\, \text{(using \eqref{pgffmgcp})}\nonumber\\
&=E_{\beta,1}\bigg(-t^\beta\sum_{i=1}^{q}\sum_{j_i=1}^{k_i}\lambda_{ij_i}(1-u_i^{j_i})\bigg),\label{invpgf}
\end{align}
where the last step follows from the Laplace transform of an inverse stable subordinator (see Section $\ref{seci}$).
It is known that the Mittag-Leffler function is an eigenfunction of the Caputo fractional
derivative. Therefore, the pgf $H^\beta(\bar{u},t)$ satisfies the following Cauchy problem:
\begin{equation}\label{invpgfde}
\frac{\partial^\beta}{\partial t^\beta}H^\beta(\bar{u},t)=-\sum_{i=1}^{q}\sum_{j_i=1}^{k_i}\lambda_{ij_i}(1-u_i^{j_i})H^\beta(\bar{u},t),\,\, H^\beta(\bar{u},0)=1,
\end{equation} 
where $\frac{\partial^\beta}{\partial t^\beta}$ is the Caputo fractional derivative given in \eqref{caputo}.
\begin{remark}
On substituting $q=1$ in \eqref{invpgfde}, it reduces to the Cauchy problem given in Di Crescenzo \textit{et al.} (2016), Proposition 2.1. Further, for $q=1$ and $\lambda_{1j_1}=\lambda_{j_1}$, the pgf \eqref{invpgf} reduces to $H^\beta(\bar{u},t)|_{q=1}=E_{\beta,1}(-t^\beta\sum_{j_1=1}^{k_1}\lambda_{j_1}(1-u_1^{j_1}))$ which agrees with the pgf of GFCP (see Kataria and Khandakar (2022a)).
\end{remark}
Next, we obtain the governing system of fractional differential equations for the pmf of MGFCP.
\begin{proposition}
The pmf $q^\beta(\bar{n},t)=\mathrm{Pr}\{\bar{\mathcal{M}}^\beta(t)=\bar{n}\}$, $\bar{n}\ge\bar{0}$ satisfies
\begin{equation}\label{invpmfde}
\frac{\mathrm{d}^\beta}{\mathrm{d}t^\beta}q^\beta(\bar{n},t)=-\lambda q^\beta(\bar{n},t)+\sum_{i=1}^{q}\sum_{j_i=1}^{k_i}\lambda_{ij_i}q^\beta(\bar{n}-\bar{\epsilon}^{j_i}_i,t),\,\, q^\beta(\bar{n},0)=\mathbb{I}_{\{\bar{n}=\bar{0}\}},
\end{equation}
where $\lambda=\sum_{i=1}^{q}\sum_{j_i=1}^{k_i}\lambda_{ij_i}$, and $\bar{\epsilon}_i^{j_i}\in \mathbb{N}_{0}^q$ is a $q$-tuple vector whose $i^{th}$ entry is $j_i$ and other entries are zero.
\end{proposition} 
\begin{proof}
	From \eqref{invrep}, we have
	\begin{equation}\label{CCCC}
		q^\beta(\bar{n},t)=\int_{0}^{\infty}p(\bar{n},x)g_\beta(t,x)\,\mathrm{d}x,
	\end{equation}
	where $g_\beta(t,\cdot)$
	is the density of $\{Y_\beta(t)\}_{t\ge0}$ and $p(\bar{n},x)$ is the pmf of MGCP. On taking the Laplace transform of \eqref{CCCC}, we get
	\begin{equation}\label{aaaa}
		\tilde{q}^\beta(\bar{n},s)=s^{\beta -1}\int_{0}^{\infty}p(\bar{n},x)e^{-s^\beta x}\, \mathrm{d}x=s^{\beta-1}\tilde{p}(\bar{n},s^\beta)
	\end{equation}
 which follows on using Eq. (3.13) of Meerschaert and Scheffler (2008).

Moreover, on taking the Laplace transform with respect to $t$ on both sides of \eqref{propmgcp}, we have
	\begin{equation*}
		s\tilde{p}(\bar{n},s)-p(\bar{n},0)=-\lambda\tilde{p}(\bar{n},s)+\sum_{i=1}^{q}\sum_{j_i=1}^{k_i}\lambda_{ij_i}\tilde{p}(\bar{n}-\bar{\epsilon}^{j_i}_i,s),
	\end{equation*}
	which on using \eqref{aaaa} gives
	\begin{equation}\label{gh}
		s^\beta\tilde{q}^\beta(\bar{n},s)-s^{\beta-1}q^\beta(\bar{n},0)=-\lambda\tilde{q}^\beta(\bar{n},s)+\sum_{i=1}^{q}\sum_{j_i=1}^{k_i}\lambda_{ij_i}\tilde{q}^\beta(\bar{n}-\bar{\epsilon}^{j_i}_i,s).
	\end{equation} 
	Now, on taking the inverse Laplace transform in \eqref{gh}, we get \eqref{invpmfde}. 
	
	The following is an alternate proof:
	From \eqref{invpgfde}, we have
{\footnotesize\begin{align}\label{invdel}
\sum_{\substack{n_l\geq 0\\  l=1,2,\dots, q}}\Big(\prod_{l=1}^{q}u_l^{n_l}\Big)\frac{\mathrm{d}^\beta}{\mathrm{d} t^\beta}q^\beta(\bar{n},t)
&=-\lambda \sum_{\substack{n_l\geq 0\\  l=1,2,\dots, q}}\Big(\prod_{l=1}^{q}u_l^{n_l}\Big)q^\beta(\bar{n},t)+\sum_{\substack{n_l\geq 0\\  l=1,2,\dots, q}} \sum_{i=1}^{q}\bigg(\prod_{\substack{l=1\\ l\neq i}}^{q}u_l^{n_l}\bigg)\sum_{j_i=1}^{k_i}\lambda_{ij_i}q^\beta(\bar{n},t)u_i^{n_i+j_i}\nonumber\\
&=-\lambda \sum_{\substack{n_l\geq 0\\  l=1,2,\dots, q}}\Big(\prod_{l=1}^{q}u_l^{n_l}\Big)q^\beta(\bar{n},t)+\sum_{i=1}^{q}\sum_{n_i\ge j_i}\sum_{\substack{n_l\geq 0\\  l=1,2,\dots, q\\l\ne i}}\Big(\prod_{l=1}^{q}u_l^{n_l}\Big) \sum_{j_i=1}^{k_i}\lambda_{ij_i}q^\beta(\bar{n}-\bar{\epsilon}^{j_i}_i,t)\nonumber\\
&=\sum_{\substack{n_l\geq 0\\  l=1,2,\dots, q}}\Big(\prod_{l=1}^{q}u_l^{n_l}\Big)\bigg(-\lambda q^\beta(\bar{n},t)+ \sum_{i=1}^{q}\sum_{j_i=1}^{k_i}\lambda_{ij_i}q^\beta(\bar{n}-\bar{\epsilon}^{j_i}_i,t)\bigg),
\end{align}}
 where in the last step we neglected the restriction $n_i\ge j_i$. On comparing the coefficient of $u_1^{n_1}u_2^{n_2}\dots u_q^{n_q}$ over the range of $\bar{n}\ge\bar{0}$ on both sides of equation \eqref{invdel}, we get the required result.
\end{proof}
\begin{remark}
	For $q=1$, the governing system of differential equations \eqref{invpmfde} of MGFCP reduces to that of GFCP (see Di Crescenzo \textit{et al.} (2016), Eq. (2.3)).
\end{remark}


\begin{theorem}
The state probabilities $q^\beta(\bar{n},t)$, $\bar{n}\ge\bar{0}$ of MGFCP are given by 
\begin{equation}\label{invpmf}
q^\beta(\bar{n},t)=\sum_{\substack{\Omega(k_i,n_i)\\i=1,2,\dots,q}}\Big(\sum_{i=1}^{q}\sum_{j_i=1}^{k_i}x_{ij_i}\Big)!E_{\beta,\,\beta\sum_{i=1}^{q}\sum_{j_i=1}^{k_i}x_{ij_i} +1}^{\sum_{i=1}^{q}\sum_{j_i=1}^{k_i}x_{ij_i}+1}(-\lambda t^\beta)\prod_{i=1}^{q}\prod_{j_i=1}^{k_i}\frac{(\lambda_{ij_i}t^{\beta })^{x_{ij_i}}}{x_{ij_i}!},
\end{equation}
where  $\lambda=\sum_{i=1}^{q}\sum_{j_i=1}^{k_i}\lambda_{ij_i}$, $\Omega(k_i,n_i)=\{(x_{i1},x_{i2},\dots,x_{ik_i}):\sum_{j_i=1}^{k_i}j_ix_{ij_i}=n_i,\, x_{ij_i}\in\mathbb{N}_{0}\}$ and $E_{\alpha,\,\beta}^{\gamma}(\cdot)$ is the generalized Mittag-Leffler function defined in \eqref{Mitagdef}.
\end{theorem}
\begin{proof}

From \eqref{invpgf}, we get
{\small\begin{align}
H^\beta(\bar{u},t)
	&=\sum_{r=0}^{\infty}\frac{(-\lambda t^\beta)^r}{\Gamma(\beta r+1)}\bigg(1-\frac{1}{\lambda}\sum_{i=1}^{q}\sum_{j_i=1}^{k_i}\lambda_{ij_i}u_i^{j_i}\bigg)^r\nonumber\\
	&=\sum_{r=0}^{\infty}\frac{(-\lambda t^\beta)^r}{\Gamma(\beta r+1)}\sum_{x=0}^{r}\frac{r!}{x!(r-x)!}\Big(-\frac{1}{\lambda}\Big)^x\bigg(\sum_{i=1}^{q}\sum_{j_i=1}^{k_i}\lambda_{ij_i}u_i^{j_i}\bigg)^x\nonumber\\
	&=\sum_{r=0}^{\infty}\frac{(-\lambda t^\beta)^r}{\Gamma(\beta r+1)}\sum_{x=0}^{r}\frac{r!}{(r-x)!}\Big(-\frac{1}{\lambda}\Big)^x\sum_{\substack{\sum_{i=1}^{q}x_i=x\\x_i\in\mathbb{N}_0}}\prod_{i=1}^{q}\frac{1}{x_i!}\bigg(\sum_{j_i=1}^{k_i}\lambda_{ij_i}u_i^{j_i}\bigg)^{x_i}\nonumber\\
	&=\sum_{r=0}^{\infty}\frac{(-\lambda t^\beta)^r}{\Gamma(\beta r+1)}\sum_{x=0}^{r}\frac{r!}{(r-x)!}\Big(-\frac{1}{\lambda}\Big)^x\sum_{\substack{\sum_{i=1}^{q}x_i=x\\x_i\in\mathbb{N}_0}}\prod_{i=1}^{q}\sum_{\substack{\sum_{j_i=1}^{k_i}x_{ij_i}=x_i\\
			x_{ij_i\in\mathbb{N}_0}}}\prod_{j_i=1}^{k_i}\frac{(\lambda_{ij_i}u_i^{j_i})^{x_{ij_i}}}{x_{ij_i}!}\nonumber\\
	&=\sum_{x=0}^{\infty}\sum_{r=x}^{\infty}\frac{r!(-\lambda t^\beta)^r }{(r-x)!\Gamma(\beta r+1)}\Big(-\frac{1}{\lambda}\Big)^x\sum_{\substack{\sum_{i=1}^{q}x_i=x\\x_i\in\mathbb{N}_0}}\prod_{i=1}^{q}\sum_{\substack{\sum_{j_i=1}^{k_i}x_{ij_i}=x_i\\
			x_{ij_i\in\mathbb{N}_0}}}\prod_{j_i=1}^{k_i}\frac{(\lambda_{ij_i}u_i^{j_i})^{x_{ij_i}}}{x_{ij_i}!}\nonumber\\
	&=\sum_{x=0}^{\infty}\sum_{r=0}^{\infty}\frac{(r+x)!(-\lambda t^\beta)^r}{r!\Gamma(\beta r+\beta x+1)}t^{\beta x}\sum_{\substack{\sum_{i=1}^{q}x_i=x\\x_i\in\mathbb{N}_0}}\prod_{i=1}^{q}\sum_{\substack{\sum_{j_i=1}^{k_i}x_{ij_i}=x_i\\
			x_{ij_i\in\mathbb{N}_0}}}\prod_{j_i=1}^{k_i}\frac{(\lambda_{ij_i}u_i^{j_i})^{x_{ij_i}}}{x_{ij_i}!}\nonumber\\
	&=\sum_{x=0}^{\infty}x!E_{\beta,\,\beta x+1}^{x+1}(-\lambda t^\beta)t^{\beta x}\sum_{\substack{\sum_{i=1}^{q}x_i=x\\x_i\in\mathbb{N}_0}}\prod_{i=1}^{q}\sum_{\substack{\sum_{j_i=1}^{k_i}x_{ij_i}=x_i\\
			x_{ij_i\in\mathbb{N}_0}}}\prod_{j_i=1}^{k_i}\frac{(\lambda_{ij_i}u_i^{j_i})^{x_{ij_i}}}{x_{ij_i}!}\nonumber\\
	&=\sum_{\substack{x_i\geq0\\ 1\leq i\leq q}} \Big(\sum_{i=1}^{q}x_i\Big)! E_{\beta,\,\beta\sum_{i=1}^{q}x_i +1}^{\sum_{i=1}^{q}x_i+1}(-\lambda t^\beta)t^{\beta\sum_{i=1}^{q}x_i} \prod_{i=1}^{q}\sum_{\substack{\sum_{j_i=1}^{k_i}x_{ij_i}=x_i\\
			x_{ij_i\in\mathbb{N}_0}}}\prod_{j_i=1}^{k_i}\frac{(\lambda_{ij_i}u_i^{j_i})^{x_{ij_i}}}{x_{ij_i}!}\nonumber\\
	&=\sum_{\substack{n_i\ge0\\i=1,2,\dots,q}}\Big(\prod_{i=1}^{q}u_i^{n_i}\Big)\sum_{\substack{\Omega(k_i,n_i)\\i=1,2,\dots,q}}\Big(\sum_{i=1}^{q}\sum_{j_i=1}^{k_i}x_{ij_i}\Big)!E_{\beta,\,\beta\sum_{i=1}^{q}\sum_{j_i=1}^{k_i}x_{ij_i} +1}^{\sum_{i=1}^{q}\sum_{j_i=1}^{k_i}x_{ij_i}+1}(-\lambda t^\beta) \prod_{i=1}^{q}\prod_{j_i=1}^{k_i}\frac{(\lambda_{ij_i}t^{\beta})^{x_{ij_i}}}{x_{ij_i}!},\label{&&}
\end{align}}
Also, we have
\begin{equation}\label{&}
	H^\beta(\bar{u},t)=\sum_{\substack{n_i\geq 0\\  i=1,2,\dots, q}}\Big(\prod_{i=1}^{q}u_i^{n_i}\Big)q^\beta(\bar{n},t).
\end{equation}
Finally, the result follows on comparing the coefficient of $u_1^{n_1}u_2^{n_2}\dots u_q^{n_q}$ over the range $\bar{n}\ge\bar{0}$ on both sides of \eqref{&&} and \eqref{&}.
\end{proof}
\begin{remark}
For $k_1=k_2=\dots=k_q=1$, the pmf \eqref{invpmf} reduces to that of multivariate time fractional Poisson process (see Beghin and Macci (2016), p. 701). Moreover, For $q=1$, the pmf of MGFCP reduces to that of GFCP (see Di Crescenzo \textit{et al.} (2016), Eq. (2.8)).
\end{remark}
\begin{proposition}\label{prpcov}
Let $0<\beta<1$. For any $i=1,2\dots,q$ and $l=1,2\dots,q$, the covariance of $\{M_i(Y_\beta(t))\}_{t\ge0}$ and $\{M_l(Y_\beta(t))\}_{t\ge0}$ is given by 
\begin{equation*}
\operatorname{Cov}(M_i(Y_\beta(t)),M_l(Y_\beta(t)))=\frac{\sum_{j_i=1}^{k_i}j_i^2\lambda_{ij_i}t^{\beta}}{\Gamma(\beta+1)}\mathbb{I}_{\{i=l\}}+\sum_{j_i=1}^{k_i}\sum_{j_l=1}^{k_l}j_i\lambda_{ij_i}j_l\lambda_{lj_l}t^{2\beta}\bigg(\frac{2}{\Gamma(2\beta+1)}-\frac{1}{\Gamma^2(\beta+1)}\bigg).
\end{equation*}
\end{proposition}
\begin{proof}
For $i=l$, we have
\begin{equation*}
\operatorname{Cov}(M_i(Y_\beta(t)),M_l(Y_\beta(t)))=\frac{\sum_{j_i=1}^{k_i}j_i^2\lambda_{ij_i}t^{\beta}}{\Gamma(\beta+1)}+\bigg(\sum_{j_i=1}^{k_i}j_i\lambda_{ij_i}t^{\beta}\bigg)^2\bigg(\frac{2}{\Gamma(2\beta+1)}-\frac{1}{\Gamma^2(\beta+1)}\bigg)
\end{equation*}
which follows from \eqref{meanvargfcpvv}.
For $i\ne l$, we have
\begin{align*}
\mathbb{E}(M_i(Y_\beta(t))M_l(Y_\beta(t)))&=\mathbb{E}(\mathbb{E}(M_i(Y_\beta(t))M_l(Y_\beta(t))|Y_\beta(t)))\nonumber\\
&=\mathbb{E}\bigg(\sum_{j_i=1}^{k_i}\sum_{j_l=1}^{k_l}j_i\lambda_{ij_i}j_l\lambda_{lj_l}Y_\beta^2(t)\bigg)\nonumber\\
&=\sum_{j_i=1}^{k_i}\sum_{j_l=1}^{k_l}j_i\lambda_{ij_i}j_l\lambda_{lj_l}\mathbb{E}(Y_\beta^2(t))\\
&=\sum_{j_i=1}^{k_i}\sum_{j_l=1}^{k_l}j_i\lambda_{ij_i}j_l\lambda_{lj_l}\frac{2t^{2\beta}}{\Gamma(2\beta+1)},
\end{align*}
where we have used \eqref{beghin}.
Thus, on using  \eqref{meanvargfcp}, we get
\begin{equation*}
\operatorname{Cov}(M_i(Y_\beta(t)),M_l(Y_\beta(t)))=\sum_{j_i=1}^{k_i}\sum_{j_l=1}^{k_l}j_i\lambda_{ij_i}j_l\lambda_{lj_l}t^{2\beta}\bigg(\frac{2}{\Gamma(2\beta+1)}-\frac{1}{\Gamma^2(\beta+1)}\bigg).
\end{equation*}
This proves the result.
\end{proof}
Next, we obtain the codifference of  $\{M_i(Y_\beta(t))\}_{t\ge0}$ and $\{M_l(Y_\beta(t))\}_{t\ge0}$. It is defined as follows (see Kokoszka and Taqqu (1996), Eq. (1.7)):  
\begin{equation}\label{cod}
\tau(M_i(Y_\beta(t)),M_l(Y_\beta(t))):=\ln\mathbb{E}(e^{\omega(M_i(Y_\beta(t))-M_l(Y_\beta(t)))})-\ln\mathbb{E}(e^{\omega M_i(Y_\beta(t))})-\ln\mathbb{E}(e^{-\omega M_l(Y_\beta(t))}),
\end{equation} 
where $\omega=\sqrt{-1}$ and $1\le i,j \le q$. The following result will be used (see Di Crescenzo \textit{et al.} (2016), Eq. (2.5)):
\begin{equation}\label{rprp}
	\mathbb{E}(e^{\omega uM_i(Y_\beta(t))})=E_{\beta,1}\bigg(-t^\beta\bigg(\sum_{j_i=1}^{k_i}\lambda_{ij_i}(1-e^{\omega uj_i})\bigg)\bigg),\, \, u\in\mathbb{R}.
\end{equation}
\begin{proposition}\label{prpcod}
For any $i=1,2,\dots,q$ and $l=1,2,\dots,q$, the codifference of $\{M_i^\beta(t)\}_{t\ge0}$ and $\{M_l^\beta(t)\}_{t\ge0}$ is given by
{\small\begin{align*}
\tau(M_i(Y_\beta(t)),M_l(Y_\beta(t)))&=\ln E_{\beta,1}\bigg(-t^\beta\bigg(\sum_{j_i=1}^{k_i}\lambda_{ij_i}(1-e^{\omega j_i})+\sum_{j_l=1}^{k_l}\lambda_{lj_l}(1-e^{-\omega j_l})\bigg)\bigg)\mathbb{I}_{\{i\ne l\}}\\
&\ \ -\ln E_{\beta,1}\bigg(-t^\beta\bigg(\sum_{j_i=1}^{k_i}\lambda_{ij_i}(1-e^{\omega j_i})\bigg)\bigg)-\ln E_{\beta,1}\bigg(-t^\beta\bigg(\sum_{j_l=1}^{k_l}\lambda_{lj_l}(1-e^{-\omega j_l})\bigg)\bigg).
\end{align*}}
\end{proposition}
\begin{proof}
For $i\ne l$, we have
\begin{align*}
\mathbb{E}(e^{\omega(M_i(Y_\beta(t))-M_l(Y_\beta(t)))})&=\mathbb{E}\bigg(\mathbb{E}(e^{\omega M_i(Y_\beta(t))}e^{-\omega M_l(Y_\beta(t))})|Y_\beta(t)\bigg)\\
&=\mathbb{E}\bigg(\exp\bigg(\bigg(\sum_{j_i=1}^{k_i}\lambda_{ij_i}(e^{\omega j_i}-1)+\sum_{j_l=1}^{k_l}\lambda_{lj_l}(e^{-\omega j_l}-1)\bigg)Y_\beta(t)\bigg)\bigg)\\
&=E_{\beta,1}\bigg(-t^\beta\bigg( \sum_{j_i=1}^{k_i}\lambda_{ij_i}(1-e^{\omega j_i})+\sum_{j_l=1}^{k_l}\lambda_{lj_l}(1-e^{-\omega j_l}) \bigg)\bigg), 
\end{align*}
where the last step follows on using the Laplace transform of inverse $\beta$-stable subordinator (see Section $\ref{seci}$).

For $i=l$, we get
 $\mathbb{E}(e^{\omega(M_i(Y_\beta(t))-M_l(Y_\beta(t)))})=1$. Finally, the result follows on using \eqref{cod} and \eqref{rprp}.
\end{proof}
\begin{remark}
For $k_i=k_l=1$, the covariance in Proposition $\ref{prpcov}$ and the codifference in Proposition $\ref{prpcod}$ reduces to a known result of  Beghin and Macci (2016), Proposition 5. 
\end{remark}

\begin{remark}
	Let us consider a process $\{\bar{\mathscr{G}}^\beta(t)\}_{t\ge0}$, $\beta\in(0,1]$ such that its state probabilities $q_{\bar{\mathscr{G}}}^\beta(\bar{n},t)$, $\bar{n}\ge\bar{0}$ satisfy the following system of differential equations:
	\begin{equation}\label{invpmfdef}
		\frac{\mathrm{d}^\beta}{\mathrm{d}t^\beta}q_{\bar{\mathscr{G}}}^\beta(\bar{n},t)=-\lambda q_{\bar{\mathscr{G}}}^\beta(\bar{n},t)+\sum_{i=1}^{q}\sum_{j_i=1}^{k_i}\lambda_{ij_i}q_{\bar{\mathscr{G}}}^\beta(\bar{n}-\bar{\epsilon}^{j_i}_i,t),
	\end{equation}
	where $\lambda=\sum_{i=1}^{q}\sum_{j_i=1}^{k_i}\lambda_{ij_i}$ and the initial condition is $q_{\bar{\mathscr{G}}}^\beta(\bar{n},0)=\mathbb{I}_{\{\bar{n}=\bar{0}\}}$.
	
	\paragraph{Case I}
	For $\beta\in(0,1)$, 
	the pgf $H_{\bar{\mathscr{G}}}^\beta(\bar{u},t)$ of $\{\bar{\mathscr{G}}^\beta(t)\}_{t\ge0}$ satisfies
	\begin{equation*}
		\frac{\partial^\beta}{\partial t^\beta}H_{\bar{\mathscr{G}}}^\beta(\bar{u},t)=-\sum_{i=1}^{q}\sum_{j_i=1}^{k_i}\lambda_{ij_i}(1-u_i^{j_i})H_{\bar{\mathscr{G}}}^\beta(\bar{u},t),\,\, H_{\bar{\mathscr{G}}}^\beta(\bar{u},0)=1.
	\end{equation*} 
	So,
	$
		H_{\bar{\mathscr{G}}}^\beta(\bar{u},t)=E_{\beta,1}(-t^\beta\sum_{i=1}^{q}\sum_{j_i=1}^{k_i}\lambda_{ij_i}(1-u_i^{j_i})), \, |u_i|\le1.
$	Thus, for $\beta\in(0,1)$, we have
	\begin{equation*}
		\bar{\mathscr{G}}^\beta(t)\overset{d}{=}\bar{\mathcal{M}}^\beta(t).
	\end{equation*}
	
	\paragraph{Case II}
	For $\beta=1$ in \eqref{invpmfdef}, we get the system of governing differential equations of the pmf of MGCP. So, the pgf $H_{\bar{\mathscr{G}}}^1(\bar{u},t)$ reduces to the pgf of MGCP. Thus,  $\bar{\mathscr{G}}^1(t)\overset{d}{=}\bar{M}(t)$.
\end{remark}
\subsection{Multivariate generalized space-time fractional counting process} A space-time fractional version of the GCP is introduced and studied by Kataria \textit{et al.} (2022). Here, we obtain a multivariate version of it by time changing MGSFCP with an independent inverse $\beta$-stable subordinator. We call it the multivariate generalized space-time fractional counting process (MGSTFCP) and denote it by  $\{\bar{M}^{\alpha,\beta}(t)\}_{t\geq0}$, $0<\alpha<1$, $0<\beta<1$. It is defined as follows:
\begin{equation}\label{mstfgcp}
\bar{M}^{\alpha,\beta}(t)\coloneqq \bar{M}(D_\alpha(Y_\beta(t))),\ t\geq0,
\end{equation}
where $\{D_\alpha(t)\}_{t\geq0}$, $\{Y_\beta(t)\}_{t\geq0}$ and $\{M_i(t)\}_{t\geq0}$, $i=1,2,\dots,q$  are independent. 
\begin{remark}
From Remark $\ref{remk}$, it follows that the component processes  $\{M_i^{\alpha,\beta}(t)\}_{t\ge0}$, $i\in\{1,2,\dots,q\}$ are conditionally independent given $\{D_\alpha(Y_\beta(t))\}_{t\geq0}$. 
\end{remark}

On using Eq. (4) of Beghin and Macci (2016) and Proposition 2.2 of Di Crescenzo \textit{et al.} (2016), we get the pgf  $G^{\alpha,\beta}(\bar{u},t)=\mathbb{E}(u_1^{M_1^{\alpha,\beta}(t)}u_2^{M_2^{\alpha,\beta}(t)}\dots u_q^{M_q^{\alpha,\beta}(t)})$, $|u_i|\leq 1$ of $\{\bar{M}^{\alpha,\beta}(t)\}_{t\geq0}$ in the following form:
\begin{align}
G^{\alpha,\beta}(\bar{u},t)
&=E_{\beta,1}\Bigg(-t^\beta\Bigg(\sum_{i=1}^{q}\sum_{j_i=1}^{k_i}\lambda_{ij_i}(1-u_i^{j_i})\Bigg)^{\alpha}\Bigg).\label{pgfmstfgcp}
\end{align}
 So, we have
\begin{equation*}
\frac{\partial^\beta}{\partial t^\beta}G^{\alpha,\beta}(\bar{u},t)=-\bigg(\sum_{i=1}^{q}\sum_{j_i=1}^{k_i}\lambda_{ij_i}(1-u_i^{j_i})\bigg)^\alpha G^{\alpha,\beta}(\bar{u},t),\ \ G^{\alpha,\beta}(\bar{u},0)=1,
\end{equation*}
which follows from the fact that the Mittag-Leffler function is an eigenfunction of the Caputo fractional derivative.

Orsingher and Polito (2012) gave a relation  between the pgf of space-time fractional Poisson process and iid uniform random variables. A similar result holds true for the case of the pgf of MGSTFCP.

\begin{theorem}
Let $X_l$, $l\geq1$ be iid uniform random variables in $[0,1]$. Then, the pgf $G^{\alpha,\beta}(\bar{u},t)$ has the following representation:
\begin{equation*}
G^{\alpha,\beta}(\bar{u},t)=\mathrm{Pr}\bigg\{\min_{0\leq l\leq N_\beta(t)}X_l^{1/\alpha}\geq 1-\frac{1}{\lambda}\sum_{i=1}^{q}\sum_{j_i=1}^{k_i}\lambda_{ij_i}u_i^{j_i}\bigg\},\, 0<u_i<1,
\end{equation*}
where $\lambda=\sum_{i=1}^{q}\sum_{j_i=1}^{k_i}\lambda_{ij_i}$ and $\{N_\beta(t)\}_{t\geq0}$ is a time fractional Poisson process with intensity $\lambda^\alpha$ such that $\min_{1\leq l\leq N_\beta(t)}X_l^{1/\alpha}=1$, $0<\alpha<1$ whenever $N_\beta(t)=0$. 
\end{theorem}
\begin{proof}
Observe that
{\footnotesize\begin{align*}
\mathrm{Pr}\Big\{\min_{0\leq l\leq N_\beta(t)}X_l^{1/\alpha}\geq 1-\frac{1}{\lambda}\sum_{i=1}^{q}\sum_{j_i=1}^{k_i}\lambda_{ij_i}u_i^{j_i}\Big\}&=\sum_{r=0}^{\infty}\Big(\mathrm{Pr}\Big\{X_l^{1/\alpha}\geq 1-\frac{1}{\lambda}\sum_{i=1}^{q}\sum_{j_i=1}^{k_i}\lambda_{ij_i}u_i^{j_i}\Big\}\Big)^r\mathrm{Pr}\{N_\beta(t)=r\}\\
&=\sum_{r=0}^{\infty}\Big(1-\Big(1-\frac{1}{\lambda}\sum_{i=1}^{q}\sum_{j_i=1}^{k_i}\lambda_{ij_i}u_i^{j_i}\Big)^\alpha\Big)^r\sum_{x=r}^{\infty}(-1)^{x-r}\binom{x}{r}\frac{(\lambda^\alpha t^\beta)^x}{\Gamma(\beta x+1)}\\
&=\sum_{x=0}^{\infty}\frac{(-1)^x(\lambda^\alpha t^\beta)^x}{\Gamma(\beta x+1)}\sum_{r=0}^{x}\binom{x}{r}(-1)^r\Big(1-\Big(1-\frac{1}{\lambda}\sum_{i=1}^{q}\sum_{j_i=1}^{k_i}\lambda_{ij_i}u_i^{j_i}\Big)^\alpha\Big)^r\\
&=\sum_{x=0}^{\infty}\frac{(-1)^x(\lambda^\alpha t^\beta)^x}{\Gamma(\beta x+1)}\Big(1-\frac{1}{\lambda}\sum_{i=1}^{q}\sum_{j_i=1}^{k_i}\lambda_{ij_i}u_i^{j_i}\Big)^{\alpha x}\\
&=E_{\beta,1}\Big(-t^\beta\Big(\sum_{i=1}^{q}\sum_{j_i=1}^{k_i}\lambda_{ij_i}(1-u_i^{j_i})\Big)^\alpha\Big)
\end{align*}}
which coincides with \eqref{pgfmstfgcp}. This establishes the result.
\end{proof}

The proof of following result is similar to that of Proposition $\ref{prppp}$. Hence, it is omitted.
 
\begin{proposition}
The pmf $p^{\alpha,\beta}(\bar{n},t)=\mathrm{Pr}\{\bar{M}^{\alpha,\beta}(t)=\bar{n}\}$, $\bar{n}\ge \bar{0}$ solves the following system of fractional differential equations:
\begin{equation*}\label{invstabdef}
\frac{\mathrm{d}^\beta}{\mathrm{d}t^\beta}p^{\alpha,\beta}(\bar{n},t)=-\sum_{\substack{m_i\ge0\\i=1,2,\dots,q}}\sum_{\substack{\Omega(k_i,m_i)\\ i=1,2,\dots, q}}\frac{\lambda^\alpha\Gamma(\alpha+1)p^{\alpha,\beta}(\bar{n}-\bar{m},t)}{\Gamma(\alpha-\sum_{i=1}^{q}\sum_{j_i=1}^{k_i}x_{ij_i}+1)}\prod_{i=1}^{q}\prod_{j_i=1}^{k_i}\frac{(-\lambda_{ij_i}/\lambda)^{x_{ij_i}}}{x_{ij_i}!},
\end{equation*}
 with initial condition $p^{\alpha,\beta}(\bar{n},0)=\mathbb{I}_{\{\bar{n}=\bar{0}\}}$. Here, $\Omega(k_i,m_i)=\{(x_{i1},x_{i2},\dots,x_{ik_i}):\sum_{j_i=1}^{k_i}j_ix_{ij_i}=m_i,\ x_{ij_i}\in\mathbb{N}_0\}$.
\end{proposition}

The next result gives an implicit expression for the pmf of MGSTFCP. 
\begin{proposition}\label{prppmf}
		For $\alpha,\beta\in(0,1)$, the pmf $p^{\alpha,\beta}(\bar{n},t)$, $\bar{n}\geq\bar{0}$ of MGSTFCP is given by
		\begin{equation*}
			p^{\alpha,\beta}(\bar{n},t)=	\sum_{\substack{\Omega(k_i,n_i)\\i=1,2,\dots,q}}\bigg(\prod_{i=1}^{q}\prod_{j_i=1}^{k_i}\frac{(-\lambda_{ij_i}\partial_{\lambda_{ij_i}})^{n_{ij_i}}}{n_{ij_i}!}\bigg)E_{\beta,1}(-\lambda^{\alpha}t^\beta),
		\end{equation*}
		where $\lambda=\sum_{i=1}^{q}\sum_{j_i=1}^{k_i}\lambda_{ij_i}$ and $\partial_ \lambda=\partial/\partial \lambda$.
	\end{proposition}
	\begin{proof}
		From \eqref{jopmf} and \eqref{mstfgcp}, we have 
		\begin{align*}
			p^{\alpha,\beta}(\bar{n},t)&=\mathbb{E}\Bigg(\sum_{\substack{\Omega(k_i,n_i)\\i=1,2,\dots,q}}\prod_{i=1}^{q}\prod_{j_i=1}^{k_i}\frac{(\lambda_{ij_i}D_\alpha(Y_\beta(t)))^{n_{ij_i}}}{n_{ij_i}!}e^{-\lambda_{ij_i}D_\alpha(Y_\beta(t))}\Bigg)\\
	&=\mathbb{E}\Bigg(\sum_{\substack{\Omega(k_i,n_i)\\i=1,2,\dots,q}}\prod_{i=1}^{q}\prod_{j_i=1}^{k_i}\frac{(-\lambda_{ij_i}\partial_{\lambda_{ij_i}})^{n_{ij_i}}}{n_{ij_i}!}e^{-\lambda_{ij_i} D_\alpha(Y_\beta(t))}\Bigg)\\
			&=\sum_{\substack{\Omega(k_i,n_i)\\i=1,2,\dots,q}}\bigg(\prod_{i=1}^{q}\prod_{j_i=1}^{k_i}\frac{(-\lambda_{ij_i}\partial_{\lambda_{ij_i}})^{n_{ij_i}}}{n_{ij_i}!}\bigg)\mathbb{E}\left(e^{-\lambda D_\alpha(Y_\beta(t))}\right)\\
			&=\sum_{\substack{\Omega(k_i,n_i)\\i=1,2,\dots,q}}\bigg(\prod_{i=1}^{q}\prod_{j_i=1}^{k_i}\frac{(-\lambda_{ij_i}\partial_{\lambda_{ij_i}})^{n_{ij_i}}}{n_{ij_i}!}\bigg)E_{\beta,1}(-\lambda^{\alpha}t^\beta),
		\end{align*}	
	where the second equality follows from Eq. (3.19) of Beghin and D'Ovidio (2014) and the last equality  follows from the Laplace transform of $\alpha$-stable and inverse $\beta$-stable subordinator (see Section $\ref{seci}$). This completes the proof.
	\end{proof}
\begin{remark}
For $k_1=k_2=\dots=k_q=1$, the Proposition \ref{prppmf} reduces to the corresponding result for multivariate space-time fractional Poisson process (see Beghin and Macci (2016), Proposition 3).
\end{remark}
In the following result, we obtain an explicit expression for the state probabilities of MGSTFCP.
\begin{theorem}\label{prppmfab}
The pmf $p^{\alpha,\beta}(\bar{n},t)$, $\bar{n}\ge\bar{0}$ is given by 
\begin{equation*}
p^{\alpha,\beta}(\bar{n},t)=\sum_{\substack{\Omega(k_i,n_i)\\i=1,2,\dots,q}}\sum_{r=0}^{\infty}\frac{(-\lambda^\alpha t^\beta)^r}{\Gamma(\beta r+1)}\frac{\Gamma(\alpha r+1)}{\Gamma(\alpha r-\sum_{i=1}^{q}\sum_{j_i=1}^{k_i}x_{ij_i}+1)}\prod_{i=1}^{q}\prod_{j_i=1}^{k_i}\frac{(-\lambda_{ij_i}/\lambda)^{x_{ij_i}}}{x_{ij_i}!},
\end{equation*} 
where $\Omega(k_i,n_i)=\{(x_{i1},x_{i2},\dots,x_{ik_i}):\sum_{j_i=1}^{k_i}j_ix_{ij_i}=n_i,\ x_{ij_i}\in\mathbb{N}_0\}$.
\end{theorem}
\begin{proof}
From \eqref{pgfmstfgcp}, we have 
{\small\begin{align}
G^{\alpha,\beta}(\bar{u},t)
&=\sum_{r=0}^{\infty}\frac{(-\lambda^\alpha t^\beta)^r}{\Gamma(\beta r+1)}\sum_{x\geq0}\binom{\alpha r}{x}\bigg(-\frac{1}{\lambda}\bigg)^x\bigg(\sum_{i=1}^{q}\sum_{j_i=1}^{k_i}\lambda_{ij_i}u_i^{j_i}\bigg)^x\nonumber\\
&=\sum_{r=0}^{\infty}\frac{(-\lambda^\alpha t^\beta)^r}{\Gamma(\beta r+1)}\sum_{x\geq0}\binom{\alpha r}{x}\bigg(-\frac{1}{\lambda}\bigg)^x\sum_{\substack{\sum_{i=1}^{q}x_i=x\\x_i\in\mathbb{N}_0}}x!\prod_{i=1}^{q}\frac{(\sum_{j_i=1}^{k_i}\lambda_{ij_i}u_i^{j_i})^{x_i}}{x_i!}\nonumber\\
&=\sum_{r=0}^{\infty}\frac{(-\lambda^\alpha t^\beta)^r}{\Gamma(\beta r+1)}\sum_{x\geq0}\binom{\alpha r}{x}\bigg(-\frac{1}{\lambda}\bigg)^x\sum_{\substack{\sum_{i=1}^{q}x_i=x\\x_i\in\mathbb{N}_0}}x!\prod_{i=1}^{q}\sum_{\substack{\sum_{j_i=1}^{k_i}x_{ij_i}=x_i\\
		x_{ij_i\in\mathbb{N}_0}}}\prod_{j_i=1}^{k_i}\frac{(\lambda_{ij_i}u_i^{j_i})^{x_{ij_i}}}{x_{ij_i}!}\nonumber\\
&=\sum_{r=0}^{\infty}\frac{(-\lambda^\alpha t^\beta)^r}{\Gamma(\beta r+1)}\sum_{x\geq0}\frac{\Gamma(\alpha r+1)}{\Gamma(\alpha r-x+1)}\bigg(-\frac{1}{\lambda}\bigg)^x\sum_{\substack{\sum_{i=1}^{q}x_i=x\\x_i\in\mathbb{N}_0}}\prod_{i=1}^{q}\sum_{\substack{\sum_{j_i=1}^{k_i}x_{ij_i}=x_i\\
		x_{ij_i\in\mathbb{N}_0}}}\prod_{j_i=1}^{k_i}\frac{(\lambda_{ij_i}u_i^{j_i})^{x_{ij_i}}}{x_{ij_i}!}\nonumber\\
&=\sum_{r=0}^{\infty}\frac{(-\lambda^\alpha t^\beta)^r}{\Gamma(\beta r+1)}\sum_{\substack{x_i\geq0\\1\leq i\leq q}}\frac{\Gamma(\alpha r+1)}{\Gamma(\alpha r-\sum_{i=1}^{q}x_i+1)}\prod_{i=1}^{q}\sum_{\substack{\sum_{j_i=1}^{k_i}x_{ij_i}=x_i\\
		x_{ij_i\in\mathbb{N}_0}}}\prod_{j_i=1}^{k_i}\frac{(-\lambda_{ij_i}/\lambda)^{x_{ij_i}}}{x_{ij_i}!}\nonumber\\
&=\sum_{r=0}^{\infty}\frac{(-\lambda^\alpha t^\beta)^r}{\Gamma(\beta r+1)}\sum_{\substack{n_i\ge0\\i=1,2,\dots,q}}\sum_{\substack{\Omega(k_i,n_i)\\i=1,2,\dots,q}}\frac{\Gamma(\alpha r+1)}{\Gamma(\alpha r-\sum_{i=1}^{q}\sum_{j_i=1}^{k_i}x_{ij_i}+1)}\prod_{i=1}^{q}u_i^{n_i}\prod_{j_i=1}^{k_i}\frac{(-\lambda_{ij_i}/\lambda)^{x_{ij_i}}}{x_{ij_i}!}\nonumber\\
&=\sum_{\substack{n_i\ge0\\i=1,2,\dots,q}}\Big(\prod_{i=1}^{q}u_i^{n_i}\Big)\sum_{\substack{\Omega(k_i,n_i)\\i=1,2,\dots,q}}\sum_{r=0}^{\infty}\frac{(-\lambda^\alpha t^\beta)^r\Gamma(\alpha r+1)}{\Gamma(\beta r+1)\Gamma(\alpha r-\sum_{i=1}^{q}\sum_{j_i=1}^{k_i}x_{ij_i}+1)}\prod_{i=1}^{q}\prod_{j_i=1}^{k_i}\frac{(-\lambda_{ij_i}/\lambda)^{x_{ij_i}}}{x_{ij_i}!}.\label{NNN}
\end{align}}
Also, we have
\begin{equation}\label{MMM}
G^{\alpha,\beta}(\bar{u},t)=\sum_{\substack{n_i\ge0\\i=1,2,\dots,q}}\Big(\prod_{i=1}^{q}u_i^{n_i}\Big)p^{\alpha,\beta}(\bar{n},t),\,\, |u_i|\le1.
\end{equation}
On comparing the coefficient of $u_1^{n_1}u_2^{n_2}\dots u_q^{n_q}$ over the range $\bar{n}\ge\bar{0}$ on both sides of \eqref{NNN} and \eqref{MMM}, we obtain the required result. 
\end{proof}
\begin{remark}
For $k_1=k_2=\dots=k_q=1$, the pmf $p^{\alpha,\beta}(\bar{n},t)$ reduces to the pmf of multivariate space-time fractional Poisson process (see Beghin and Macci (2016), Eq. (15)). For $q=1$, the pmf $p^{\alpha,\beta}(\bar{n},t)$ reduces to the pmf of generalized space-time fractional counting process (see Kataria {\it et al.} (2022), Eq. (6.12)). 
\end{remark}

\begin{remark}
The codifference of $\{M_i(D_\alpha(Y_\beta(t)))\}_{t\ge0}$ and $\{M_l(D_\alpha(Y_\beta(t)))\}_{t\ge0}$, $1\le i,l\le q$ is given by
{\scriptsize\begin{align}
	\tau(M_i(D_\alpha(Y_\beta(t))),M_l(D_\alpha(Y_\beta(t))))&=\ln E_{\beta,1}\bigg(-t^\beta\bigg(\sum_{j_i=1}^{k_i}\lambda_{ij_i}(1-e^{\omega j_i})+\sum_{j_l=1}^{k_l}\lambda_{lj_l}(1-e^{-\omega j_l})\bigg)^\alpha\bigg)\mathbb{I}_{\{i\ne l\}}\nonumber\\
	&\,\,\,-\ln E_{\beta,1}\bigg(-t^\beta\bigg(\sum_{j_i=1}^{k_i}\lambda_{ij_i}(1-e^{\omega j_i})\bigg)^\alpha\bigg)-\ln E_{\beta,1}\bigg(-t^\beta\bigg(\sum_{j_l=1}^{k_l}\lambda_{lj_l}(1-e^{-\omega j_l})\bigg)^\alpha\bigg).\label{roro}
\end{align}}
Its proof follows similar lines to that of Proposition \ref{prpcod}.

On substituting $k_1=k_2=\dots=k_q=1$ in \eqref{roro}, it reduces to the codifference of component processes in multivariate space-time fractional Poisson process (see Beghin and Macci (2016), Proposition 5).
\end{remark}
\begin{remark}
		For $1\le j_i\le k_i$, $i=1,2,\dots,q$, if $\lambda_{ij_i} = \lambda_i$ and $\lambda_{ij_i} = \lambda_i(1- \nu)\nu^{j_i-1}/(1-\nu^{k_i}), 0 \le \nu< 1$ then the results obtained above for MGSTFCP
	reduces to that for the multivariate version of space-time fractional Poisson process of order $k$ and P\'olya-Aeppli process of order $k$, respectively.
\end{remark}
\subsection{MGCP time-changed by tempered stable subordinator}
Next, we consider a multivariate tempered space fractional generalized counting process $\{\bar{Z}^{\alpha,\theta}(t)\}_{t\ge0}$. It is defined as follows:
\begin{equation}\label{temprepr}
	\bar{Z}^{\alpha,\theta}(t)=(M_1(D_{\alpha,\theta}(t)), M_2(D_{\alpha,\theta}(t)),\dots,M_q(D_{\alpha,\theta}(t))), \,\,0<\alpha<1,\,  \theta>0,
\end{equation}
where $\{D_{\alpha,\theta}(t)\}_{t\ge0}$ is a tempered $\alpha$-stable subordinator independent of $\{M_i(t)\}_{t\ge0}$, $i=1,2,\dots,q$. 

The transition probabilities of $\{\bar{Z}^{\alpha,\theta}(t)\}_{t\ge0}$ are given by
{\scriptsize\begin{align}\label{temptrans}
\mathrm{Pr}\{\bar{Z}^{\alpha,\theta}(t+h)=\bar{n}+\bar{l}|\bar{Z}^{\alpha,\theta}(t)=\bar{n}\}&=\begin{cases}
-h\displaystyle\sum_{\substack{\Omega(k_i,l_i)\\i=1,2,\dots,q}}\frac{(\lambda+\theta)^\alpha\Gamma(1+\alpha)}{\Gamma(\alpha-\sum_{i=1}^{q}\sum_{j_i=1}^{k_i}l_{ij_i}+1)}\prod_{i=1}^{q}\prod_{j_i=1}^{k_i}\frac{(-\lambda_{ij_i}/(\lambda+\theta))^{l_{ij_i}}}{l_{ij_i}!}+o(h),\ \bar{l}\succ\bar{0},\vspace{.2cm}\\
1-h((\lambda+\theta)^\alpha-\theta^\alpha) +o(h),\ \bar{l}=\bar{0},
\end{cases}
\end{align}}
where $o(h)\to0$ as $h\to0$ and $\Omega(k_i,l_i)=\{(l_{i1},l_{i2},\dots,l_{ik_i}):\sum_{j_i=1}^{k_i}j_il_{ij_i}=l_i,\,\, l_{ij_i}\in\mathbb{N}_0\}$.
\begin{remark}
For $q=1$, the transition probabilities in \eqref{temptrans} reduces to that of GCP time-changed by tempered stable subordinator (see Kataria \textit{et al.} (2022), Eq. 4.16).
\end{remark}
  The pgf $G_{\bar{Z}}(\bar{u},t)=\mathbb{E}\big(\prod_{i=1}^{q}u_i^{M_i(D_{\alpha,\theta}(t))} \big)$, $\bar{u}\in[0,1]^q$ of $\{\bar{Z}^{\alpha,\theta}(t)\}_{t\ge0}$ is given by
\begin{align}
	G_{\bar{Z}}(\bar{u},t)&=\mathbb{E}\Big(\mathbb{E}\Big(\prod_{i=1}^{q}u_i^{M_i(D_{\alpha,\theta}(t))} \Big|\, D_{\alpha,\theta}(t)\Big)\Big)\nonumber\\
	&=\mathbb{E}\bigg(\exp{\bigg(-D_{\alpha,\theta}(t)\sum_{i=1}^{q}\sum_{j_i=1}^{k_i}\lambda_{ij_i}(1-u_i^{j_i})\bigg)}\bigg)\nonumber\\
	&=\exp\bigg(-t\bigg(\bigg(\sum_{i=1}^{q}\sum_{j_i=1}^{k_i}\lambda_{ij_i}(1-u_i^{j_i})+\theta\bigg)^\alpha-\theta^\alpha\bigg)\bigg),\label{mnmn}
\end{align}
where we have used \eqref{temppp} in the last step to obtain the pgf given in \eqref{mnmn}. It solves the following differential equation:
\begin{equation}\label{pgftempde}
	\frac{\partial}{\partial t}G_{\bar{Z}}(\bar{u},t)=-\bigg(\bigg(\sum_{i=1}^{q}\sum_{j_i=1}^{k_i}\lambda_{ij_i}(1-u_i^{j_i})+\theta\bigg)^\alpha-\theta^\alpha\bigg)G_{\bar{Z}}(\bar{u},t), \,\, G_{\bar{Z}}(\bar{u},0)=1.
\end{equation}
\begin{remark}
	On substituting $q=2$ and $k_i=1$ in \eqref{mnmn}, we obtain the pgf of bivariate tempered space-factional Poisson process	(see Soni \textit{et al.} (2024)).
\end{remark}
In the following result, we give the system of governing differential equations for the state probabilities of $\{\bar{Z}^{\alpha,\theta}(t)\}_{t\ge0}$.
\begin{proposition}\label{prpalphatheta}
	The state probabilities $p_{\bar{Z}}(\bar{n},t)=\mathrm{Pr}\{\bar{Z}^{\alpha,\theta}(t)=\bar{n}\}$, $\bar{n}\ge\bar{0}$ solve
{\small	\begin{align*}
		\frac{\mathrm{d}}{\mathrm{d}t}p_{\bar{Z}}(\bar{n},t)&=\theta^\alpha p_{\bar{Z}}(\bar{n},t)-(\lambda+\theta)^\alpha\Gamma(\alpha+1) \sum_{\substack{m_i\ge0\\i=1,2,\dots,q}}\sum_{\substack{\Omega(k_i,m_i)\\ i=1,2,\dots, q}}\frac{p_{\bar{Z}}(\bar{n}-\bar{m},t)}{\Gamma(\alpha+1-\sum_{i=1}^{q}\sum_{j_i=1}^{k_i}x_{ij_i})}\\
		&\hspace*{7cm}\cdot\prod_{i=1}^{q}\prod_{j_i=1}^{k_i}\frac{(-\lambda_{ij_i}/(\lambda+\theta))^{x_{ij_i}}}{x_{ij_i}!}
	\end{align*}}
	with $p_{\bar{Z}}(\bar{n},0)=\mathbb{I}_{\{\bar{n}=\bar{0}\}}$. Here,  $\Omega(k_i,m_i)=\{(x_{i1},x_{i2},\dots,x_{ik_i}):\sum_{j_i=1}^{k_i}j_ix_{ij_i}=m_i,\,\, x_{ij_i}\in\mathbb{N}_0\}$.
\end{proposition}
\begin{proof}
On using \eqref{temptrans}, the proof follows similar lines to that of Proposition \ref{prppp}. Hence, it is omitted.
An alternate proof of this result is given in Appendix A2.
\end{proof}
\begin{theorem}\label{thmalphatheta}
	The state probabilities of $\{\bar{Z}^{\alpha,\theta}(t)\}_{t\ge0}$ are given by
	{\small	\begin{align*}
			p_{\bar{Z}}(\bar{n},t)&=e^{t\theta^\alpha}\sum_{\substack{\Omega(k_i,n_i)\\ i=1,2,\dots,q}}\sum_{r=0}^{\infty}\frac{(- t(\lambda+\theta)^{\alpha })^r}{r!}\frac{\Gamma(\alpha r+1)}{\Gamma(\alpha r+1-\sum_{i=1}^{q}\sum_{J_i=1}^{k_i}x_{ij_i})}\prod_{i=1}^{q}\prod_{j_i=1}^{k_i}\frac{(-\lambda_{ij_i}/(\lambda+\theta))^{x_{ij_i}}}{x_{ij_i}!},\, \, \bar{n}\ge\bar{0},
	\end{align*}}
	where  $\lambda=\sum_{i=1}^{q}\sum_{j_i=1}^{k_i}\lambda_{ij_i}$ and $\Omega(k_i,n_i)=\{(x_{i1},x_{i2},\dots,x_{ik_i}):\sum_{j_i=1}^{k_i}j_ix_{ij_i}=n_i,\,x_{ij_i}\in\mathbb{N}_0 \}$.
\end{theorem}
\begin{proof}
See Appendix A3.
\end{proof}	
\begin{corollary}
The pmf of $\{\bar{Z}^{\alpha,\theta}(t)\}_{t\ge0}$ has the following equivalent representation:
{\small\begin{equation*}
p_{\bar{Z}}(\bar{n},t)=\sum_{\substack{\Omega(k_i,n_i)\\i=1,2,\dots,q}}
			{}_1\Psi_1\left[\begin{matrix}
				(1,\alpha)_{1,1}\\
				\Big(1-\sum_{i=1}^{q}\sum_{j_i=1}^{k_i}x_{ij_i}, \alpha\Big)_{1,1}
			\end{matrix}\Bigg| -(\lambda+\theta)^\alpha t\right]\prod_{i=1}^{q}\prod_{j_i=1}^{k_i}\frac{(-\lambda_{ij_i}/(\lambda+\theta))^{x_{ij_i}}}{x_{ij_i}!},\, \bar{n}\ge\bar{0},
\end{equation*}}
where ${}_1\Psi_1(\cdot)$ is the generalized Wright function whose definition is given in \eqref{wrightfn}.
	\end{corollary}

Let $q(\cdot,t)$ be the density of $\{D_{\alpha,\theta}(t)\}_{t\ge0}$ and $\delta_0$ denotes the Dirac delta function. The following relation will be used to obtain the next result (see Beghin (2015), Eq. (15)): 
\begin{equation}\label{llml}
	\frac{\partial}{\partial x}q(x,t)=-\theta q(x,t)+\bigg(\theta^\alpha-\frac{\partial}{\partial t}\bigg)^{1/\alpha}q(x,t)
\end{equation}
with $q(x,0)=0$ and $q(0,t)=\delta_0(x)$.

\begin{proposition}
The state probabilities of $\{\bar{Z}^{\alpha,\theta}(t)\}_{t\ge0}$ satisfy
	\begin{equation*}
		\bigg(\theta^\alpha-\frac{\partial}{\partial t}\bigg)^{1/\alpha}p_{\bar{Z}}(\bar{n},t)=(\theta+\lambda)p_{\bar{Z}}(\bar{n},t)-\sum_{i=1}^{q}\sum_{j_i=1}^{k_i}\lambda_{ij_i}p_{\bar{Z}}(\bar{n}-\bar{\epsilon}^{j_i}_i,t),\,\,\,\bar{n}\ge\bar{0}.
	\end{equation*}
\end{proposition}
\begin{proof}
	From \eqref{temprepr}, we have
	\begin{equation}\label{abnb}
		p_{\bar{Z}}(\bar{n},t)=\int_{0}^{\infty}p(\bar{n},x)q(x,t)\,\mathrm{d}x.
	\end{equation}
	On using \eqref{llml} and the fact that $\lim_{x\to0}q(x,t)=\lim_{x\to\infty}q(x,t)=0$,  \eqref{abnb} becomes
	\begin{align*}
		\bigg(\theta^\alpha-\frac{\partial}{\partial t}\bigg)^{1/\alpha}p_{\bar{Z}}(\bar{n},t)&=\int_{0}^{\infty}p(\bar{n},x)\bigg(\theta q(x,t)+\frac{\partial}{\partial x}q(x,t)\bigg)\mathrm{d}x\\
		&=\theta p_{\bar{Z}}(\bar{n},t)-\int_{0}^{\infty}q(x,t)\frac{\partial}{\partial x}p(\bar{n},x)\, \mathrm{d}x\\
		&=\theta p_{\bar{Z}}(\bar{n},t)-\int_{0}^{\infty}q(x,t)\Big(-\lambda p(\bar{n},x)+\sum_{i=1}^{q}\sum_{j_i=1}^{k_i}\lambda_{ij_i}p(\bar{n}-\bar{\epsilon}^{j_i}_i,x)\Big)\,\mathrm{d}x\\
		&=(\theta+\lambda)p_{\bar{Z}}(\bar{n},t)-\sum_{i=1}^{q}\sum_{j_i=1}^{k_i}\lambda_{ij_i}\int_{0}^{\infty}q(x,t)p(\bar{n}-\bar{\epsilon}^{j_i}_i,x)\,\mathrm{d}x,
	\end{align*}
	which gives the required result.
\end{proof}
 The proof of following result is similar to that of Proposition $\ref{levymeasalpha}$. Hence, it is omitted. Here, we need to work with the L\'evy measure of tempered stable subordinator, that is, $\mu_{D_{\alpha,\theta}}(\mathrm{d}s)=\alpha s^{-\alpha-1}e^{-\theta s}$ $\mathrm{d}s/\Gamma(1-\alpha)$, $\theta>0$, $0<\alpha<1$.
\begin{proposition}\label{prpppo}
The L\'evy measure of  $\{\bar{Z}^{\alpha,\theta}(t)\}_{t\ge0}$ is given by
{\small\begin{align*}
\Pi^{\bar{Z}}(A_1\times A_2\times\dots\times A_q)
&=\frac{\alpha(\lambda+\theta)^\alpha }{\Gamma(1-\alpha)}\sum_{\bar{n}\succ\bar{0}}\sum_{\substack{\Omega(k_i,n_i)\\i=1,2,\dots,q}}\Gamma\Big(\sum_{i=1}^{q}\sum_{j_i=1}^{k_i}n_{ij_i}-\alpha\Big)\prod_{i=1}^{q}\prod_{j_i=1}^{k_i}\frac{(\lambda_{ij_i}/(\lambda+\theta))^{n_{ij_i}}}{n_{ij_i}!}\mathbb{I}_{\{n_i\in A_i\}},
\end{align*}}
where $\Omega(k_i,n_i)=\{(n_{i1},n_{i2},\dots,n_{ik_i}):\sum_{j_i=1}^{k_i}j_in_{ij_i}=n_i,\,n_{ij_i}\in\mathbb{N}_0\}$ and $\lambda=\sum_{i=1}^{q}\sum_{x_l=1}^{k_i}\lambda_{ij_i}$.
\end{proposition}

\begin{remark}
For any $i=1,2\dots,q$ and $l=1,2\dots,q$, the covariance and the codifference of $\{M_i(D_{\alpha,\theta}(t))\}_{t\ge0}$ and $\{M_l(D_{\alpha,\theta}(t))\}_{t\ge0}$ are given by 
\begin{equation*}
	\operatorname{Cov}(M_i(D_{\alpha,\theta}(t)),M_l(D_{\alpha,\theta}(t))=\sum_{j_i=1}^{k_i}j_i^2\lambda_{ij_i}^2\alpha(1-\alpha)t\theta^{\alpha-2}\mathbb{I}_{\{i=l\}}+\sum_{j_i=1}^{k_i}\sum_{j_l=1}^{k_l}j_i\lambda_{ij_i}j_l\lambda_{lj_l}\alpha t\theta^{\alpha-2}(\alpha t+1-\alpha)
\end{equation*}
and
\begin{align*}
	\tau(M_i(D_{\alpha,\theta}(t)),M_l(D_{\alpha,\theta}(t))&=-t\bigg(\bigg(\sum_{j_i=1}^{k_i}\lambda_{ij_i}(1-e^{\omega j_i})+\sum_{j_l=1}^{k_l}\lambda_{lj_l}(1-e^{-\omega j_l})+\theta\bigg)^\alpha-\theta^\alpha\bigg)\mathbb{I}_{\{i\ne l\}}\\
	&\,\,\,+t\bigg(\bigg(\sum_{j_i=1}^{k_i}\lambda_{ij_i}(1-e^{\omega j_i})+\theta\bigg)^\alpha+\bigg(\sum_{j_l=1}^{k_l}\lambda_{lj_l}(1-e^{-\omega j_l})+\theta\bigg)^\alpha-2\theta^\alpha\bigg),
\end{align*}
 respectively, where $\omega=\sqrt{-1}$.
\end{remark}
\subsection{MGCP time-changed by gamma subordinator}
Here, we consider MGCP time-changed with an independent gamma subordinator $\{G_{a,b}(t)\}_{t\ge0}$ with parameters $a>0,b>0$. We denote it by $\{\bar{\mathcal{W}}^{a,b}(t)\}_{t\ge0}$. It is defined as follows:
\begin{equation}\label{Mgammadef}
	\bar{\mathcal{W}}^{a,b}(t)\coloneqq(M_1(G_{a,b}(t)), M_2(G_{a,b}(t)),\dots,M_q(G_{a,b}(t))), \,\,t\ge0,
\end{equation}

Note that the component processes $\{M_i(G_{a,b}(t))\}_{t\ge0}$, $i=1,2,\dots,q$ are conditionally independent given $\{G_{a,b}(t)\}_{t\ge0}$. 

In a small time interval of length $h$ such that $o(h)/h\to0$ as $h\to0$, the transition probabilities of $\{\bar{\mathcal{W}}^{a,b}(t)\}_{t\ge0}$ are given by
{\footnotesize\begin{equation}\label{transgamma}
	\mathrm{Pr}\{\bar{\mathcal{W}}^{a,b}(t+h)=\bar{n}+\bar{l}|\bar{\mathcal{W}}^{a,b}(t)=\bar{n}\}=\begin{cases}
		hb \displaystyle\sum_{\substack{\Omega(k_i,l_i)\\i=1,2,\dots,q}}\Big(\sum_{i=1}^{q}\sum_{j_i=1}^{k_i}l_{ij_i}-1\Big)!\prod_{i=1}^{q}\prod_{j_i=1}^{k_i}\frac{(\lambda_{ij_i}/(\lambda+a))^{l_{ij_i}}}{l_{ij_i}!}+o(h),\ \bar{l}\succ\bar{0},\vspace{.2cm}\\
		1-hb\log(1+\lambda/a) +o(h),\ \bar{l}=\bar{0},
		\end{cases}
\end{equation}}
where   $\Omega(k_i,l_i)=\{(l_{i1},l_{i2},\dots,l_{ik_i}):\sum_{j_i=1}^{k_i}j_il_{ij_i}=l_i,\,\, l_{ij_i}\in\mathbb{N}_0\}$ and $\lambda=\sum_{i=1}^{q}\sum_{j_i=1}^{k_i}\lambda_{ij_i}$.
\begin{remark}
On substituting $q=1$ in \eqref{transgamma}, we get the transition probabilities given in Eq. (4.11), Kataria and Khandakar (2022b).
\end{remark}
The pgf of $\{\bar{\mathcal{W}}^{a,b}(t)\}_{t\ge0}$ can be obtained as follows:
\begin{align*}
G_{\bar{\mathcal{W}}}(\bar{u},t)&=\mathbb{E}\Big(\mathbb{E}\Big(\prod_{i=1}^{q}u_i^{M_i(G_{a,b}(t))}\big|G_{a,b}(t)\Big)\Big)\\
&=\mathbb{E}\bigg(\exp\bigg(-G_{a,b}(t)\sum_{i=1}^{q}\sum_{j_i=1}^{k_i}\lambda_{ij_i}(1-u_i^{j_i})\bigg)\bigg),\,\,\,\text{(using \eqref{pgffmgcp})}\\
&=\bigg(1+\frac{1}{a}\sum_{i=1}^{q}\sum_{j_i=1}^{k_i}\lambda_{ij_i}(1-u_i^{j_i})\bigg)^{-bt},
\end{align*}
where the last step follows on using \eqref{gammaberfn}.

Next, on using \eqref{erkq} the proof of the following result is similar to that of Proposition 5.2 of Kataria \textit{et al.} (2022). Hence, it is omitted.
\begin{proposition}
The pmf $p_{\bar{\mathcal{W}}}(\bar{n},t)=\mathrm{Pr}\{\bar{\mathcal{W}}(t)=\bar{\bar{n}}\}$, $\bar{n}\ge\bar{0}$ satisfies
\begin{equation*}
e^{-\partial_t/a}\,p_{\bar{\mathcal{W}}}(\bar{n},t)=(1+\lambda/b)p_{\bar{\mathcal{W}}}(\bar{n},t)-\frac{1}{b}\sum_{i=1}^{q}\sum_{j_i=1}^{k_i}\lambda_{ij_i}p_{\bar{\mathcal{W}}}(\bar{n}-\bar{\epsilon}^{j_i}_i,t)
\end{equation*}
with initial condition $p_{\bar{\mathcal{W}}}(\bar{n},0)=\mathbb{I}_{\{\bar{n}=\bar{0}\}}$. Here, $e^{-\partial_t/a}$ is the shift operator defined in \eqref{gammaoper}.
\end{proposition}
 On using the transition probabilities of MGCP time-changed by gamma subordinator given in \eqref{transgamma} and following along the similar lines to the proof of Proposition \ref{prppp}, we get the next result. 
\begin{proposition}The state probabilities $p_{\bar{\mathcal{W}}}(\bar{n},t)=\bar{n}\}$, $\bar{n}\ge\bar{0}$ solve
{\footnotesize	\begin{align*}
		\frac{\mathrm{d}}{\mathrm{d}t}p_{\bar{\mathcal{W}}}(\bar{n},t)&=-b\log(1+\lambda/a) p_{\bar{\mathcal{W}}}(\bar{n},t)+b\sum_{\bar{l}\succ\bar{0}}p_{\bar{\mathcal{W}}}(\bar{n}-\bar{l},t)\displaystyle\sum_{\substack{\Omega(k_i,l_i)\\i=1,2,\dots,q}}\Big(\sum_{i=1}^{q}\sum_{j_i=1}^{k_i}l_{ij_i}-1\Big)!\prod_{i=1}^{q}\prod_{j_i=1}^{k_i}\frac{(\lambda_{ij_i}/(\lambda+a))^{l_{ij_i}}}{l_{ij_i}!}
\end{align*}}
with $p_{\bar{\mathcal{W}}}(\bar{n},0)=\mathbb{I}_{\{\bar{n}=\bar{0}\}}$. Here,  $\Omega(k_i,l_i)=\{(l_{i1},l_{i2},\dots,l_{ik_i}):\sum_{j_i=1}^{k_i}j_il_{ij_i}=l_i,\,\, l_{ij_i}\in\mathbb{N}_0\}$.
\end{proposition}

\begin{theorem}
The state probabilities of $\{\bar{\mathcal{W}}^{a,b}(t)\}_{t\ge0}$, $a>0,b>0,$ are given by
\begin{equation*}
p_{\bar{\mathcal{W}}}(\bar{n},t)=\frac{(a/(\lambda+a))^{bt}}{\Gamma(bt)}\sum_{\substack{\Omega(k_i,n_i)\\i=1,2,\dots,q}}\Gamma\bigg(\sum_{i=1}^{q}\sum_{j_i=1}^{k_i}n_{ij_i}+bt\bigg)\prod_{i=1}^{q}\prod_{j_i=1}^{k_i}\frac{(\lambda_{ij_i}/(\lambda+a))^{n_{ij_i}}}{n_{ij_i}!},\,\,\bar{n}\ge\bar{0}, \\
\end{equation*}
where $\Omega(k_i,n_i)=\{(n_{i1},n_{i2},\dots,n_{ik_i}):\sum_{j_i=1}^{k_i}j_in_{ij_i}=n_i,\,\, n_{ij_i}\in\mathbb{N}_0\}$ and $\lambda=\sum_{i=1}^{q}\sum_{j_i=1}^{k_i}\lambda_{ij_i}$.
\end{theorem}
\begin{proof}
From \eqref{Mgammadef}, we can write
\begin{align*}
p_{\bar{\mathcal{W}}}(\bar{n},t)&=\int_{0}^{\infty}p(\bar{n},x)h(x,t)\,\mathrm{d}x\\
&=\frac{a^{bt}}{\Gamma(bt)}\sum_{\substack{\Omega(k_i,n_i)\\i=1,2,\dots,q}}\bigg(\prod_{i=1}^{q}\prod_{j_i=1}^{k_i}\frac{\lambda_{ij_i}^{n_{ij_i}}}{n_{ij_i}!}\bigg)\int_{0}^{\infty}e^{-(\lambda+a)x}x^{\sum_{i=1}^{q}\sum_{j_i=1}^{k_i}n_{ij_i}+bt-1}\,\mathrm{d}x\\
&=\frac{(a/(\lambda+a))^{bt}}{\Gamma(bt)}\sum_{\substack{\Omega(k_i,n_i)\\i=1,2,\dots,q}}\Gamma\bigg(\sum_{i=1}^{q}\sum_{j_i=1}^{k_i}n_{ij_i}+bt\bigg)\prod_{i=1}^{q}\prod_{j_i=1}^{k_i}\frac{(\lambda_{ij_i}/(\lambda+a))^{n_{ij_i}}}{n_{ij_i}!},
\end{align*}
where $h(x,t)$ is the density of gamma subordinator and second step follows on using \eqref{gammadensity} and \eqref{jopmf}. This completes the proof.
\end{proof}
The proof of next result follows similar lines to that of Proposition $\ref{levymeasalpha}$. Here, we need to use the L\'evy measure of gamma subordinator, that is, $\mu_{G_{a,b}}(\mathrm{d}s)=bs^{-1}e^{-as}\mathrm{d}s$, $a>0$, $b>0$.
\begin{proposition}
The L\'evy measure of $\{\bar{\mathcal{W}}^{a,b}(t)\}_{t\ge0}$ is given by
{\small\begin{align}\label{lemeasgammaa}
\Pi^{\bar{\mathcal{W}}}(A_1\times A_2\times\dots\times A_q)
&=b \sum_{\bar{n}\succ\bar{0}}\sum_{\substack{\Omega(k_i,n_i)\\i=1,2,\dots,q}}\Gamma\bigg(\sum_{i=1}^{q}\sum_{j_i=1}^{k_i}n_{ij_i}\bigg)\prod_{i=1}^{q}\prod_{j_i=1}^{k_i}\frac{(\lambda_{ij_i}/(\lambda+a))^{n_{ij_i}}}{n_{ij_i}!}\mathbb{I}_{\{n_i\in A_i\}},
\end{align}}
where $\Omega(k_i,n_i)=\{(n_{i1},n_{i2},\dots,n_{ik_i}):\sum_{j_i=1}^{k_i}j_in_{ij_i}=n_i,\,n_{ij_i}\in\mathbb{N}_0\}$ and $\lambda=\sum_{i=1}^{q}\sum_{x_l=1}^{k_i}\lambda_{ij_i}$.
\end{proposition}

\begin{remark}
For $q=1$, the L\'evy measure \eqref{lemeasgammaa} reduces to that of GCP time-changed by a gamma subordinator (see Kataria and Khandakar (2022b), Eq. 4.12).
\end{remark}
\subsection{MGCP time-changed by inverse Gaussian subordinator}
For $i=1,2,\dots,q$, let the GCPs $\{M_i(t)\}_{t\ge0}$ be independent of an inverse Gaussian subordinator $\{I_{\delta,\gamma}(t)\}_{t\ge0}$, $\delta>0,\gamma>0$. Here, we consider the MGCP time-changed by inverse Gaussian subordinator. It is defined as follows:
\begin{equation}\label{invgaudef}
	\bar{\mathcal{X}}^{\delta,\gamma}(t)\coloneqq \bar{M}(I_{\delta,\gamma}(t)), \, t\ge0.
\end{equation} 
\begin{remark}
For $q=1$, the process defined in \eqref{invgaudef} reduces to a time-changed variant of the GCP (see Kataria and Khandakar (2022b), Section 4.2.3).
\end{remark}
From Remark $\ref{remk}$, we observe that the component processes in $\{\bar{\mathcal{X}}^{\delta,\gamma}(t)\}_{t\ge0}$ are conditionally independent given $\{I_{\delta,\gamma}(t)\}_{t\ge0}$.

The transition probabilities of $\{\bar{\mathcal{X}}^{\delta,\gamma}(t)\}_{t\ge0}$ are given by
{\footnotesize\begin{align}\label{transinvgauss}
\mathrm{Pr}\{\bar{\mathcal{X}}^{\delta,\gamma}(t+h)&=\bar{n}+\bar{l}|\bar{\mathcal{X}}^{\delta,\gamma}(t)=\bar{n}\}\nonumber\\
&=\begin{cases}
	h\delta \frac{\sqrt{2\lambda+\gamma^2}}{2\sqrt{\pi}}\displaystyle\sum_{\substack{\Omega(k_i,l_i)\\i=1,2,\dots,q}}\Gamma\bigg(\sum_{i=1}^{q}\sum_{j_i=1}^{k_i}l_{ij_i}-1/2\bigg)\prod_{i=1}^{q}\prod_{j_i=1}^{k_i}\bigg(\frac{2\lambda_{ij_i}}{2\lambda+\gamma^2}\bigg)^{l_{ij_i}}\frac{1}{l_{ij_i}!}+o(h),\,\, \bar{l}\succ\bar{0},\\
	1-\delta h(\sqrt{2\lambda+\gamma^2}-\gamma)+o(h),\,\, \bar{l}=\bar{0},
\end{cases}
\end{align}}
where $o(h)/h\to0$ as $h\to0$. Here,  $\Omega(k_i,l_i)=\{(l_{i1},l_{i2},\dots,l_{ik_i}):\sum_{j_i=1}^{k_i}j_il_{ij_i}=l_i,\,\, l_{ij_i}\in\mathbb{N}_0\}$ and $\lambda=\sum_{i=1}^{q}\sum_{j_i=1}^{k_i}\lambda_{ij_i}$.
\begin{remark}
For $q=1$, the transition probabilities in \eqref{transinvgauss} reduces to that of the GCP time-changed by an inverse Gaussian subordinator (see Kataria and Khandakar (2022b), Section 4.2.3).
\end{remark}
On using \eqref{lapinv} and the relation in \eqref{invgaudef}, the pgf of $\{\bar{\mathcal{X}}^{\delta,\gamma}(t)\}_{t\ge0}$ is given by
\begin{equation}\label{pgfinvgau}
G_{\bar{\mathcal{X}}}(\bar{u},t)=\exp\left(-\delta t\left(\sqrt{2\sum_{i=1}^{q}\sum_{j_i=1}^{k_i}\lambda_{ij_i}(1-u_i^{j_i})}-\gamma\right)\right),\,\,\, |u_i|\le1.
\end{equation}
So, 
\begin{equation*}
\frac{\partial}{\partial t}G_{\bar{\mathcal{X}}}(\bar{u},t)=-\delta \left(\sqrt{2\sum_{i=1}^{q}\sum_{j_i=1}^{k_i}\lambda_{ij_i}(1-u_i^{j_i})}-\gamma\right)G_{\bar{\mathcal{X}}}(\bar{u},t),\,\,\, G_{\bar{\mathcal{X}}}(\bar{u},0)=1.
\end{equation*}

The proof of next result follows similar lines to that of Proposition \ref{prppp}. Hence, it is omitted. Here, we need to use the transition probabilities of MGCP time-changed by inverse Gaussian subordinator given in \eqref{transinvgauss}.
	
	\begin{proposition} The state probabilities $p_{\bar{\mathcal{X}}}(\bar{n},t)=\mathrm{Pr}\{\bar{\mathcal{X}}^{\delta,\gamma}(t)=\bar{n}\}$, $\bar{n}\ge\bar{0}$ satisfy
		\begin{align*}
				\frac{\mathrm{d}}{\mathrm{d}t}p_{\bar{\mathcal{X}}}(\bar{n},t)&=-\delta (\sqrt{2\lambda+\gamma^2}-\gamma) p_{\bar{\mathcal{X}}}(\bar{n},t)+\delta \frac{\sqrt{2\lambda+\gamma^2}}{2\sqrt{\pi}}\sum_{\bar{l}\succ\bar{0}}p_{\bar{\mathcal{X}}}(\bar{n}-\bar{l},t)\\
				&\hspace{4cm}\cdot\displaystyle\sum_{\substack{\Omega(k_i,l_i)\\i=1,2,\dots,q}}\Gamma\bigg(\sum_{i=1}^{q}\sum_{j_i=1}^{k_i}l_{ij_i}-1/2\bigg)\prod_{i=1}^{q}\prod_{j_i=1}^{k_i}\bigg(\frac{2\lambda_{ij_i}}{2\lambda+\gamma^2}\bigg)^{l_{ij_i}}\frac{1}{l_{ij_i}!}
		\end{align*}
		with $p_{\bar{\mathcal{X}}}(\bar{n},0)=\mathbb{I}_{\{\bar{n}=\bar{0}\}}$. Here,  $\Omega(k_i,l_i)=\{(l_{i1},l_{i2},\dots,l_{ik_i}):\sum_{j_i=1}^{k_i}j_il_{ij_i}=l_i,\,\, l_{ij_i}\in\mathbb{N}_0\}$.
\end{proposition}

\begin{theorem}\label{thmappen}
The pmf  of $\{\bar{\mathcal{X}}^{\delta,\gamma}(t)\}_{t\ge0}$ is given by
\begin{equation*}
p_{\bar{\mathcal{X}}}(\bar{n},t)=e^{\delta \gamma t}\sum_{r=0}^{\infty}\frac{(-\sqrt{2\lambda}\delta t)^r}{r!}\sum_{\substack{\Omega(k_i,n_i)\\ i=1,2,\dots,q}}\Big(\frac{r}{2}\Big)_{\sum_{i=1}^{q}\sum_{j_i=1}^{k_i}x_{ij_i}}\prod_{i=1}^{q}\prod_{j_i=1}^{k_i}\frac{(-\lambda_{ij_i}/\lambda)^{x_{ij_i}}}{x_{ij_i}!},\, \bar{n}\ge\bar{0},
\end{equation*} 
where $(r/2)_x=(r/2)(r/2-1)\dots(r/2-x+1)$, $\lambda=\sum_{i=1}^{q}\sum_{j_i=1}^{k_i}\lambda_{ij_i}$ and $\Omega(k_i,n_i)=\{(x_{i1},x_{i2},\dots$, $x_{ik_i}):\sum_{j_i=1}^{k_i}j_ix_{ij_i}=n_i,\,x_{ij_i}\in\mathbb{N}_0\}$. 
\begin{proof}
See Appendix A4.
\end{proof}
\end{theorem}

	To prove the next proposition, we will use the following result (see Vellaisamy and Kumar (2018), Eq. (3.3)):  
	\begin{equation}\label{vellinv}
		\frac{\partial^2}{\partial t^2}g(x,t)-2\delta \gamma \frac{\partial}{\partial t}g(x,t)=2\delta^2\frac{\partial}{\partial x}g(x,t),
	\end{equation}
	such that $\lim_{x\to\infty}g(x,t)=\lim_{x\to0}g(x,t)=0$ where $g(x,t)$ is the density of inverse Gaussian subordinator given in \eqref{gxtinv}.
\begin{proposition}
	The state probabilities $p_{\bar{\mathcal{X}}}(\bar{n},t)$, $\bar{n}\ge\bar{0}$ solve
	\begin{equation*}
		\Big(\frac{\partial^2}{\partial t^2}-2\delta \gamma \frac{\partial}{\partial t}\Big)p_{\bar{\mathcal{X}}}(\bar{n},t)=2\delta^2\lambda p_{\bar{\mathcal{X}}}(\bar{n},t)-2\delta^2\sum_{i=1}^{q}\sum_{j_i=1}^{k_i}\lambda_{ij_i}p_{\bar{\mathcal{X}}}(\bar{n}-\bar{\epsilon}^{j_i}_i,t),
	\end{equation*}
	where $\lambda=\sum_{i=1}^{q}\sum_{j_i=1}^{k_i}\lambda_{ij_i}$ and $\bar{\epsilon}_i^{j_i}\in \mathbb{N}_{0}^q$ is a $q$-tuple vector whose $i^{th}$ entry is $j_i$ and other entries are zero. 
	\begin{proof}
		On using \eqref{invgaudef}, we can write
		\begin{equation}\label{aruninv}
			p_{\bar{\mathcal{X}}}(\bar{n},t)=\int_{0}^{	\infty}p(\bar{n},x)g(x,t)\,\mathrm{d}x.
		\end{equation}
		So, 
		\begin{align*}
			\Big(\frac{\partial^2}{\partial t^2}-2\delta \gamma \frac{\partial}{\partial t}\Big)p_{\bar{\mathcal{X}}}(\bar{n},t)&=\int_{0}^{\infty}\Big(\frac{\partial^2}{\partial t^2}-2\delta \gamma \frac{\partial}{\partial t}\Big)g(x,t)\,\mathrm{d}x\\
			&=2\delta^2\int_{0}^{\infty}p(\bar{n},x)\frac{\partial}{\partial x}g(x,t)\, \mathrm{d}x,\,\,\, \text{(using \eqref{vellinv})}\\
			&=-2\delta^2\int_{0}^{\infty}g(x,t)\frac{\mathrm{d}}{\mathrm{d} x}p(\bar{n},x)\, \mathrm{d}x\\
			&=-2\delta^2\int_{0}^{\infty}g(x,t)\bigg(-\lambda p(\bar{n},x)+\sum_{i=1}^{q}\sum_{j_i=1}^{k_i}\lambda_{ij_i}p(\bar{n}-\bar{\epsilon}^{j_i}_i,x)\bigg)\\
			&=2\delta^2\lambda p_{\bar{\mathcal{X}}}(\bar{n},t)-2\delta^2\sum_{i=1}^{q}\sum_{j_i=1}^{k_i}\lambda_{ij_i}p_{\bar{\mathcal{X}}}(\bar{n}-\bar{\epsilon}^{j_i}_i,t),
		\end{align*}
		where the penultimate and the last step follows on using \eqref{propmgcp} and \eqref{aruninv}, respectively. This completes the proof.
	\end{proof}
\end{proposition}
The next result can be proved along the similar lines to that of Proposition $\ref{levymeasalpha}$. Here, we have to use the L\'evy measure of inverse Gaussian subordinator, that is, $\mu_{I_{\delta,\gamma}}(\mathrm{d}s)=\delta e^{-\gamma^2s/2}$ $\mathrm{d}s/\sqrt{2\pi s^3}$, $\delta>0$, $\gamma>0$.

\begin{proposition}
	The L\'evy measure of $\{\bar{\mathcal{X}}^{\delta,\gamma}(t)\}_{t\ge0}$ is given by
	{\small\begin{align*}
			\Pi^{\bar{\mathcal{X}}}(A_1\times A_2\times\dots\times A_q)
			&=\frac{\delta\sqrt{2\lambda+\gamma^2}}{2\sqrt{\pi}} \sum_{\bar{n}\succ\bar{0}}\sum_{\substack{\Omega(k_i,n_i)\\i=1,2,\dots,q}}\Gamma\bigg(\sum_{i=1}^{q}\sum_{j_i=1}^{k_i}n_{ij_i}-\frac{1}{2}\bigg)\prod_{i=1}^{q}\prod_{j_i=1}^{k_i}\Big(\frac{2\lambda_{ij_i}}{2\lambda+\gamma^2}\Big)^{n_{ij_i}}\frac{\mathbb{I}_{\{n_i\in A_i\}}}{n_{ij_i}!},
	\end{align*}}
	where $\Omega(k_i,n_i)=\{(n_{i1},n_{i2},\dots,n_{ik_{i}}):\sum_{j_{i}=1}^{k_i}j_in_{ij_{i}}=n_i,\,n_{ij_{i}}\in\mathbb{N}_{0}\}$ and $\lambda=\sum_{i=1}^{q}\sum_{j_{i}=1}^{k_i}\lambda_{ij_{i}}$.
\end{proposition}

\section{Multivariate GCP with Bern\v{s}tein intertimes}
Kataria and Khandakar (2022b) studied a time-changed variant of the GCP in which the random time-change component is a L\'evy subordinator associated with some Bern\v{s}tein function. It is called the time-changed generalized counting process-I (TCGCP-I). For more details on TCGCP-I, we refer the reader to Kataria \textit{et al.} (2022).
Here, we introduce a multivariate version of it which is defined as follows:
\begin{equation*}
\bar{\mathcal{Z}}^f(t)=\bar{M}(D_f(t)):=(M_1(D_f(t)),M_2(D_f(t)),\dots,M_q(D_f(t))),
\end{equation*}
where the MGCP $\{\bar{M}(t)\}_{t\ge0}$ is independent of $\{D_f(t)\}_{t\geq0}$, a L\'evy subordinator associated with Bern\v{s}tein function $f$.
\begin{remark}
As a consequence of Remark $\ref{remk}$, it can be observed that the component processes $\{M_i(D_f(t))\}_{t\geq0}$ of $\{\bar{\mathcal{Z}}^f(t)\}_{t\geq0}$ are conditionally independent given $\{D_f(t)\}_{t\geq0}$. 
\end{remark}

In an infinitesimal time interval of length $h$ such that $o(h)/h\to0$ as $h\to0$, the transition probabilities of $\{\bar{\mathcal{Z}}^f(t)\}_{t\geq0}$ are given by
{\footnotesize\begin{equation}\label{mgcptrns}
\mathrm{Pr}\{\bar{\mathcal{Z}}^f(t+h)=\bar{n}+\bar{l}|\bar{\mathcal{Z}}^f(t)=\bar{n}\}=\begin{cases}
h\displaystyle\sum_{\substack{\Omega(k_i,l_i)\\i=1,2,\dots,q}}\bigg(\prod_{i=1}^{q}\prod_{j_i=1}^{k_i}\frac{\lambda_{ij_i}^{l_{ij_i}}}{l_{ij_i}!}\bigg)\int_{0}^{\infty}e^{-\lambda r}r^{\sum_{i=1}^{q}\sum_{j_i=1}^{k_i}l_{ij_i}}\mu_{D_f}(\mathrm{d}r)+o(h),\ \bar{l}\succ\bar{0},\vspace{.2cm}\\
1-h\int_{0}^{\infty}\left(1-e^{-\lambda r}\right)\mu_{D_f}(\mathrm{d}r)+o(h),\ \bar{l}=\bar{0},
\end{cases}
\end{equation}}
where $\Omega (k_i,l_i)=\{(l_{i1},l_{i2},\dots,l_{ik_i}):\sum_{j_i=1}^{k_i}j_il_{ij_i}=l_i,\  l_{ij_i}\in \mathbb{N}_{0}\}$,  $\lambda=\sum_{i=1}^{q}\sum_{j_i=1}^{k_i}\lambda_{ij_i}$ and $\mu_{D_f}(\cdot)$ is the L\'evy measure associated with Bern\v{s}tein function $f$. 

Note that the time-changed processes studied in Section \ref{secc1} are particular cases of $\{\bar{\mathcal{Z}}^f(t)\}_{t\geq0}$. It can be observed as follows:
\paragraph{Case I}
If we take 
$\mu_{D_f}(\mathrm{d}r)=\alpha r^{-\alpha-1} \mathrm{d}r/\Gamma(1-\alpha)$, $0<\alpha<1$, that is, the L\'evy measure associated with an $\alpha$-stable subordinator then \eqref{mgcptrns} reduces to the transition probabilities of MGSFCP given in \eqref{transprbstable}.

\paragraph{Case II}
If we take $\mu_{D_f}(\mathrm{d}r)=\alpha r^{-\alpha-1}e^{-\theta r}\mathrm{d}r/\Gamma(1-\alpha)$, $\theta>0,\, 0<\alpha<1$ in \eqref{mgcptrns}, that is, the L\'evy measure associated with a tempered $\alpha$-stable subordinator then \eqref{mgcptrns} reduces to the transition probabilities of multivariate tempered space fractional generalized counting process given in \eqref{temptrans}.

\paragraph{Case III}
If we take $\mu_{D_f}(\mathrm{d}r)=be^{-ar}\, \mathrm{d}r/r$, $a,b>0$, that is, the L\'evy measure associated with a gamma subordinator then \eqref{mgcptrns} reduces to the transition probabilities of MGCP time-changed by gamma subordinator given in \eqref{transgamma}.
\paragraph{Case IV}If we take $\mu_{D_f}(\mathrm{d}r)=\delta e^{-\gamma^2r/2}$ $\mathrm{d}r/\sqrt{2\pi r^3}$, $\delta>0, \gamma>0$, that is, the L\'evy measure associated with inverse Gaussian  subordinator then \eqref{mgcptrns} reduces to the transition probabilities of MGCP time-changed by an inverse Gaussian subordinator given in \eqref{transinvgauss}.

 We will use the following notations: Let
\begin{equation}\label{flambdau}
g(\lambda;\bar{u})=\int_{0}^{\infty}\bigg(1-e^{-\lambda r}\sum_{\bar{l}\ge\bar{0}}\sum_{\substack{\Omega(k_i,l_i)\\i=1,2,\dots,q}}\prod_{i=1}^{q}u_i^{l_i}\prod_{j_i=1}^{k_i}\frac{(\lambda_{ij_i}r)^{l_{ij_i}}}{l_{ij_i}!}\bigg)\mu_{D_f}(\mathrm{d}r)
\end{equation}
 and
\begin{equation}\label{flambda0}
g(\lambda;\bar{0})=\int_{0}^{\infty}(1-e^{-\lambda r})\mu_{D_f}(\mathrm{d}r)=f(\lambda).
\end{equation}
\begin{proposition}
The state probabilities $p^f_{\bar{\mathcal{Z}}}(\bar{n},t)=\mathrm{Pr}\{\bar{\mathcal{Z}}^f(t)=\bar{n}\}$, $\bar{n}\geq\bar{0}$ satisfy the following system of differential equations:
\begin{equation*}
	\frac{\mathrm{d}}{\mathrm{d}t}p^f_{\bar{\mathcal{Z}}}(\bar{n},t)=-g(\lambda;\bar{B})p^f_{\bar{\mathcal{Z}}}(\bar{n},t),\,\, p^f_{\bar{\mathcal{Z}}}(\bar{n},0)=\mathbb{I}_{\{\bar{n}=\bar{0}\}},
\end{equation*}
 where  $\lambda=\sum_{i=1}^{q}\sum_{j_i=1}^{k_i}\lambda_{ij_i}$ and $\bar{B}=(B_1,B_2,\dots,B_q)$ is a $q$-tuple backward shift operator vector such that $B_i^{r}p^f_{\bar{\mathcal{Z}}}(\bar{n},t)=p^f_{\bar{\mathcal{Z}}}((n_1,n_2,\dots$, $n_{i-1},n_i-r,n_{i+1},\dots,n_q),t)$.
\end{proposition}
\begin{proof}
As $\{\bar{\mathcal{Z}}^f(t)\}_{t\geq0}$ has independent and stationary increments, we have
\begin{align}
p^f_{\bar{\mathcal{Z}}}(\bar{n},t+h)&=\sum_{\bar{0}\le \bar{l}\leq \bar{n}}\mathrm{Pr}(\bar{\mathcal{Z}}^f(t)=\bar{l},\bar{\mathcal{Z}}^f(t+h)-\bar{\mathcal{Z}}^f(t)=\bar{n}-\bar{l})\nonumber\\
&=\sum_{\bar{0}\le \bar{l}\prec \bar{n}}p^f_{\bar{\mathcal{Z}}}(\bar{l},t)p^f_{\bar{\mathcal{Z}}}(\bar{n}-\bar{l},h)+p^f_{\bar{\mathcal{Z}}}(\bar{n},t)p^f_{\bar{\mathcal{Z}}}(\bar{0},h).\label{berpmf}
\end{align}
 On using \eqref{mgcptrns} and changing the summation indices in \eqref{berpmf}, we get
{\small\begin{align*}
p^f_{\bar{\mathcal{Z}}}(\bar{n},t+h)-p^f_{\bar{\mathcal{Z}}}(\bar{n},t)&=\sum_{\bar{0}\prec \bar{l}\leq \bar{n}}p^f_{\bar{\mathcal{Z}}}(\bar{n}-\bar{l},t)p^f_{\bar{\mathcal{Z}}}(\bar{l},h)-hp^f_{\bar{\mathcal{Z}}}(\bar{n},t)\int_{0}^{\infty}(1-e^{-\lambda r})\mu_{D_f}(\mathrm{d}r)+o(h)\\
&=h\sum_{\bar{0}\prec \bar{l}\leq \bar{n}}p^f_{\bar{\mathcal{Z}}}(\bar{n}-\bar{l},t)
\sum_{\substack{\Omega(k_i,l_i)\\i=1,2,\dots,q}}\bigg(\prod_{i=1}^{q}\prod_{j_i=1}^{k_i}\frac{\lambda_{ij_i}^{l_{ij_i}}}{l_{ij_i}!}\bigg)\int_{0}^{\infty}e^{-\lambda r}r^{\sum_{i=1}^{q}\sum_{j_i=1}^{k_i}l_{ij_i}}
\mu_{D_f}(\mathrm{d}r)\\
&\,\,\,\, -hp^f_{\bar{\mathcal{Z}}}(\bar{n},t)f(\lambda)+o(h),
\end{align*}}
where $\Omega (k_i,l_i)=\{(l_{i1},l_{i2},\dots,l_{ik_i}):\sum_{j_i=1}^{k_i}j_il_{ij_i}=l_i,\  l_{ij_i}\in \mathbb{N}_{0}\}$, $\lambda=\sum_{i=1}^{q}\sum_{j_i=1}^{k_i}\lambda_{ij_i}$ and in the the last step, we have used \eqref{flambda0}. 

Now, on taking the limit $h\to0$, we get
{\small\begin{align}
	\frac{\mathrm{d}}{\mathrm{d}t}p^f_{\bar{\mathcal{Z}}}(\bar{n},t)&=-f(\lambda)p^f_{\bar{\mathcal{Z}}}(\bar{n},t)+\sum_{\bar{0}\prec \bar{l}\leq \bar{n}}p^f_{\bar{\mathcal{Z}}}(\bar{n}-\bar{l},t)\sum_{\substack{\Omega(k_i,l_i)\\i=1,2,\dots,q}}\bigg(\prod_{i=1}^{q}\prod_{j_i=1}^{k_i}\frac{\lambda_{ij_i}^{l_{ij_i}}}{l_{ij_i}!}\bigg)\int_{0}^{\infty}e^{-\lambda r}r^{\sum_{i=1}^{q}\sum_{j_i=1}^{k_i}l_{ij_i}}\mu_{D_f}(\mathrm{d}r)\label{mgfcpde}\\
	&=p^f_{\bar{\mathcal{Z}}}(\bar{n},t)\Big(\int_{0}^{\infty}(1-e^{-\lambda r})\mu_{D_f}(\mathrm{d}r)\nonumber\\
	&\hspace{3cm} -\sum_{\bar{l}\succ\bar{0}}\Big(\prod_{i=1}^{q}B_i^{l_i}\Big)\sum_{\substack{\Omega(k_i,l_i)\\i=1,2,\dots,q}}\bigg(\prod_{i=1}^{q}\prod_{j_i=1}^{k_i}\frac{\lambda_{ij_i}^{l_{ij_i}}}{l_{ij_i}!}\bigg)\int_{0}^{\infty}e^{-\lambda r}r^{\sum_{i=1}^{q}\sum_{j_i=1}^{k_i}l_{ij_i}}\mu_{D_f}(\mathrm{d}r)\Big)\nonumber\\
	&= \int_{0}^{\infty}\bigg(1-e^{-\lambda r}\sum_{\bar{l}\ge\bar{0}}\sum_{\substack{\Omega(k_i,l_i)\\i=1,2,\dots,q}}\prod_{i=1}^{q}B_i^{l_i}\prod_{j_i=1}^{k_i}\frac{(\lambda_{ij_i}r)^{l_{ij_i}}}{l_{ij_i}!}\bigg)\mu_{D_f}(\mathrm{d}r)p^f_{\bar{\mathcal{Z}}}(\bar{n},t)
	=-g(\lambda;\bar{B})p^f_{\bar{\mathcal{Z}}}(\bar{n},t).\nonumber
\end{align}}
This proves the result.
\end{proof}

 Next, we obtain the differential equation that governs the pgf 
\begin{equation*}
G^f_{\bar{\mathcal{Z}}}(\bar{u},t)=\mathbb{E}\bigg(u_1^{M_1(D_f(t))}u_2^{M_2(D_f(t))}\dots u_q^{M_q(D_f(t))}\bigg)
	=\sum_{\bar{n}\ge\bar{0}}u_1^{n_1} u_2^{n_2}\dots u_q^{n_q}p^f_{\bar{\mathcal{Z}}}(\bar{n},t)
\end{equation*}
of multivariate TCGCP-I.

\begin{proposition}
The pgf $G^f_{\bar{\mathcal{Z}}}(\bar{u},t)$, $|u_i|\leq1$ solves the following differential equation:
\begin{equation*}
\frac{\partial}{\partial t}G^f_{\bar{\mathcal{Z}}}(\bar{u},t)=-g(\lambda;\bar{u})G^f_{\bar{\mathcal{Z}}}(\bar{u},t),\,\, G^f_{\bar{\mathcal{Z}}}(\bar{u},0)=1,
\end{equation*}
where $g(\lambda,\bar{u})$ is given in \eqref{flambdau}.
\end{proposition}
\begin{proof}
On multiplying both sides of \eqref{mgfcpde} with $u_1^{n_1}u_2^{n_2}\dots u_q^{n_q}$ and summing over the range of $\bar{n}\geq\bar{0}$, we have 
{\scriptsize\begin{align*}
\frac{\partial}{\partial t}G^f_{\bar{\mathcal{Z}}}(\bar{u},t)
&=-f(\lambda)G^f_{\bar{\mathcal{Z}}}(\bar{u},t)+\sum_{\bar{n}\geq\bar{0}}\Big(\prod_{i=1}^{q}u_i^{n_i}\Big)\sum_{\bar{0}\prec\bar{l}\leq \bar{n}}p^f_{\bar{\mathcal{Z}}}(\bar{n}-\bar{l},t)\sum_{\substack{\Omega(k_i,l_i)\\i=1,2,\dots,q}}\bigg(\prod_{i=1}^{q}\prod_{j_i=1}^{k_i}\frac{\lambda_{ij_i}^{l_{ij_i}}}{l_{ij_i}!}\bigg)\int_{0}^{\infty}e^{-\lambda r}r^{\sum_{i=1}^{q}\sum_{j_i=1}^{k_i}l_{ij_i}}\mu_{D_f}(\mathrm{d}r)\\
&=-f(\lambda)G^f_{\bar{\mathcal{Z}}}(\bar{u},t)+\sum_{\bar{l}\succ\bar{0}}\sum_{\bar{n}\geq \bar{l}}\Big(\prod_{i=1}^{q}u_i^{n_i-l_i}\Big)p^f_{\bar{\mathcal{Z}}}(\bar{n}-\bar{l},t)\sum_{\substack{\Omega(k_i,l_i)\\i=1,2,\dots,q}}\bigg(\prod_{i=1}^{q}\prod_{j_i=1}^{k_i}\frac{\lambda_{ij_i}^{l_{ij_i}}}{l_{ij_i}!}\bigg)\int_{0}^{\infty}e^{-\lambda r}r^{\sum_{i=1}^{q}\sum_{j_i=1}^{k_i}l_{ij_i}}\mu_{D_f}(\mathrm{d}r)\\
&=\bigg(-f(\lambda)+\sum_{\bar{l}\succ\bar{0}}\sum_{\substack{\Omega(k_i,l_i)\\i=1,2,\dots,q}}\bigg(\prod_{i=1}^{q}\prod_{j_i=1}^{k_i}\frac{\lambda_{ij_i}^{l_{ij_i}}}{l_{ij_i}!}\bigg)\int_{0}^{\infty}e^{-\lambda r}r^{\sum_{i=1}^{q}\sum_{j_i=1}^{k_i}l_{ij_i}}\mu_{D_f}(\mathrm{d}r)\bigg)G^f_{\bar{\mathcal{Z}}}(\bar{u},t)\\
&=-\bigg(\int_{0}^{\infty}\bigg(1-e^{-\lambda r}\sum_{\bar{l}\ge\bar{0}}\sum_{\substack{\Omega(k_i,l_i)\\i=1,2,\dots,q}}\prod_{i=1}^{q}u_i^{l_i}\prod_{j_i=1}^{k_i}\frac{(\lambda_{ij_i}r)^{l_{ij_i}}}{l_{ij_i}!}\bigg)\mu_{D_f}(\mathrm{d}r)\bigg)G^f_{\bar{\mathcal{Z}}}(\bar{u},t)=-g(\lambda;\bar{u})G^f_{\bar{\mathcal{Z}}}(\bar{u},t),
\end{align*}}
where we have used \eqref{flambdau} in the last step.
As $p^f_{\bar{\mathcal{Z}}}(\bar{n},0)=\mathbb{I}_{\{\bar{n}=\bar{0}\}}$, we have
\begin{equation*} G^f_{\bar{\mathcal{Z}}}(\bar{u},0)=\sum_{\bar{n}\ge\bar{0}}u_1^{n_1}u_2^{n_2}\dots u_q^{n_q}p^f_{\bar{\mathcal{Z}}}(\bar{n},0)=1.
\end{equation*}
This completes the proof.
\end{proof}
\begin{remark}
For $\beta\in(0,1)$, we consider the following time-changed variant of $\{\bar{\mathcal{Z}}^f(t)\}_{t\ge0}$:
\begin{equation*}
\bar{\mathcal{Z}}^{f,\beta}(t):=\bar{\mathcal{Z}}^{f}(Y_\beta(t)),
\end{equation*}
where $\{Y_\beta(t)\}_{t\geq0}$ is the first passage time of a $\beta$-stable subordinator $\{D_\beta(t)\}_{t\geq0}$. 

On using the same argument as given in Remark 2.3 of Orsingher and Toaldo (2015), it can be shown that the state probabilities $p^{f,\beta}_{\bar{\mathcal{Z}}}(\bar{n},t)$, $\bar{n}\ge\bar{0}$ satisfy the following system of fractional differential equations:
{\footnotesize\begin{equation*}
	\frac{\mathrm{d}^\beta}{\mathrm{d}t^\beta}p^{f,\beta}_{\bar{\mathcal{Z}}}(\bar{n},t)=-f(\lambda)p^{f,\beta}_{\bar{\mathcal{Z}}}(\bar{n},t)+\sum_{\bar{0}\prec \bar{l}\le \bar{n}}p^{f,\beta}_{\bar{\mathcal{Z}}}(\bar{n}-\bar{l},t)\sum_{\substack{\Omega(k_i,l_i)\\i=1,2,\dots,q}}\bigg(\prod_{i=1}^{q}\prod_{j_i=1}^{k_i}\frac{\lambda_{ij_i}^{l_{ij_i}}}{l_{ij_i}!}\bigg)\int_{0}^{\infty}e^{-\lambda r}r^{\sum_{i=1}^{q}\sum_{j_i=1}^{k_i}l_{ij_i}})\mu_{D_f}(\mathrm{d}r),
\end{equation*}}
with initial condition $p^{f,\beta}_{\bar{\mathcal{Z}}}(\bar{n},0)=\mathbb{I}_{\{\bar{n}=\bar{0}\}}$. Here, $\frac{\partial^\beta}{\partial t^\beta}$is the Caputo fractional derivative given in \eqref{caputo}.
Equivalently,
\begin{equation*}
\frac{\mathrm{d}^\beta}{\mathrm{d}t^\beta}p^{f,\beta}_{\bar{\mathcal{Z}}}(\bar{n},t)=-g(\lambda;\bar{B})p^{f,\beta}_{\bar{\mathcal{Z}}}(\bar{n},t),\ \ p^{f,\beta}_{\bar{\mathcal{Z}}}(\bar{n},0)=\mathbb{I}_{\{\bar{n}=\bar{0}\}}.
\end{equation*}
Moreover, its pgf $G^{f,\beta}_{\bar{\mathcal{Z}}}(\bar{u},t)$ solves the following equation:
\begin{equation*}
\frac{\partial^\beta}{\partial t^\beta}G^{f,\beta}_{\bar{\mathcal{Z}}}(\bar{u},t)=-g(\lambda;\bar{u})G^{f,\beta}_{\bar{\mathcal{Z}}}(\bar{u},t),\ \ G^{f,\beta}_{\bar{\mathcal{Z}}}(\bar{u},0)=1.
\end{equation*}

Thus, 
$
G^{f,\beta}_{\bar{\mathcal{Z}}}(\bar{u},t)=E_{\beta,1}(-t^\beta g(\lambda;\bar{u}))
$
which follows from the fact that the Mittag-Leffeler function is an eigenfunction of the Caputo fractional derivative. 
\end{remark}
\section{Application of MGSFCP to shock models}
In this section, we discuss an application of the MGSFCP $\{\bar{M}^\alpha(t)\}_{t\ge0}$ to a shock model which is subject to $q$-types of shocks. Let $\mathcal{Q}(t)$ denote the total number of shocks in $[0,t]$. Then,
\begin{equation*}
\mathcal{Q}(t)=\sum_{i=1}^{q}M_i^\alpha(t),
\end{equation*}
where for each $i=1,2, \dots,q$, $M_i^\alpha(t)$ models the number of shocks of type $i$ in $[0,t]$. The failure of the system due to shock of type $i$ is denoted by $\sigma=i$, $i=1,2, \dots,q$. 
Let $\mathcal{N}$ be a random threshold which takes values in $\mathbb{N}$ and we assume that the system fails when $\mathcal{Q}(t)\ge\mathcal{N}$. The probability distribution and the reliability
function of $\mathcal{N}$ are given by
\begin{align}
\mathrm{Pr}\{\mathcal{N}=n\}&=p_n,\,\, n\in\mathbb{N},\label{p_n}\\
\mathrm{Pr}\{\mathcal{N}>n\}&=q_n,\,\, n\in\mathbb{N}_0,\nonumber
\end{align}
respectively.

We consider a
non-negative absolutely continuous random variable $T$ which represents the random failure time of the system. It is defined as
\begin{equation*}
	T\coloneqq\inf\{t\ge0:\mathcal{Q}(t)\ge\mathcal{N}\}
\end{equation*}
which is first hitting time of $\{\mathcal{Q}(t)\}_{t\ge0}$. 
For $i=1,2,\dots,q$, its pdf $h_T(t)$ is given by
\begin{equation*}
h_T(t)=\sum_{i=1}^{q}h_{i}(t),\ \ t\ge0,
\end{equation*} 
where the failure densities $h_{i}(t)$ are defined as
\begin{equation*}
h_{i}(t)=\frac{\mathrm{d}}{\mathrm{d}t}\mathrm{Pr}\{T\le t,\, \sigma=i\}, \, t\ge0.
\end{equation*}
Thus, the probability that failure occurs due to shock of type $i$ is given by
\begin{equation}\label{hzsigma}
\mathrm{Pr}\{\sigma=i\}=\int_{0}^{\infty}h_{i}(t)\, \mathrm{d}t.
\end{equation}
 For any  $i=1,2,\dots,q$ and $l_i=1,2,\dots,k_i$ and $\bar{n}\ge\bar{0}$, the hazard rate $R_{il_i}(\bar{n};t)$, $t\ge0$ represents the intensity of occurrence of shock of type $i$ due to a jump of size $l_i$ immediately after $t$. It is defined as
\begin{equation}\label{hzrates}
R_{il_i}(\bar{n};t)=\lim_{h\to0}\frac{1}{h}\mathrm{Pr}\Big(M_i^\alpha(t+h)=n_i+l_i,\,\cap_{\substack{r=1\\r\ne i}}^{q} \{M_r^\alpha(t+h)=n_r\}\Big|\cap_{r=1}^{q}\{M_r^\alpha(t)=n_r\}\Big).
\end{equation}
Now, conditional on $\mathcal{N}$, and on using \eqref{p_n} and \eqref{hzrates}, the failure densities can be written as
\begin{equation}\label{faidenex}
h_{i}(t)=\sum_{n=1}^{\infty}p_n\sum_{n-k_i\le n_1+n_2+\dots+n_q\le n-1}\mathrm{Pr}(\cap_{i=1}^{q}\{M_i^\alpha(t)=n_i\})\sum_{n-(n_1+n_2+\dots+n_q)}^{k_i}R_{il_i}(\bar{n};t),\ \ t\ge0.
\end{equation}

The survival function $L_{T}(t)=\mathrm{Pr}\{T>t\}$, $t\ge0$ of random failure time is given by
\begin{equation}\label{relfn}
L_{T}(t)=\sum_{n=0}^{\infty}q_n\sum_{n_1+n_2+\dots+n_q=n}p^\alpha(\bar{n},t),\, \, q_0=1.
\end{equation}

\begin{proposition}
For $i=1,2,\dots,q$, $l_i=1,2,\dots,k_i$ and $\bar{n}\in\mathbb{N}_0^q$, the hazard rate is given by
\begin{equation}\label{hzratex}
R_{il_i}(\bar{n};t)=\frac{C(\alpha,t)}{p^\alpha(\bar{n},t)}, \,\,t\ge0,
\end{equation}
where $p^\alpha(\bar{n},t)$ is the pmf of MGSFCP given in \eqref{alphawright} and
\begin{align*}
C(\alpha,t)&=\alpha \lambda^{\alpha-1}\lambda_{il_i}e^{-\lambda^\alpha t}\sum_{\substack{\Omega(k_r,n_r)\\r=1,2,\dots,q}}\Big(\prod_{r=1}^{q}\prod_{j_r=1}^{k_r}\frac{(-\lambda_{rj_r}/\lambda)^{n_{rj_r}}}{n_{rj_r}!}\Big) \sum_{x=0}^{z_{k_1}+z_{k_2}+\dots+z_{k_q}}\frac{(\lambda^\alpha t)^{x}}{x!}\\
&\hspace{7cm}\cdot\sum_{y=0}^{x}(-1)^y\binom{x}{y}(\alpha y)_{z_{k_1}+z_{k_2}+\dots+z_{k_q}}.
\end{align*}
Here, $z_{k_r}=\sum_{j_r=1}^{k_r}n_{rj_r}$, $\lambda=\sum_{r=1}^{q}\sum_{j_r=1}^{k_r}\lambda_{rj_r}$, $\Omega(k_r,n_r)=\{(n_{r1},n_{r2},\dots,n_{rk_r}):\sum_{j_r=1}^{k_r}j_rn_{rj_r}=n_r,\,\,n_{rj_r}\in\mathbb{N}_0\}$ and $(x)_k=x(x-1)\dots (x-k+1)$ denotes the falling factorial.
\end{proposition}
\begin{proof}
For each $i\in\{1,2,\dots,q\}$ and $l_i\in\{1,2,\dots,k_i\}$, the hazard rates can be written as
\begin{equation}\label{tttmt}
R_{il_i}(\bar{n};t)=\lim_{h\to t}\frac{\mathrm{Pr}\Big(M_i^\alpha(h)=n_i+l_i, \,\cap_{\substack{r=1\\r\ne i}}^{q}\{M_r^\alpha(h)=n_r\},\, \cap_{r=1}^{q}\{M_r^\alpha(t)=n_r\}\Big) }{(h-t)p^\alpha(\bar{n},t)}.
\end{equation}
On using \eqref{eqjoalpha}, we have
\begin{align}\label{NUM1}
\mathrm{Pr}&\{M_i^\alpha(h)=n_i+l_i, \cap_{\substack{r=1\\r\ne i}}^{q}\{M_r^\alpha(h)=n_r\},\, \cap_{r=1}^{q}\{M_r^\alpha(t)=n_r\}\nonumber\\
&\ \ =\int_{0}^{\infty}\int_{0}^{s_2}\mathrm{Pr}\Big(M_i(s_2)=n_i+l_i, \cap_{\substack{r=1\\r\ne i}}^{q}\{M_r(s_2)=n_r\},\, \cap_{r=1}^{q}\{M_r(s_1)=n_r\}\Big)\nonumber\\
&\hspace{6cm}\cdot h_{D_\alpha(h-t)}(s_2-s_1,h-t)h_{D_\alpha(t)}(s_1,t)\mathrm{d}s_1\,\mathrm{d}s_2, 
\end{align}
where $h_{D_\alpha(t)}(\cdot,t)$ is the density function of $\alpha$-stable subordinator $\{D_\alpha(t)\}_{t\ge0}$.

Now, using the independence of $\{M_i(t)\}_{t\ge0}$, we have
{\small\begin{align}\label{jojoj}
\mathrm{Pr}\Big(M_i(s_2)=n_i+l_i, \cap_{\substack{r=1\\r\ne i}}^{q}\{&M_r(s_2)=n_r\},\, \cap_{r=1}^{q}\{M_r(s_1)=n_r\}\Big)\nonumber\\
&=\sum_{\Omega(k_i,l_i)}\Big(\prod_{j_i=1}^{k_i}\frac{(\lambda_{ij_i}(s_2-s_1))^{x_{ij_i}}}{x_{ij_i}!}\Big)\sum_{\substack{\Omega(k_r,n_r)\\r=1,2,\dots,q}}\Big(\prod_{r=1}^{q}\prod_{j_r=1}^{k_r}\frac{(\lambda_{rj_r}s_1)^{n_{rj_r}}}{n_{rj_r}!}\Big)e^{-\lambda s_2},
\end{align}}
where $\Omega(k_i,l_i)=\{(x_{i1},x_{i2},\dots,x_{ik_i}):\sum_{j_i=1}^{k_i}j_ix_{ij_i}=l_i,\, x_{ij_i}\in\mathbb{N}_0\}$.

On substituting \eqref{jojoj} in \eqref{NUM1} and changing the order of integration, we get
\begin{align*}
\mathrm{Pr}&\{M_i^\alpha(h)=n_i+l_i, \cap_{\substack{r=1\\r\ne i}}^{q}\{M_r^\alpha(h)=n_r\},\, \cap_{r=1}^{q}\{M_r^\alpha(t)=n_r\}\nonumber\\
\ \ &=\sum_{\Omega(k_i,l_i)}\Big(\prod_{j_i=1}^{k_i}\frac{\lambda_{ij_i}^{x_{ij_i}}}{x_{ij_i}!}\Big)\sum_{\substack{\Omega(k_r,n_r)\\r=1,2,\dots,q}}\Big(\prod_{r=1}^{q}\prod_{j_r=1}^{k_r}\frac{\lambda_{rj_r}^{n_{rj_r}}}{n_{rj_r}!}\Big)\nonumber\\ &\hspace{2cm}\cdot\int_{0}^{\infty}\int_{s_1}^{\infty}e^{-\lambda s_2}(s_2-s_1)^{c_{k_i}} s_1^{\sum_{r=1}^{q}z_{k_r}} h_{D_\alpha(h-t)}(s_2-s_1,h-t)h_{D_\alpha(t)}(s_1,t)\,\mathrm{d}s_2\,\mathrm{d}s_1\\
&=\sum_{\Omega(k_i,l_i)}\Big(\prod_{j_i=1}^{k_i}\frac{\lambda_{ij_i}^{x_{ij_i}}}{x_{ij_i}!}\Big)\sum_{\substack{\Omega(k_r,n_r)\\r=1,2,\dots,q}}\Big(\prod_{r=1}^{q}\prod_{j_r=1}^{k_r}\frac{\lambda_{rj_r}^{n_{rj_r}}}{n_{rj_r}!}\Big) \int_{0}^{\infty}e^{-\lambda s_1}s_1^{\sum_{r=1}^{q}z_{k_r}}h_{D_\alpha(t)}(s_1,t)\,\mathrm{d}s_1\\
&\hspace{8.5cm}\cdot\int_{0}^{\infty}e^{-\lambda s_2}s_2^{c_{k_i}}h_{D_\alpha(h-t)}(s_2,h-t)\,\mathrm{d}s_2\\
&=\sum_{\Omega(k_i,l_i)}\Big(\prod_{j_i=1}^{k_i}\frac{\lambda_{ij_i}^{x_{ij_i}}}{x_{ij_i}!}\Big)\sum_{\substack{\Omega(k_r,n_r)\\r=1,2,\dots,q}}\Big(\prod_{r=1}^{q}\prod_{j_r=1}^{k_r}\frac{\lambda_{rj_r}^{n_{rj_r}}}{n_{rj_r}!}\Big)(-1)^{\sum_{r=1}^{q}z_{k_r}}\\
&\hspace{5.5cm}\cdot \bigg(\frac{\mathrm{d}^{\sum_{r=1}^{q}z_{k_r}}}{\mathrm{d}\lambda^{\sum_{r=1}^{q}z_{k_r}}}\mathbb{E}(e^{-\lambda D_\alpha(t)})\bigg) (-1)^{c_{k_i}}\frac{\mathrm{d}^{c_{k_i}}}{\mathrm{d}\lambda^{c_{k_i}}}\mathbb{E}(e^{-\lambda D_\alpha(h-t)}),
\end{align*}
where $c_{k_i}=\sum_{j_i=1}^{k_i}x_{ij_i}$.
On using Lemma 3.1 of Di Crescenzo and Meoli (2022), we get
\begin{align}\label{hznum}
\mathrm{Pr}&\{M_i^\alpha(h)=n_i+l_i, \cap_{\substack{r=1\\r\ne i}}^{q}\{M_r^\alpha(h)=n_r\},\, \cap_{r=1}^{q}\{M_r^\alpha(t)=n_r\}\nonumber\\
\ \
&=\sum_{\Omega(k_i,l_i)}\Big(\prod_{j_i=1}^{k_i}\frac{\lambda_{ij_i}^{x_{ij_i}}}{x_{ij_i}!}\Big)\sum_{\substack{\Omega(k_r,n_r)\\r=1,2,\dots,q}}\Big(\prod_{r=1}^{q}\prod_{j_r=1}^{k_r}\frac{(-\lambda_{rj_r}/\lambda)^{n_{rj_r}}}{n_{rj_r}!}\Big) \sum_{x=0}^{z_{k_1}+z_{k_2}+\dots+z_{k_q}}\frac{(\lambda^\alpha t)^{x}}{x!}\nonumber\\
&\ \ \cdot \sum_{y=0}^{x}(-1)^y\binom{x}{y}(\alpha y)_{z_{k_1}+z_{k_2}+\dots+z_{k_q}}\Big(-\frac{1}{\lambda}\Big)^{c_{k_i}}\sum_{u=0}^{c_{k_i}}\frac{(\lambda^\alpha (h-t))^{u}}{u!}\sum_{v=0}^{u}(-1)^v\binom{u}{v}(\alpha v)_{c_{k_i}}e^{-h\lambda^\alpha}.
\end{align}
Finally, on substituting \eqref{hznum} in \eqref{tttmt},  we get the required result.
\end{proof}
\begin{remark}
On substituting $q=2$ and $k_1=k_2=1$ in \eqref{hzratex}, we get the hazard rates corresponding to bivariate space fractional Poisson process (see Di Crescenzo and Meoli (2022), Proposition 3.1).
\end{remark}
\begin{remark}
For $i=1,2,\dots,q$, the following results hold true:\\ 
\noindent (i) On substituting \eqref{hzratex} in \eqref{faidenex}, the failure densities are given by
\begin{align}\label{faidensity}
	h_i(t)&=\alpha \lambda^{\alpha-1}e^{-\lambda^\alpha t}\sum_{n=1}^{\infty}p_n\sum_{n-1\le \sum_{r=1}^{q}n_r\le n-k_i}\sum_{l_i=n-(\sum_{r=1}^{q}n_r)}^{k_i}\lambda_{il_i}\sum_{\substack{\Omega(k_r,n_r)\\r=1,2,\dots,q}}\Big(\prod_{r=1}^{q}\prod_{j_r=1}^{k_r}\frac{(-\lambda_{rj_r}/\lambda)^{n_{rj_r}}}{n_{rj_r}!}\Big)\nonumber\\
	&\hspace{4cm} \cdot\sum_{x=0}^{z_{k_1}+z_{k_2}+\dots+z_{k_q}}\frac{(\lambda^\alpha t)^{x}}{x!}\sum_{y=0}^{x}(-1)^y\binom{x}{y}(\alpha y)_{z_{k_1}+z_{k_2}+\dots+z_{k_q}}.
\end{align}
\noindent (ii) On substituting \eqref{alphawright} in \eqref{relfn}, the reliability function of random failure time is given by
{\small\begin{align}\label{reliabexp}
	L_T(t)&=\sum_{n=0}^{\infty}q_n\sum_{n_1+n_2+\dots+n_q=n}\sum_{\substack{\Omega(k_r,n_r)\\r=1,2,\dots,q}}
	{}_1\Psi_1\left[\begin{matrix}
		(1,\alpha)_{1,1}\\
		\Big(1-\sum_{r=1}^{q}\sum_{j_r=1}^{k_r}n_{rj_r}, \alpha\Big)_{1,1}
	\end{matrix}\Bigg| -\lambda^\alpha t\right]\prod_{r=1}^{q}\prod_{j_r=1}^{k_r}\frac{\big(-\lambda_{rj_r}/\lambda\big)^{n_{rj_r}}}{n_{rj_r}!}.
\end{align}}
Here, $\Omega(k_r,n_r)=\{(n_{r1},n_{r2},\dots,n_{rk_r}):\sum_{j_r=1}^{k_r}j_rn_{rj_r}=n_r,\,\,n_{rj_r}\in\mathbb{N}_0\}$, $\lambda=\sum_{r=1}^{q}\sum_{j_r=1}^{k_r}\lambda_{rj_r}$ and $z_{k_r}=\sum_{j_r=1}^{k_r}n_{rj_r}$.
\end{remark}
\begin{remark}
For $q=2$ and $k_1=k_2=1$, the failure densities and the reliability function given in \eqref{faidensity} and \eqref{reliabexp} reduces to corresponding results for the case of bivariate space fractional Poisson process (see Di Crescenzo and Meoli (2022), Proposition 3.2).
\end{remark}
\begin{proposition}
For $i=1,2,\dots,q$, we have
\begin{align*}
\mathrm{Pr}\{\sigma=i\}&=\frac{\alpha}{\lambda}\sum_{n=1}^{\infty}p_n\sum_{n-1\le\sum_{r=1}^{q}n_r\le n-k_i}\,\sum_{l_i=n-(\sum_{r=1}^{q}n_r)}^{k_i}\lambda_{il_i}\sum_{\substack{\Omega(k_r,n_r)\\r=1,2,\dots,q}}\Big(\prod_{r=1}^{q}\prod_{j_r=1}^{k_r}\frac{(-\lambda_{rj_r}/\lambda)^{n_{rj_r}}}{n_{rj_r}!}\Big)\nonumber\\
&\hspace{5cm} \cdot\sum_{x=0}^{z_{k_1}+z_{k_2}+\dots+z_{k_q}}\sum_{y=0}^{x}(-1)^y\binom{x}{y}(\alpha y)_{z_{k_1}+z_{k_2}+\dots+z_{k_q}},
\end{align*}
where $\Omega(k_r,n_r)=\{(n_{r1},n_{r2},\dots,n_{rk_r}):\sum_{j_r=1}^{k_r}j_rn_{rj_r}=n_r,\,\,n_{rj_r}\in\mathbb{N}_0\}$, $\lambda=\sum_{r=1}^{q}\sum_{j_r=1}^{k_r}\lambda_{rj_r}$, $z_{k_r}=\sum_{j_r=1}^{k_r}n_{rj_r}$ and $(x)_k=x(x-1)\dots (x-k+1)$ denotes the falling factorial.
\end{proposition}
\begin{proof}
On substituting \eqref{faidensity} in \eqref{hzsigma}, we have
\begin{align*}
\mathrm{Pr}\{\sigma=i\}&=\alpha \lambda^{\alpha-1}\sum_{n=1}^{\infty}p_n\sum_{n-1\le \sum_{r=1}^{q}n_r\le n-k_i}\sum_{l_i=n-(\sum_{r=1}^{q}n_r)}^{k_i}\lambda_{il_i}\sum_{\substack{\Omega(k_r,n_r)\\r=1,2,\dots,q}}\Big(\prod_{r=1}^{q}\prod_{j_r=1}^{k_r}\frac{(-\lambda_{rj_r}/\lambda)^{n_{rj_r}}}{n_{rj_r}!}\Big)\nonumber\\
&\hspace{2cm} \cdot\sum_{x=0}^{z_{k_1}+z_{k_2}+\dots+z_{k_q}}\frac{\lambda^{\alpha x} }{x!}\sum_{y=0}^{x}(-1)^y\binom{x}{y}(\alpha y)_{z_{k_1}+z_{k_2}+\dots+z_{k_q}}\int_{0}^{\infty}t^{x}e^{-\lambda^\alpha t}\, \mathrm{d}t\\
&=\alpha \lambda^{\alpha-1}\sum_{n=1}^{\infty}p_n\sum_{n-1\le \sum_{r=1}^{q}n_r\le n-k_i}\sum_{l_i=n-(\sum_{r=1}^{q}n_r)}^{k_i}\lambda_{il_i}\sum_{\substack{\Omega(k_r,n_r)\\r=1,2,\dots,q}}\Big(\prod_{r=1}^{q}\prod_{j_r=1}^{k_r}\frac{(-\lambda_{rj_r}/\lambda)^{n_{rj_r}}}{n_{rj_r}!}\Big)\nonumber\\
&\hspace{2cm} \cdot\sum_{x=0}^{z_{k_1}+z_{k_2}+\dots+z_{k_q}}\frac{\lambda^{\alpha x} }{x!}\sum_{y=0}^{x}(-1)^y\binom{x}{y}(\alpha y)_{z_{k_1}+z_{k_2}+\dots+z_{k_q}}
\frac{\Gamma(x+1)}{\lambda^{\alpha x+\alpha}}
\end{align*}
which reduces to the required result.
\end{proof}
\subsection{Some particular cases of the reliability function} Here, we present four examples of the reliability function for random failure time by examining specific distributions $q_n=\mathrm{Pr}\{\mathcal{N}>n\}$, $n\geq0$, of the random threshold $\mathcal{N}$. For different choices of parameters, we include plots of the reliability function $L_T(\cdot)$ corresponding to the following cases:\\
	\noindent I. Let $\mathcal{N}$ be geometrically distributed with parameter $p$, that is, $q_n=(1-p)^n$, $0<p\le1$. Then, on substituting $q=2$, $k_1=1$ and $k_2=2$ in \eqref{reliabexp}, we get

	\begin{align}
		L_T(t)&=\sum_{n=0}^{\infty}(1-p)^n\sum_{n_1+n_{21}+2n_{22}=n}\Big(-\frac{1}{\lambda_1+\lambda_2}\Big)^{n_1+n_{21}+n_{22}}\frac{\lambda_1^{n_1}}{n_1!}\frac{\lambda_{21}^{n_{21}}}{n_{21}!}\frac{\lambda_{22}^{n_{22}}}{n_{22}!}\nonumber\\
		&\hspace{3.3cm} \cdot\sum_{r=0}^{\infty}\frac{(-(\lambda_1+\lambda_2)^\alpha t)^r}{r!}\frac{\Gamma(\alpha r+1)}{\Gamma(\alpha r-(n_1+n_{21}+n_{22})+1)}.\label{fig1}
	\end{align}	
	\noindent II. Let $\mathcal{N}$ be logarithmically distributed with parameter $p$, that is,
	$q_n=-B(p;n+1,0)/\ln(1-p)$, $0<p\le1$,
	where $B(n;x,y)$ is the incomplete beta function (see \eqref{defbeta}). Then, on substituting $q=2$, $k_1=1$ and $k_2=2$ in \eqref{reliabexp}, we get
	\begin{align}
		L_T(t)&=\frac{-1}{\ln(1-p)}\sum_{n=0}^{\infty}\int_{0}^{p}\frac{z^n}{(1-z)}\mathrm{d}z\sum_{n_1+n_{21}+2n_{22}=n}\Big(-\frac{1}{\lambda_1+\lambda_2}\Big)^{n_1+n_{21}+n_{22}}\nonumber\\
		&\hspace{1.2cm}  \cdot\frac{\lambda_1^{n_1}}{n_1!}\frac{\lambda_{21}^{n_{21}}}{n_{21}!}\frac{\lambda_{22}^{n_{22}}}{n_{22}!}\sum_{r=0}^{\infty}\frac{(-(\lambda_1+\lambda_2)^\alpha t)^r}{r!}\frac{\Gamma(\alpha r+1)}{\Gamma(\alpha r-(n_1+n_{21}+n_{22})+1)}.\label{fig2}
	\end{align}	
	
	\noindent III. Let $\mathcal{N}$ has the following distribution:
	\begin{equation*}
		q_n=\frac{\gamma(n+1;a,p)}{e^{-a}-e^{-p}},\, 0\le a < p \le 1,
	\end{equation*}
	where $\gamma(n;x,y)$ is the generalized incomplete gamma function (see \eqref{defgama}). Then, on taking $q=2$, $k_1=1$ and $k_2=2$ in \eqref{reliabexp}, we obtain
	\begin{align}
		L_T(t)&=\frac{1}{e^{-a}-e^{-p}}\sum_{n=0}^{\infty}\int_{a}^{p}e^{-z}z^n\mathrm{d}z\sum_{n_1+n_{21}+2n_{22}=n}\Big(-\frac{1}{\lambda_1+\lambda_2}\Big)^{n_1+n_{21}+n_{22}}\nonumber\\
		&\hspace{1.2cm} \cdot\frac{\lambda_1^{n_1}}{n_1!}\frac{\lambda_{21}^{n_{21}}}{n_{21}!}\frac{\lambda_{22}^{n_{22}}}{n_{22}!}\sum_{r=0}^{\infty}\frac{(-(\lambda_1+\lambda_2)^\alpha t)^r}{r!}\frac{\Gamma(\alpha r+1)}{\Gamma(\alpha r-(n_1+n_{21}+n_{22})+1)}.\label{fig3}
	\end{align}	
	\noindent IV. Let us consider the following distribution of $\mathcal{N}$ (see Di Crescenzo and Meoli (2022), Section 4): 
	\begin{equation*}
		q_n=\frac{Si(n+1;a,p)}{\cos(a)-\cos(p)},\, 0\le a < p \le 1,
	\end{equation*} 
	where   
$Si(x;p,q)=\int_{p}^{q}t^{x-1}\sin(t)\mathrm{d}t$
is the generalized sine integral.
 Then, on substituting $q=2$, $k_1=1$ and $k_2=2$ in \eqref{reliabexp}, we get
	\begin{align}
		L_T(t)&=\frac{1}{\cos(a)-\cos(p)}\sum_{n=0}^{\infty}\int_{a}^{p}z^n\sin(z)\mathrm{d}z\sum_{n_1+n_{21}+2n_{22}=n}\Big(-\frac{1}{\lambda_1+\lambda_2}\Big)^{n_1+n_{21}+n_{22}}\nonumber\\
		&\hspace{1.2cm} \cdot\frac{\lambda_1^{n_1}}{n_1!}\frac{\lambda_{21}^{n_{21}}}{n_{21}!}\frac{\lambda_{22}^{n_{22}}}{n_{22}!}\sum_{r=0}^{\infty}\frac{(-(\lambda_1+\lambda_2)^\alpha t)^r}{r!}\frac{\Gamma(\alpha r+1)}{\Gamma(\alpha r-(n_1+n_{21}+n_{22})+1)}.\label{fig4}
	\end{align}	
	\begin{figure}[htp]
		\includegraphics[width=15cm]{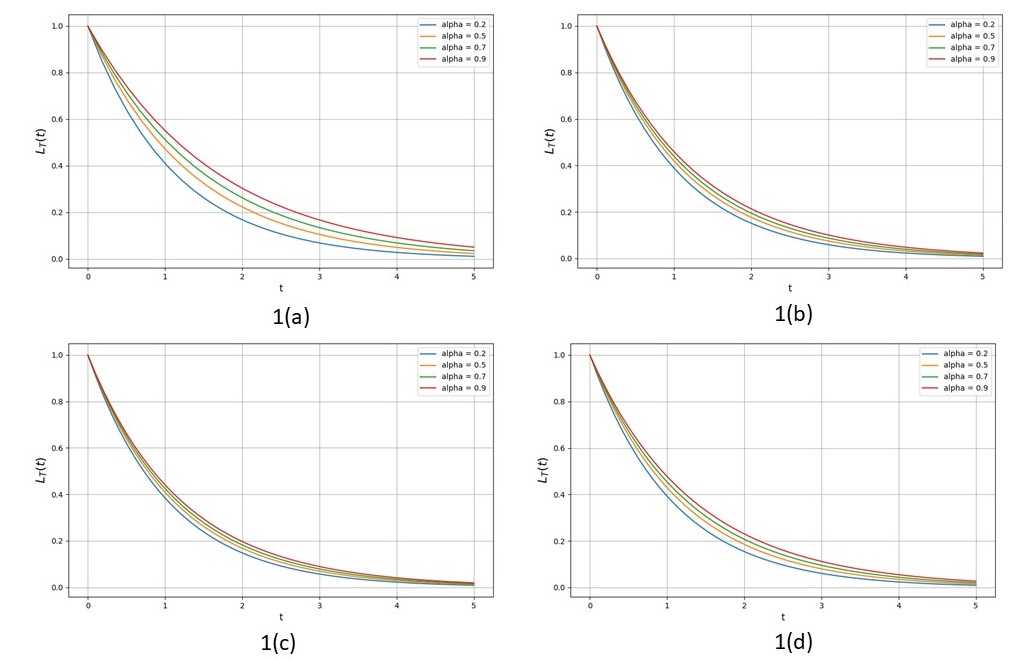}
		\vspace{-.3cm}\caption{Plots 1(a), 1(b), 1(c) and 1(d) represents the reliability function \eqref{fig1}, \eqref{fig2}, \eqref{fig3} and \eqref{fig4}, respectively for the parameters values $a=0$, $p=0.5$,  $\lambda_1=0.5$, $\lambda_{21}=0.5$, $\lambda_{22}=0.5$, that is, $\lambda_2=\lambda_{21}+\lambda_{22}=1$.}\label{slide1}
	\end{figure}
	
	\begin{figure}[htp]
		\includegraphics[width=15cm]{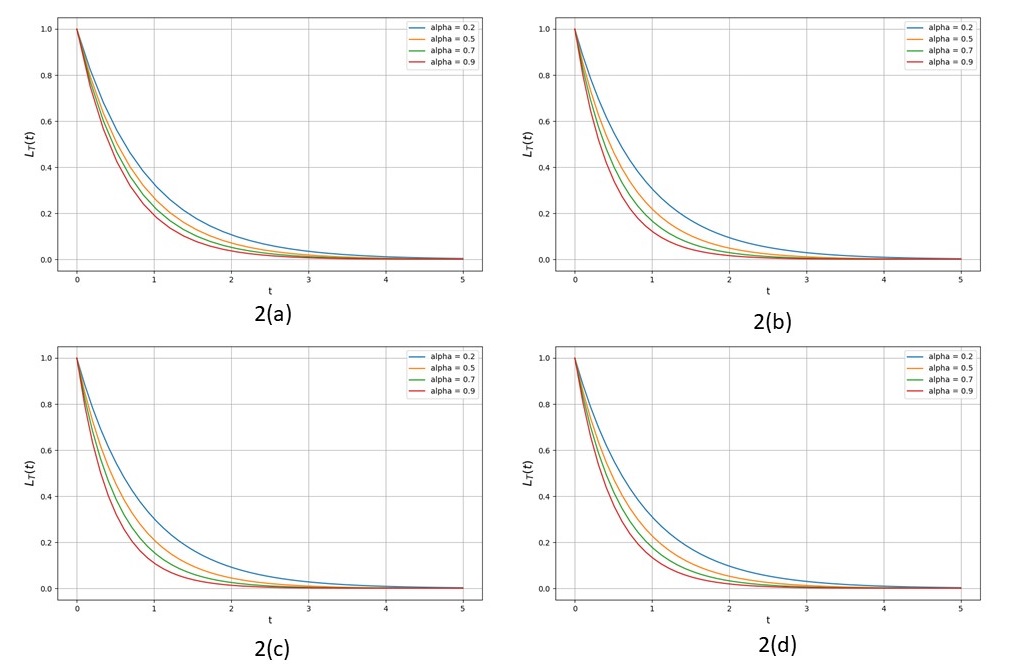}
		\vspace{-.3cm}\caption{Plots 2(a), 2(b), 2(c) and 2(d) represents the reliability function \eqref{fig1}, \eqref{fig2}, \eqref{fig3} and \eqref{fig4}, respectively  for the parameters values $a=0$, $p=0.5$, $\lambda_1=1$, $\lambda_{21}=1$, $\lambda_{22}=1$, that is,  $\lambda_2=\lambda_{21}+\lambda_{22}=2$.}\label{slide2}
	\end{figure}
	
	\begin{remark} 
		From the plots of the reliability function corresponding to the four cases discussed above, we have the following observations:
		
		\noindent(i) In Figure \ref{slide1}, we can observe that the reliability function $L_T(\cdot)$ is increasing with the increase in value of $\alpha$ for all the cases I to IV. Moreover, the distributions corresponding to \eqref{fig2}, \eqref{fig3} and \eqref{fig4} have thinner tails than that of \eqref{fig1}. \\
		\noindent(ii) In Figure \ref{slide2}, we can observe that the reliability function $L_T(\cdot)$ is decreasing with the increase in value of $\alpha$ for all the cases I to IV. Moreover, the distributions for all the four cases have thinner tails in Figure \ref{slide2} than that in Figure \ref{slide1}.
\end{remark}

\section*{Acknowledgments} 
The second author thanks Government of India for the grant of Prime Minister Research Fellowship, ID 1003066. 
The authors thank Mr. Pradeep Vishwakarma for his help in coding to plot the graphs in Figure~\ref{slide1} and Figure~\ref{slide2}.
\appendix
\renewcommand{\thesection}{\Alph{section}}
\section{}

\paragraph{A1}{\textit{Proof of Proposition \ref{prppp}}:} On using $\eqref{zzeq}$, we have
{\footnotesize	\begin{align}
		\frac{\partial}{\partial t}G^\alpha(\bar{u},t)&=-\lambda^\alpha\Big(1-\frac{1}{\lambda}\sum_{i=1}^{q}\sum_{j_i=1}^{k_i}\lambda_{ij_i}u_i^{j_i}\Big)^{\alpha } G^\alpha(\bar{u},t)\nonumber\\
		&=-\lambda^\alpha\sum_{x\geq0}\binom{\alpha }{x}\bigg(-\frac{1}{\lambda}\bigg)^x\bigg(\sum_{i=1}^{q}\sum_{j_i=1}^{k_i}\lambda_{ij_i}u_i^{j_i}\bigg)^xG^\alpha(\bar{u},t)\nonumber\\
		&=-\lambda^\alpha\sum_{x\geq0}\binom{\alpha }{x}\bigg(-\frac{1}{\lambda}\bigg)^xG^\alpha(\bar{u},t)\sum_{\substack{\sum_{i=1}^{q}x_i=x\\x_i\in\mathbb{N}_0}}x!\prod_{i=1}^{q}\frac{(\sum_{j_i=1}^{k_i}\lambda_{ij_i}u_i^{j_i})^{x_i}}{x_i!}\nonumber\\
		&=-\lambda^\alpha\sum_{x\geq0}\frac{\Gamma(\alpha +1)}{\Gamma(\alpha -x+1)}\bigg(-\frac{1}{\lambda}\bigg)^xG^\alpha(\bar{u},t)\sum_{\substack{\sum_{i=1}^{q}x_i=x\\x_i\in\mathbb{N}_0}}\prod_{i=1}^{q}\sum_{\substack{\sum_{j_i=1}^{k_i}x_{ij_i}=x_i\\
				x_{ij_i\in\mathbb{N}_0}}}\prod_{j_i=1}^{k_i}\frac{(\lambda_{ij_i}u_i^{j_i})^{x_{ij_i}}}{x_{ij_i}!}\nonumber\\
		&=-\lambda^\alpha\sum_{\substack{x_i\geq0\\i=1,2,\dots,q}}\frac{\Gamma(\alpha +1)}{\Gamma(\alpha -\sum_{i=1}^{q}x_i+1)}G^\alpha(\bar{u},t)\prod_{i=1}^{q}\sum_{\substack{\sum_{j_i=1}^{k_i}x_{ij_i}=x_i\\
				x_{ij_i\in\mathbb{N}_0}}}\prod_{j_i=1}^{k_i}\Big(-\frac{\lambda_{ij_i}u_i^{j_i}}{\lambda}\Big)^{x_{ij_i}}\frac{1}{x_{ij_i}!}\nonumber\\
		&=-\lambda^\alpha\sum_{\substack{m_i\geq 0\\ i=1,2,\dots,q}}\sum_{\substack{\Omega(k_i,m_i)\\i=1,2,\dots,q}}\sum_{\substack{n_i\geq 0\\ i=1,2,\dots,q}}\frac{\Gamma(\alpha +1)p^\alpha(\bar{n},t)}{\Gamma(\alpha -\sum_{i=1}^{q}\sum_{j_i=1}^{k_i}x_{ij_i}+1)}\Big(\prod_{i=1}^{q}u_i^{m_i+n_i}\Big)\prod_{i=1}^{q}\prod_{j_i=1}^{k_i}\frac{\big(-\lambda_{ij_i}/\lambda\big)^{x_{ij_i}}}{x_{ij_i}!}\nonumber\\
		&=-\lambda^\alpha\sum_{\substack{m_i\geq 0\\ i=1,2,\dots,q}}\sum_{\substack{\Omega(k_i,m_i)\\i=1,2,\dots,q}}\sum_{\substack{n_i\geq m_i\\ i=1,2,\dots,q}}\frac{\Gamma(\alpha +1)p^\alpha(\bar{n}-\bar{m},t)}{\Gamma(\alpha -\sum_{i=1}^{q}\sum_{j_i=1}^{k_i}x_{ij_i}+1)}\Big(\prod_{i=1}^{q}u_i^{n_i}\Big)\prod_{i=1}^{q}\prod_{j_i=1}^{k_i}\frac{\big(-\lambda_{ij_i}/\lambda\big)^{x_{ij_i}}}{x_{ij_i}!}\nonumber\\
		&=-\lambda^\alpha\sum_{\substack{n_i\geq 0\\ i=1,2,\dots,q}}\Big(\prod_{i=1}^{q}u_i^{n_i}\Big)\sum_{\substack{m_i\geq 0\\ i=1,2,\dots,q}}\sum_{\substack{\Omega(k_i,m_i)\\i=1,2,\dots,q}}\frac{\Gamma(\alpha +1)p^\alpha(\bar{n}-\bar{m},t)}{\Gamma(\alpha -\sum_{i=1}^{q}\sum_{j_i=1}^{k_i}x_{ij_i}+1)}\prod_{i=1}^{q}\prod_{j_i=1}^{k_i}\frac{\big(-\lambda_{ij_i}/\lambda\big)^{x_{ij_i}}}{x_{ij_i}!},\label{***}
\end{align}}
where in the last step we ignored the restriction $n_i\ge m_i$, $i=1,2,\dots,q$ because $p^\alpha(\bar{n}-\bar{m},t)=0$ when $\bar{n}\prec\bar{m}$.

Also, on taking the derivative on both sides of $G^\alpha(\bar{u},t)=\sum_{\substack{n_i\geq 0\\ 1\le i\le q}}(\prod_{i=1}^{q}u_i^{n_i})p^\alpha(\bar{n},t)$, we get
\begin{equation}\label{****}
	\frac{\partial}{\partial t}G^\alpha(\bar{u},t)=\sum_{\substack{n_i\geq 0\\i=1,2,\dots,q}}\Big(\prod_{i=1}^{q}u_i^{n_i}\Big)\frac{\mathrm{d}}{\mathrm{d}t}p^\alpha(\bar{n},t).
\end{equation}
Now, comparing the coefficients of $u_1^{n_1} u_2^{n_2}\dots u_q^{n_q}$ over the range $\bar{n}\ge\bar{0}$ on both sides of \eqref{***} and \eqref{****}, we get the required result.

\paragraph{A2}{\textit{Proof of Proposition \ref{prpalphatheta}}:}
On using \eqref{pgftempde}, we get
{\small\begin{align}\label{A1eq1}
\frac{\partial}{\partial t}G_{\bar{Z}}(\bar{u},t)&=\bigg(\theta^\alpha-\sum_{x=0}^{\infty}\binom{\alpha}{x}(\lambda+\theta)^{\alpha-x}(-1)^x\bigg(\sum_{i=1}^{q}\sum_{j_i=1}^{k_i}\lambda_{ij_i}u_i^{j_i}\bigg)^x\bigg)G_{\bar{Z}}(\bar{u},t)\nonumber\\
&=\bigg(\theta^\alpha-\sum_{x=0}^{\infty}\binom{\alpha}{x}(\lambda+\theta)^{\alpha-x}(-1)^x\sum_{\substack{\sum_{i=1}^{q}x_i=x\\x_i\in\mathbb{N}_0}}x!\prod_{i=1}^{q}\frac{(\sum_{j_i=1}^{k_i}\lambda_{ij_i}u_i^{j_i})^{x_i}}{x_i!}\bigg)G_{\bar{Z}}(\bar{u},t)\nonumber\\
&=\bigg(\theta^\alpha-\sum_{x=0}^{\infty}\binom{\alpha}{x}(\lambda+\theta)^{\alpha-x}(-1)^x\sum_{\substack{\sum_{i=1}^{q}x_i=x\\x_i\in\mathbb{N}_0}}x!\prod_{i=1}^{q}\sum_{\substack{\sum_{J_i=1}^{k_i}x_{ij_i}=x_i\\x_{ij_i}\in\mathbb{N}_0}}\prod_{j_i=1}^{k_i}\frac{(\lambda_{ij_i}u_i^{j_i})^{x_{ij_i}}}{x_{ij_i}!}\bigg)G_{\bar{Z}}(\bar{u},t)\nonumber\\
&=\bigg(\theta^\alpha-(\lambda+\theta)^{\alpha}\sum_{\substack{m_i\ge0\\i=1,2,\dots,q}}\sum_{\substack{\Omega(k_i,m_i)\\ i=1,2,\dots, q}}\binom{\alpha}{\sum_{i=1}^{q}\sum_{j_i=1}^{k_i}x_{ij_i}}\nonumber\Big(\frac{-1}{\lambda+\theta}\Big)^{\sum_{i=1}^{q}\sum_{j_i=1}^{k_i}x_{ij_i}}\\
&\hspace{7cm}\cdot\Big(\sum_{i=1}^{q}\sum_{j_i=1}^{k_i}x_{ij_i}\Big)!\prod_{i=1}^{q}u_i^{m_i}\prod_{j_i=1}^{k_i}\frac{\lambda_{ij_i}^{x_{ij_i}}}{x_{ij_i}!}\bigg)G_{\bar{Z}}(\bar{u},t)\nonumber\\
&=\sum_{\substack{n_i\ge 0\\ i=1,2,\dots,q}}\Big(\prod_{i=1}^{q}u_i^{n_i}\Big)\bigg(\theta^\alpha p_{\bar{Z}}(\bar{n},t)-(\lambda+\theta)^\alpha\Gamma(\alpha+1)\nonumber\\
&\hspace{2.8cm} \cdot \sum_{\substack{m_i\ge0\\i=1,2,\dots,q}}\sum_{\substack{\Omega(k_i,m_i)\\ i=1,2,\dots, q}}\frac{p_{\bar{Z}}(\bar{n}-\bar{m},t)}{\Gamma(\alpha+1-\sum_{i=1}^{q}\sum_{j_i=1}^{k_i}x_{ij_i})}\prod_{i=1}^{q}\prod_{j_i=1}^{k_i}\frac{(-\lambda_{ij_i}/(\lambda+\theta))^{x_{ij_i}}}{x_{ij_i}!}\bigg).
\end{align}}
Also, from \eqref{temprepr}, we have 
\begin{equation}\label{A1eq2}
\frac{\partial}{\partial t}G_{\bar{Z}}(\bar{u},t)=\sum_{\substack{n_i\ge 0\\ i=1,2,\dots,q}}\Big(\prod_{i=1}^{q}u_i^{n_i}\Big)p_{\bar{Z}}(\bar{n},t).
\end{equation}
Finally, the result follows on comparing the coefficient of $u_1^{n_1}u_2^{n_2}\dots u_q^{n_q}$ over the range $\bar{n}\ge\bar{0}$ on both sides of \eqref{A1eq1} and \eqref{A1eq2}.
\vspace{.5cm}
\paragraph{A3}{\textit{Proof of Theorem \ref{thmalphatheta}}:} On using \eqref{mnmn}, we get
\begin{align}\label{A2eq1}
G_{\bar{Z}}(\bar{u},t)&=e^{t\theta^\alpha}\sum_{r=0}^{\infty}\frac{(-\lambda^\alpha t)^r}{r!}\Big(1+\frac{\theta}{\lambda}-\frac{\sum_{i=1}^{q}\sum_{j_i=1}^{k_i}\lambda_{ij_i}u_i^{j_i}}{\lambda}\Big)^{\alpha r}\nonumber\\
&=e^{t\theta^\alpha}\sum_{r=0}^{\infty}\frac{(-t)^r}{r!}\sum_{x\ge0}\binom{\alpha r}{x}(\lambda+\theta)^{\alpha r-x}(-1)^x\sum_{\substack{\sum_{i=1}^{q}x_i=x\\x_i\in\mathbb{N}_0}}x!\prod_{i=1}^{q}\frac{(\sum_{j_i=1}^{k_i}\lambda_{ij_i}u_i^{j_i})^{x_i}}{x_i!}\nonumber\\
&=e^{t\theta^\alpha}\sum_{r=0}^{\infty}\frac{(- t)^r}{r!}\sum_{\substack{x_i\ge0\\i=1,2,\dots,q}}\binom{\alpha r}{\sum_{i=1}^{q}x_i}(\lambda+\theta)^{\alpha r}\Big(\frac{-1}{\lambda+\theta}\Big)^{\sum_{i=1}^{q}x_i}\Big(\sum_{i=1}^{q}x_i\Big)!\nonumber\\
&\hspace{7cm} \cdot\prod_{i=1}^{q}\sum_{\substack{\sum_{J_i=1}^{k_i}x_{ij_i}=x_i\\x_{ij_i}\in\mathbb{N}_0}}\prod_{j_i=1}^{k_i}\frac{(\lambda_{ij_i}u_i^{j_i})^{x_{ij_i}}}{x_{ij_i}!}\nonumber\\
&=e^{t\theta^\alpha}\sum_{r=0}^{\infty}\frac{(- t(\lambda+\theta)^{\alpha })^r}{r!}\sum_{\substack{n_i\ge0\\i=1,2,\dots,q}}\Big(\prod_{i=1}^{q}u_i^{n_i}\Big)\nonumber\\
&\hspace{2cm}\cdot\sum_{\substack{\Omega(k_i,n_i)\\ i=1,2,\dots,q}}\frac{\Gamma(\alpha r+1)}{\Gamma(\alpha r+1-\sum_{i=1}^{q}\sum_{J_i=1}^{k_i}x_{ij_i})}\prod_{i=1}^{q}\prod_{j_i=1}^{k_i}\frac{(-\lambda_{ij_i}/(\lambda+\theta))^{x_{ij_i}}}{x_{ij_i}!}.
\end{align}
Also, from \eqref{temprepr}, we can write
\begin{equation}\label{A2eq2}
G_{\bar{Z}}(\bar{u},t)=\sum_{\substack{n_i= 0\\ i=1,2,\dots,q}}\Big(\prod_{i=1}^{q}u_i^{n_i}\Big)p_{\bar{Z}}(\bar{n},t).
\end{equation}
On comparing the coefficient of $u_1^{n_1}u_2^{n_2}\dots u_q^{n_q}$ over the range $\bar{n}\ge\bar{0}$ on both sides of \eqref{A2eq1} and \eqref{A2eq2}, the result follows.
\vspace{.5cm}
\paragraph{A4}{\textit{Proof of Theorem \ref{thmappen}}:}
On using \eqref{pgfinvgau}, we get
{\small\begin{align}\label{CAA}
		G_{\bar{\mathcal{X}}}(\bar{u},t)&=e^{\delta \gamma t}\sum_{r=0}^{\infty}\frac{(-\sqrt{2\lambda}\delta t)^r}{r!}\bigg(1-\frac{1}{\lambda}\sum_{i=1}^{q}\sum_{j_i=1}^{k_i}\lambda_{ij_i}u_i^{j_i}\bigg)^{r/2}\nonumber\\
		&=e^{\delta \gamma t}\sum_{r=0}^{\infty}\frac{(-\sqrt{2\lambda}\delta t)^r}{r!}\sum_{x=0}^{\infty}\Big(\frac{r}{2}\Big)_{x}\Big(\frac{-1}{\lambda}\Big)^x\sum_{\substack{\sum_{i=1}^{q}x_i=x\\x_i\in\mathbb{N}_0}}\prod_{i=1}^{q}\frac{(\sum_{j_i=1}^{k_i}\lambda_{ij_i}u_i^{j_i})^{x_i}}{x_i!}\nonumber\\
		&=e^{\delta \gamma t}\sum_{r=0}^{\infty}\frac{(-\sqrt{2\lambda}\delta t)^r}{r!}\sum_{\substack{x_i\ge0\\ i=1,2,\dots,q}}\Big(\frac{r}{2}\Big)_{\sum_{i=1}^{q}x_i}\Big(\frac{-1}{\lambda}\Big)^{\sum_{i=1}^{q}x_i}\prod_{i=1}^{q}\sum_{\substack{\sum_{J_i=1}^{k_i}x_{ij_i}=x_i\\x_{ij_i}\in\mathbb{N}_0}}\prod_{j_i=1}^{k_i}\frac{(\lambda_{ij_i}u_{i}^{j_i})^{x_{ij_i}}}{x_{ij_i}!}\nonumber\\
		&=e^{\delta \gamma t}\sum_{r=0}^{\infty}\frac{(-\sqrt{2\lambda}\delta t)^r}{r!}\sum_{\substack{n_i\ge0\\i=1,2,\dots,q}}\sum_{\substack{\Omega(k_i,n_i)\\ i=1,2,\dots,q}}\Big(\frac{r}{2}\Big)_{\sum_{i=1}^{q}\sum_{j_i=1}^{k_i}x_{ij_i}}\Big(\frac{-1}{\lambda}\Big)^{\sum_{i=1}^{q}\sum_{j_i=1}^{k_i}x_{ij_i}}\prod_{i=1}^{q}\prod_{j_i=1}^{k_i}\frac{\lambda_{ij_i}^{x_{ij_i}}u_i^{n_i}}{x_{ij_i}!}.
\end{align}}
Also, on using \eqref{invgaudef}, we have
\begin{equation}\label{CAAC}
	G_{\bar{\mathcal{X}}}(\bar{u},t)=\sum_{\substack{n_i\ge0\\i=1,2,\dots,q}}\Big(\prod_{i=1}^{q}u_i^{n_i}\Big)p_{\bar{\mathcal{X}}}(\bar{n},t).
\end{equation}
Finally, on comparing the coefficient of $u_1^{n_1}u_2^{n_2}\dots u_q^{n_q}$ over the range $\bar{n}\ge \bar{0}$ in \eqref{CAA} and \eqref{CAAC}, the required result follows.

\end{document}